\documentclass[11pt,reqno]{amsproc}
\usepackage[T1]{fontenc}
\usepackage{color}
\usepackage[semicolon,square,authoryear]{natbib}
\numberwithin{equation}{section}
\usepackage[colorinlistoftodos, shadow]{todonotes}
\usepackage{MnSymbol}
\usepackage{stmaryrd}
\usepackage{fullpage}
\usepackage{graphicx}
\usepackage{subfigure}
\usepackage{psfrag}
\usepackage{float}
\usepackage{cite}
\usepackage{enumerate}
\usepackage{algorithm}
\usepackage{algorithmic}
\newtheorem{theorem}{Theorem}

\newtheorem{remark}[theorem]{Remark}

\newtheorem{lemma}[theorem]{Lemma}

\usepackage{morefloats}
\newlength{\drop}
\definecolor{amethyst}{rgb}{0.6, 0.4, 0.8}
\definecolor{burgundy}{rgb}{0.5, 0.0, 0.13}

\title{A monolithic multi-time-step computational framework 
  for first-order transient systems with disparate scales}

\author{\textbf{S.~Karimi and K.~B.~Nakshatrala}\\
Department of Civil and Environmental Engineering, 
University of Houston.}

\date{\today}

\begin{document}

%===================;
%  Title page of the paper  ;
%===================;

\begin{titlepage}
    \drop=0.1\textheight
    \centering
    \vspace*{\baselineskip}
    \rule{\textwidth}{1.6pt}\vspace*{-\baselineskip}\vspace*{2pt}
    \rule{\textwidth}{0.4pt}\\[\baselineskip]
    {\LARGE \textbf{\color{burgundy}A monolithic multi-time-step 
    computational \\[0.3\baselineskip]
    framework for first-order transient systems \\[0.3\baselineskip]
    with disparate scales}}\\[0.3\baselineskip]
    \rule{\textwidth}{0.4pt}\vspace*{-\baselineskip}\vspace{3.2pt}
    \rule{\textwidth}{1.6pt}\\[\baselineskip]
    \scshape
    An e-print of the paper is available on arXiv: http://arxiv.org/abs/1405.3230. \par
    
    \vspace*{2\baselineskip}
    Authored by \\[\baselineskip]
    
    {\Large S.~Karimi\par}
    {\itshape Graduate Student, University of Houston.}\\[\baselineskip]
    
    {\Large K.~B.~Nakshatrala\par}
    {\itshape Department of Civil \& Environmental Engineering \\
    University of Houston, Houston, Texas 77204--4003. \\ 
    \textbf{phone:} +1-713-743-4418, \textbf{e-mail:} knakshatrala@uh.edu \\
    \textbf{website:} http://www.cive.uh.edu/faculty/nakshatrala\par}
        \begin{figure}[h]
  \centering
  \includegraphics[scale = 0.32, clip]{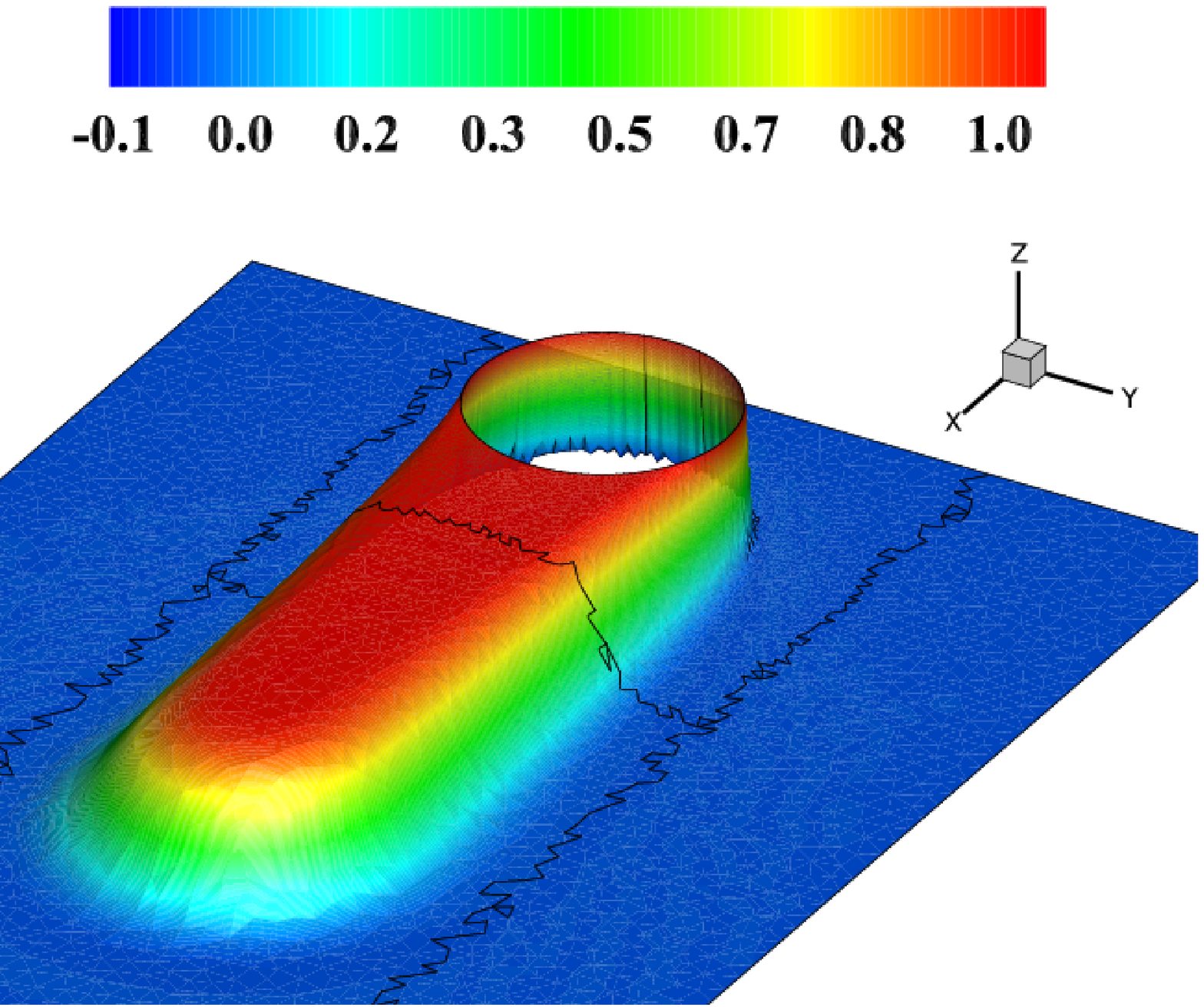}
\end{figure}
    \vfill
    {\scshape 2014} \\
    {\small Computational \& Applied Mechanics Laboratory} \par
  \end{titlepage}
  
%=========================;
%  Abstract and Keywords  ;
%=========================;
\begin{abstract}
Developing robust simulation tools for problems involving 
multiple mathematical scales has been a subject of great 
interest in computational mathematics and engineering. A 
desirable feature to have in a numerical formulation for 
multiscale transient problems is to be able to employ 
different time-steps (multi-time-step coupling), and 
different time integrators and different numerical 
formulations (mixed methods) in different regions 
of the computational domain.
To this end, we present two new monolithic multi-time-step mixed 
coupling methods for \emph{first-order transient systems}. 
We shall employ unsteady advection-diffusion-reaction equation with
linear decay as the model problem, which offers several unique 
challenges in terms of non-self-adjoint spatial operator and 
rich features in the solutions.
We shall employ the dual Schur domain decomposition technique 
to split the computational domain into an arbitrary number of 
subdomains. It will be shown that the governing equations of 
the decomposed problem, after spatial discretization, will be 
differential/algebraic equations. This is a crucial observation 
to obtain stable numerical results. 
  Two different methods of enforcing compatibility along the 
  subdomain interface will be used in the time discrete 
  setting. A systematic theoretical analysis (which includes 
  numerical stability, influence of perturbations, bounds on 
  drift along the subdomain interface) will be performed.  
  The first coupling method ensures that there is no drift 
  along the subdomain interface, but does not facilitate explicit/implicit 
  coupling. The second coupling method allows explicit/implicit 
  coupling with controlled (but non-zero) drift in the solution 
  along the subdomain interface. 
  Several canonical problems will be solved to numerically 
  verify the theoretical predictions, and to illustrate the 
  overall performance of the proposed coupling methods. 
  Finally, we shall illustrate the robustness of the 
  proposed coupling methods using a multi-time-step 
  transient simulation of a fast bimolecular 
  advective-diffusive-reactive system. 
\end{abstract}

\maketitle

%========================;
%  Include all sections  ;  
%========================;

%*********************************************;
%                                             ;
%  NAME                                       ;
%    S1_Monolithic_Intro.tex                  ;
%                                             ;
%  WRITTEN BY                                 ;
%    Saeed Karimi                             ;
%    Kalyana Babu Nakshatrala                 ;
%                                             ;
%*********************************************;
\section{INTRODUCTION AND MOTIVATION}
\label{Sec:First_Introduction}
Advection-diffusion-reaction equations can exhibit 
several mathematical (i.e., temporal and spatial) 
scales depending on the relative strengths of 
advection, diffusion and reaction processes, 
and on the strength of the volumetric source/sink. 
The presence of these mathematical scales is evident 
from the qualitative richness that the solutions of 
advection-diffusion-reaction equations exhibit. 
For example, it is well-known that solutions to 
advection-dominated problems typically exhibit 
steep gradients near the boundaries \citep{Gresho_Sani_v1}. 
Solutions to diffusion-dominated problems tend to be 
diffusive and smooth \citep{McOwen}, whereas reaction-dominated 
solutions typically exhibit sharp fronts and complex spatial 
patterns \citep{Walgraef}. These scales can be systematically 
characterized using the well-known non-dimensional numbers -- 
the P\'eclet number and the Damk\"ohler 
numbers \citep{Bird_Stewart_Lightfoot}. 
It needs to be emphasized that these equations, 
in general, are \emph{not} amenable to analytical 
solutions. Therefore, one has to rely on predictive 
numerical simulations for solving problems of any 
practical relevance. Due to the presence of disparate 
mathematical scales in these systems, it is highly 
desirable to have a stable computational framework 
that facilitates tailored numerical formulations in 
different regions of the computational domain. 

Several advances have been made in developing numerical 
formulations for advection-diffusion-reaction equations, 
especially in the area of stabilized formulations 
\citep{Codina_CMAME_2000_v188_p61,
Augustin_Caiazzo_Fiebach_Fuhrmann_John_Linke_Umla_CMAME_2011}, 
and in the area of discrete maximum principles 
\citep{Burman_Ern_CMAME_2002_v191_p3833}. 
However, the main research challenge that still remains is 
to develop numerical methodologies for these type of problems 
to adequately resolve different mathematical scales in time 
and in space. \emph{This 
paper precisely aims at addressing this issue by developing 
a stable multi-time-step coupling framework for first-order 
transient systems that allows different time-steps, different 
time integrators and different numerical formulations in 
different regions of a computational domain.}

%===================================================;
%  \subsection{A need for multi-time-step coupling  ;
%===================================================;
Most of the prior works on multi-time-step coupling 
methods have focused on the second-order transient 
systems arising in the area of structural dynamics (e.g., 
see the discussion in \citep{Karimi_Nakshatrala_JCP_2014}, 
and references therein). Some attempts regarding time 
integration of partitioned first-order systems can be 
found in \citep{Nakshatrala_Hjelmstad_Tortorelli_IJNME_2008_v75_p1385,
Nakshatrala_Prakash_Hjelmstad_JCP_2009_v228_p7957}. 
In \citep{Nakshatrala_Hjelmstad_Tortorelli_IJNME_2008_v75_p1385}, 
a \emph{staggered} multi-time-step coupling method is proposed. 
This method is considered as a staggered scheme as the Lagrange 
multipliers are calculated in 
an explicit fashion (i.e., based on the quantities 
known at prior time-levels). The stability and 
accuracy (especially, the control of drift along 
the subdomain interface) have been improved through 
the use of projection methods at appropriate time-levels. 
Since the method is a staggered scheme the obvious 
drawback is that the overall accuracy is first-order. 
However, it needs to be emphasized that the method 
proposed in 
\citep{Nakshatrala_Hjelmstad_Tortorelli_IJNME_2008_v75_p1385} 
has better accuracy and stability properties than 
the previously proposed staggered schemes 
(e.g., \citep{Piperno_Farhat_Larrouturou_CMAME_1995_v124_p79,
Piperno_IJNMF_1997_v25_p1207}).
In \citep{Nakshatrala_Prakash_Hjelmstad_JCP_2009_v228_p7957}, 
several monolithic schemes are discussed for first-order 
transient systems but the treatment is restricted to transient 
diffusion equations (i.e., self-adjoint spatial operators) and 
multi-time-stepping was not addressed. 
Motivated by the work of Akkasale 
\citep{Akkasale_MSThesis_TAMU_2011}; in which it has 
been systematically shown that many popular staggered 
schemes (e.g., \citep{Piperno_Farhat_Larrouturou_CMAME_1995_v124_p79,
Piperno_IJNMF_1997_v25_p1207}) suffer from numerical instabilities 
for both first- and second-order transient systems; we herein 
choose a monolithic approach to develop coupling methods 
that allow multi-time-steps.

%========================================================;
%  \subsection{Difference second-order vs. first-order}  ;
%========================================================;
Recently, a multi-time-step monolithic coupling 
method for linear elastodynamics, which is a 
second-order transient system, has been proposed 
in \citep{Karimi_Nakshatrala_JCP_2014}. However, 
developing a multi-time-step coupling method for 
first-order transient systems (e.g., unsteady 
advection-diffusion and advection-diffusion-reaction 
equations) will bring unique challenges. To name a few:
\begin{enumerate}[(i)]
\item As shown in \citep{Karimi_Nakshatrala_JCP_2014}, 
  coupling explicit and implicit time-stepping schemes 
  is \emph{always} possible in the case of second-order 
  transient systems. We will show later in this paper 
  that such coupling is \emph{not} always possible 
  for first-order transient systems, and can be 
  achieved only if an appropriate stabilized form 
  of the interface continuity constraint is employed. We will 
  also show that this explicit/implicit coupling 
  for first-order transient systems will come at 
  an expense of controlled drift.
\item Spatial operators in advection-diffusion-reaction 
  equations are not self-adjoint. Symmetry and positive 
  definiteness of the discretized operators should 
  be carefully examined to ensure the stability of 
  multi-time-step coupling methods.
  For second-order transient systems, the overall 
  stability of the coupling method can be achieved 
  provided the stability criterion in each subdomain 
  is satisfied (which depends on the choice of the 
  time-stepping scheme in the subdomain and the 
  choice of the subdomain time-step) 
  \citep{Karimi_Nakshatrala_JCP_2014}. 
  We will show in a subsequent section that ensuring 
  the stability of the time-stepping schemes in 
  subdomains alone will not guarantee the overall 
  stability of the coupling method. There is a 
  need to place additional restrictions on the 
  continuity constraints along the subdomain 
  interface.  
\item The governing equations of decomposed 
  first-order transient systems form a system 
  of differential/algebraic equations (DAEs) 
  in Hessenberg form with a differential index 
  2. On the other hand, the governing equations for 
  second-order transient systems form a system 
  of DAEs with differential index 3. For more 
  details on DAEs and associated terminology, 
  see the brief discussion provided in subsection 
  \ref{Subsec:First_DAEs} or consult 
  \citep{Hairer_Wanner_V2}.
\end{enumerate}

The current paper builds upon the ideas presented in 
\citep{Nakshatrala_Prakash_Hjelmstad_JCP_2009_v228_p7957,
Karimi_Nakshatrala_JCP_2014}. 
\emph{The central hypothesis on which the proposed 
multi-time-step coupling framework has been developed 
is two-fold}: 
(i) The governing equations before the domain 
decomposition form a system of ordinary differential 
equations (ODEs). On the other hand, the governing 
equations resulting from the decomposition of the 
domain form a system of differential/algebraic 
equations. It needs to be emphasized that many of 
the popular time-stepping schemes (which are 
developed for solving ODEs) are not appropriate 
for solving DAEs \citep{Gear_Petzold_SIAM_JNA_1984_v21, 
Petzold_PhysicaD_1992}. At least, the accuracy 
and the stability properties will be altered 
considerably. The title of an influential paper 
in the area of numerical solutions of DAEs by 
Petzold \citep{Petzold_SIAMJSciStatComp_1982_v3_p367} 
clearly conveys the aforementioned sentiment: 
``\emph{Differential/algebraic equations are not ODEs}.'' 
Therefore, we shall take a differential/algebraic
equations perspective in posing the governing 
equations of the decomposed problems, and apply 
time-stepping strategies that are appropriate to 
solve DAEs. 
(ii) Development and performance of multi-time-step 
coupling methods for first-order transient systems 
is different from that of second-order transient 
systems.

%===================================;
%  Main contributions of the paper  ;
%===================================;

The proposed monolithic multi-time-step coupling 
framework for first-order transient systems enjoys 
several attractive features, which will be illustrated 
in the subsequent sections by both theoretical analysis 
and numerical results. In the remainder of this paper, 
we shall closely follow the notation introduced for 
multi-time-step coupling in 
\citep{Karimi_Nakshatrala_JCP_2014}.

%********************************************;
%                                            ;
%  Name                                      ;
%    S2_First_Continuous.tex                 ;
%                                            ;
%  Written By                                ;
%    Kalyana Babu Nakshatrala                ;
%                                            ;
%********************************************;
\section{CONTINUOUS MODEL PROBLEM:~TRANSIENT ADVECTION-DIFFUSION-REACTION 
  EQUATION}
\label{Sec:First_Continuous}
We shall consider transient advection-diffusion-reaction 
equation as the continuous model problem. Our choice 
provides an ideal setting for developing multi-time-step 
coupling methods for first-order transient systems, as 
the governing equations pose several unique challenges.
\emph{First}, the relative strengths of advection, 
diffusion, reaction, and volumetric source introduce 
multiple temporal scales, which compel a need for a 
multi-time-step computational framework. 
\emph{Second}, the spatial operator is not self-adjoint, which 
adds to the complexity of obtaining stability proofs. It needs 
to be emphasized that the current efforts on multi-time-step 
coupling have focused on second-order transient systems, and 
the stability analyses have been restricted to the cases in 
which the coefficient (i.e., ``stiffness'') matrix is symmetric 
and positive definite \citep{Karimi_Nakshatrala_JCP_2014}. 
This will not be the case with respect to the advective-diffusive 
and advective-diffusive-reactive systems.
\emph{Third}, a numerical method to the chosen model problem 
can serve as a template for developing multi-time-step coupling 
methods for more complicated and important problems like 
transport-controlled bimolecular reactions, which exhibit 
complex spatial and temporal patterns. None of the prior 
works on multi-time-step coupling methods have undertaken such 
a comprehensive study, which this paper strives to achieve. 

Consider a chemical species that is transported 
by both advection and diffusion processes, and 
simultaneously undergoes a chemical reaction. 
Let $\Omega \subset \mathbb{R}^{nd}$ denote the 
spatial domain, where ``$nd$'' denotes the number 
of spatial dimensions. The boundary is denoted 
by $\partial \Omega$, which is assumed to be 
piece-wise smooth. The gradient and divergence 
operators with respect to $\mathbf{x} \in \Omega$ 
are, respectively, denoted by $\mathrm{grad}[\cdot]$ 
and $\mathrm{div}[\cdot]$.
The time is denoted by $t \in \mathcal{I} := (0,T]$, 
where $\mathcal{I}$ is the time interval of interest.
Let $c(\mathbf{x},t)$ denote the concentration of the 
chemical species. As usual, the boundary is divided 
into two parts: $\Gamma^{\mathrm{D}}$ and $\Gamma^{\mathrm{N}}$ 
such that $\Gamma^{\mathrm{D}} \cup \Gamma^{\mathrm{N}} 
= \partial \Omega$ and $\Gamma^{\mathrm{D}} \cap 
\Gamma^{\mathrm{N}} = \emptyset$. $\Gamma^{\mathrm{D}}$ 
is the part of the boundary on which concentration 
is prescribed (i.e., Dirichlet boundary condition), 
and $\Gamma^{\mathrm{N}}$ is that part of the boundary 
on which flux is prescribed (i.e., Neumann boundary 
condition). 
We shall denote the advection velocity vector field by 
$\mathbf{v} (\mathbf{x},t)$. The diffusivity tensor, 
which is a second-order tensor, is denoted by $\mathbf{D}
(\mathbf{x})$, and is assumed to be symmetric and uniformly 
elliptic \citep{Evans_PDE}. 
The initial boundary value problem for a transient 
advective-diffusive-reactive system can be written 
as follows:
%--------------------------------;
%  Equation: Continuous problem  ;
%--------------------------------;
\begin{subequations}
\label{Eqn:AD_PDE_system}
  \begin{alignat}{2}
  \label{Eqn:AD_PDE}
    &\frac{\partial \mathrm{c}}{\partial t} + \mathrm{div}\left[ \mathbf{v} \mathrm{c} 
    - \mathbf{D}\left( \mathbf{x}\right) \mathrm{grad}[\mathrm{c}] \right] 
  + \beta \mathrm{c} = f(\mathbf{x},t) 
    \quad &&\mathrm{in} \; \Omega \times \mathcal{I} \\
    &\mathrm{c}(\mathbf{x},t) = \mathrm{c}^{\mathrm{p}}(\mathbf{x},t) 
    \quad &&\mathrm{on} \; \Gamma^{\mathrm{D}} \times 
    \mathcal{I} \\
    \label{Eqn:First_Neumann_BC}
    -&\widehat{\mathbf{n}}(\mathbf{x}) \cdot 
    \mathbf{D}(\mathbf{x}) \mathrm{grad}[\mathrm{c}] 
    = q^{\mathrm{p}}(\mathbf{x},t) \quad &&\mathrm{on} 
    \; \Gamma^{\mathrm{N}} \times \mathcal{I} \\
      \label{Eqn:AD_IC}
    &\mathrm{c}(\mathbf{x},t = 0) = \mathrm{c}_{0}(\mathbf{x}) 
    \quad &&\mathrm{in} \; \Omega 
  \end{alignat}
\end{subequations}
where $\widehat{\mathbf{n}}(\mathbf{x})$ denotes the unit 
outward normal to the boundary, $\mathrm{c}_0(\mathbf{x})$ is the 
prescribed initial concentration, $\mathrm{c}^{\mathrm{p}}(\mathbf{x},
t)$ is the prescribed concentration on the boundary, 
$q^{\mathrm{p}}(\mathbf{x},t)$ is the prescribed diffusive 
flux on the boundary, $f(\mathbf{x},t)$ is the prescribed 
volumetric source/sink, and $\beta \geq 0$ is the 
coefficient of decay due to a chemical reaction. 

As mentioned earlier, the mathematical scales 
in advective-diffusive-reactive systems can 
be characterized using popular non-dimensional 
numbers. A non-dimensional measure to identify the 
relative dominance of advection is the P\'eclet 
number, which can be defined as follows:
%---------------------------;
%  Equation: Peclet number  ;
%---------------------------;
\begin{align}
  P_e(\mathbf{x},t) := \frac{L \|\mathbf{v}
    (\mathbf{x},t)\|}{D(\mathbf{x})}
\end{align}
where $L$ is the characteristic length, $D$ is the 
characteristic diffusivity, and $\|\cdot\|$ denotes 
the standard 2-norm. In the case of anisotropic 
diffusion tensor, $D(\mathbf{x})$ can be taken 
as the minimum eigenvalue of the diffusivity 
tensor at $\mathbf{x}$ (i.e., $D(\mathbf{x}) 
= \mathrm{min} \left\{\kappa \; | \; \mathrm{det} 
\left( \mathbf{D}(\mathbf{x}) - \kappa \mathbf{I}
\right)= 0 \right\}$). Clearly, the higher the 
P\'eclet number the greater will be the relative 
dominance of advection. 
A non-dimensional quantity to measure the relative 
dominance of the chemical reaction is the Damk\"ohler 
number, which takes the following form:
%------------------------------;
%  Equation: Damkohler number  ;
%------------------------------;
\begin{align}
  D_a := \frac{\beta L^2}{D(\mathbf{x})}
\end{align}
In the context of numerical solutions, the characteristic 
length is typically associated with an appropriate measure 
of the mesh size. A popular choice under the finite element 
method is $L = h_e/2$, where $h_e$ is the diameter of the 
circumscribed circle of the element and the factor $1/2$ is 
for convenience. This choice gives rise to what is commonly 
referred to as the element P\'eclet number (e.g., see 
\citep{Donea_Huerta}):
%-----------------------------------;
%  Equation: Element Peclet number  ;
%-----------------------------------;
\begin{align}
  P_{e}^{h} = \frac{h_{e} 
    \|\mathbf{v}(\mathbf{x},t)\|}{2 D}
\end{align}
which will be used in subsequent sections, especially, 
in defining stabilized weak formulations. We shall employ 
the semi-discrete methodology \citep{Zienkiewicz} based on 
the finite element method for spatial discretization and 
the trapezoidal family of time-stepping schemes for the 
temporal discretization. 

%=================================================;
%  Subsection: Trapezoidal time-stepping schemes  ;
%=================================================;
\subsection{Trapezoidal family of time-stepping schemes}
The time interval of interest is divided 
into $\mathcal{N}$ sub-intervals such that 
%-------------------------------;
%  Equation: Interval division  ;
%-------------------------------;
\begin{align}
  \mathcal{I} = \left(0,T\right] 
    = \bigcup_{n = 1}^{\mathcal{N}} 
    \left( t^{(n-1)},t^{(n)} \right]
\end{align}  
where $t^{(0)} = 0$ and $t^{(\mathcal{N})} = T$. To 
make the presentation simple, we shall assume that 
the sub-intervals are uniform. That is,
%-------------------------------;
%  Equation: Uniform time step  ;
%-------------------------------;
\begin{align}
  t^{(n)} - t^{(n-1)} = \Delta t \quad 
  \forall n = 1,\cdots, \mathcal{N}
\end{align}
where $\Delta t$ will be referred to as the time-step. 
However, it should be noted that the methods presented 
in this paper can be easily extended to variable 
time-steps. The primary variable 
(which, in our case, will be the concentration) and 
the corresponding time derivative at discrete time 
levels are denoted as follows: 
%---------------------------------------;
%  Equation: Primary variable and rate  ;
%---------------------------------------;
\begin{align}
  d^{(n)} \approx \mathrm{c}(t = t^{(n)}), 
  \quad v^{(n)} \approx \left. 
  \frac{\partial \mathrm{c}}{\partial t}
  \right\vert_{t = t^{(n)}}
\end{align}
The trapezoidal family of time-stepping schemes 
can be compactly written as follows: 
%--------------------------------;
%  Equation: Trapezoidal family  ;
%--------------------------------;
\begin{align}
  d^{(n+1)} = d^{(n)} + 
  \Delta t \left((1 - \vartheta)v^{(n)} 
  + \vartheta v^{(n+1)}\right)
\end{align}
where $\vartheta \in [0, 1]$ is a user-specified parameter. 
Some popular time-stepping schemes under the trapezoidal 
family include the explicit Euler ($\vartheta = 0$, which is also 
known as the forward Euler), the midpoint rule $(\vartheta = 1/2)$, 
and the implicit Euler ($\vartheta = 1$, which is also known as the 
backward Euler). The forward Euler method is an explicit scheme, 
and the midpoint and the backward Euler schemes are implicit. 
The stability and accuracy properties of these time-stepping 
schemes in the context of \emph{ordinary differential equations} 
are well-known (e.g., see \citep{Hairer_Wanner_V1}). 

%=================================;
%  Subsection: Weak formulations  ;
%=================================;
\subsection{Weak formulations}
We will now present several weak formulations for the 
initial boundary value problem given by equations 
\eqref{Eqn:AD_PDE}--\eqref{Eqn:AD_IC}, which will be 
used in the remainder of the paper. Since we address 
advection-dominated and reaction-dominated problems, 
we will present two popular stabilized weak formulations 
in addition to the Galerkin weak formulation. Let us introduce 
the following function spaces:
%-----------------------------;
%  Equation: Function spaces  ;
%-----------------------------;
\begin{subequations}
  \begin{align}
    \mathsf{C}_{t} &:= \left\{\mathrm{c}(\mathbf{x},\cdot) \in 
    H^{1}(\Omega) \; 
    \vert \; \mathrm{c}(\mathbf{x},t) = \mathrm{c}^{\mathrm{p}} 
    (\mathbf{x},t) \; \mathrm{on} \; \Gamma^{\mathrm{D}} 
    \right\} \\
    \mathsf{W} &:= \left\{\mathrm{w}(\mathbf{x}) \in 
    H^{1}(\Omega) \; \vert \; \mathrm{w}(\mathbf{x}) = 0 
    \; \mathrm{on} \; \Gamma^{\mathrm{D}} 
    \right\}
  \end{align}
\end{subequations}
where $H^{1}(\Omega)$ is a standard Sobolev space 
\citep{Brezzi_Fortin}. For convenience, we shall 
denote the standard $L_2$ inner-product over a 
set $K$ as follows:
%------------------------------;
%  Equation: L2 inner-product  ;
%------------------------------;
\begin{align}
  \left(a;b\right)_{K} \equiv \int_{K} a \cdot b 
  \; \mathrm{d} K 
\end{align}
The subscript $K$ will be dropped if the set is 
the entire spatial domain (i.e., $K = \Omega$). 

%=========================================;
%  Subsection: Galerkin weak formulation  ;
%=========================================;
\subsubsection{Galerkin weak formulation}
The Galerkin formulation for the initial boundary 
value problem \eqref{Eqn:AD_PDE}--\eqref{Eqn:AD_IC} 
can be written as follows: Find $\mathrm{c}(\mathbf{x},t) \in 
\mathsf{C}_t$ such that we have 
%----------------------------------;
%  Equation: Galerkin formulation  ;
%----------------------------------;
\begin{align}
  \label{Eqn:AD_Galerkin}
  \left(\mathrm{w};\partial \mathrm{c}/\partial t\right)
  + \left(\mathrm{w};\mathrm{div}[\mathbf{v} \mathrm{c}] \right)
  + \left(\mathrm{grad}[\mathrm{w}];\mathbf{D}(\mathbf{x}) 
  \mathrm{grad}[\mathrm{c}]\right) + (\mathrm{w};\beta \mathrm{c} - f) 
  = \left(\mathrm{w} ; q^{\mathrm{p}}\right)_{\mathrm{\Gamma}^{\mathrm{N}}} 
  \; \forall \mathrm{w}(\mathbf{x}) \in \mathsf{W}
\end{align}
It is well-known that the Galerkin formulation may exhibit numerical 
instabilities (e.g., spurious node-to-node oscillations) for non-self-adjoint 
spatial operators like the advective-diffusive and advective-diffusive-reactive 
systems. The reason can be attributed to the presence of boundary layers 
and interior layers in the solutions of these systems when advection 
is more dominant than the diffusion and reaction processes. 
Designing stable numerical formulations for advection-diffusion and 
advection-diffusion-reaction problems is still an 
active area of research (e.g., see 
\citep{Turner_Nakshatrala_Hjelmstad_IJNMF_v66_2011, 
Franca_Hauke_Masud_CMAME_2006_v195_p1560,
Gresho_Sani_v1}). This paper is not concerned 
with developing new stabilized formulations.

In order to avoid spurious oscillations and obtain 
accurate numerical solutions, it is sufficient to 
have the element P\'eclet number to be smaller than 
unity. To put it differently, if the element P\'eclet 
number is greater than unity, the computational mesh 
may not be adequate to resolve the steep gradients 
due to boundary layers and internal layers, which 
are typical in the solutions of advection dominated 
problems. One can always achieve smaller values 
for the element P\'eclet number by refining the 
computational mesh adequately. However, in some 
cases, the mesh has to be so fine that it may be 
computationally prohibitive to employ such a mesh. 
In order to alleviate the deficiencies of the Galerkin 
formulation for advection-dominated problems, many 
alternative methods have been proposed in the 
literature. For example, see 
\citep{Augustin_Caiazzo_Fiebach_Fuhrmann_John_Linke_Umla_CMAME_2011} 
for a short description and comparison of these 
methods. In this paper, we shall employ 
the SUPG formulation \citep{Brooks_Hughes_CMAME_1982_v32_p199} 
and the GLS formulation 
\citep{Hughes_Franca_Hulbert_CMAME_1989_v73_p173}, which 
are two popular approaches employed to enhance 
the stability of the Galerkin formulation. 
For completeness and future reference, we now briefly 
outline these two stabilized formulations. 
 
%=====================================;
%  Subsection: SUPG weak formulation  ;
%=====================================;
\subsubsection{Streamline Upwind/Petrov-Galerkin (SUPG) weak formulation}
The SUPG formulation reads as follows: Find 
$c(\mathbf{x},t) \in \mathsf{C}_t$ such that 
we have
%------------------------------;
%  Equation: SUPG formulation  ;
%------------------------------;
\begin{align}
  \label{Eqn:SUPG_Weak_Form}
  (\mathrm{w};\partial \mathrm{c} / \partial t) &+ (\mathrm{w};\mathrm{div}[\mathbf{v}\mathrm{c}]) 
  + (\mathrm{grad}[\mathrm{w}];\mathbf{D}(\mathbf{x}) \mathrm{grad}[\mathrm{c}]) 
  + (\mathrm{w};\beta \mathrm{c}) \nonumber \\ 
  &+ \sum_{e = 1}^{Nele} 
  \left(\tau_{\mathrm{SUPG}} \mathbf{v} \cdot 
  \mathrm{grad}[\mathrm{w}]; \partial \mathrm{c} / \partial t 
  + \mathrm{div} \left[\mathbf{v}\mathrm{c} - \mathbf{D}
    (\mathbf{x}) \mathrm{grad}[\mathrm{c}]\right] + \beta 
  \mathrm{c} - f\right)_{\Omega_e} \nonumber \\
  &= (\mathrm{w};f) + \left(\mathrm{w};q^{\mathrm{p}}\right)_{\mathrm{\Gamma}^{\mathrm{N}}} 
  \quad \forall \mathrm{w}(\mathbf{x}) \in \mathsf{W}
\end{align}
where $Nele$ is the number of elements, and 
$\tau_{\mathrm{SUPG}}$ is the stabilization 
parameter under the SUPG formulation. We shall 
use the stabilization parameter proposed in 
\citep{John_Knobloch_CMAME_2007_v196_p2197}: 
%------------------------------------------;
%  Equation: SUPG stabilization parameter  ;
%------------------------------------------;
\begin{align}
  \label{Eqn:SUPG_Parameter}
  \tau_{\mathrm{SUPG}} = \frac{h_e}{2 \| \mathbf{v} \|} \xi_0 
  \left(P_{e}^{h}\right), \quad \xi_{0} \left(\chi \right) 
  = \coth \left( \chi\right) - \frac{1}{\chi}
\end{align}
where $h_e$ is the element length, and $\xi_0$ 
is known as the upwind function. Recall that 
$P_e^h$ is the local (element) P\'eclet number. 

%=============================================================;
%  Subsection: Galerkin/least-squares (GLS) weak formulation  ;
%=============================================================;
\subsubsection{Galerkin/least-squares (GLS) weak formulation}
The GLS formulation reads as follows: Find 
$c(\mathbf{x},t) \in \mathsf{C}_t$ such 
that we have
%------------------------------;
%  Equation: SUPG formulation  ;
%------------------------------;
\begin{align}
  \label{Eqn:GLS_weak_form}
  (\mathrm{w};\partial \mathrm{c} / \partial t) &+ (\mathrm{w};\mathrm{div}[\mathbf{v}\mathrm{c}]) 
  + (\mathrm{grad}[\mathrm{w}];\mathbf{D}(\mathbf{x}) 
  \mathrm{grad}[\mathrm{c}]) + (\mathrm{w};\beta \mathrm{c}) \nonumber \\ 
  &+ \sum_{e = 1}^{Nele} 
  \left( \mathrm{w} / \Delta t + \mathrm{div} [\mathbf{v}\mathrm{w} - \mathbf{D}
    (\mathbf{x}) \mathrm{grad}[\mathrm{w}]] + \beta \mathrm{w}; \tau_{\mathrm{GLS}} 
  \left( \partial \mathrm{c} / \partial t + \mathrm{div} \left[ \mathbf{v}\mathrm{c} 
    - \mathbf{D}(\mathbf{x}) \mathrm{grad}[\mathrm{c}]\right] + \beta \mathrm{c} - f
  \right) \right)_{\Omega_e} \nonumber \\
  &= (\mathrm{w};f) + \left(\mathrm{w} ; q^{\mathrm{p}}\right)_{\mathrm{\Gamma}^{\mathrm{N}}} 
  \quad \forall \mathrm{w}(\mathbf{x}) \in \mathsf{W}
\end{align}
where $\tau_{\mathrm{GLS}}$ is the stabilization parameter under 
the GLS formulation, and $\Delta t$ is the time-step. In this 
paper, we shall take $\tau_{\mathrm{GLS}} = \tau_{\mathrm{SUPG}}$, 
which is a common practice. It should be emphasized that an 
optimal choice of stabilization parameter for stabilized 
formulations in two- and three-dimensions is still an active 
area of research (e.g., see 
\citep{Augustin_Caiazzo_Fiebach_Fuhrmann_John_Linke_Umla_CMAME_2011}).

%********************************************;
%                                            ;
%  Name                                      ;
%    S3_First_Proposed.tex                   ;
%                                            ;
%  Written By                                ;
%    Kalyana Babu Nakshatrala                ;
%                                            ;
%********************************************;
\section{PROPOSED MULTI-TIME-STEP COMPUTATIONAL FRAMEWORK}
\label{Sec:S3_First_Discrete}
The proposed multi-time-step computational framework 
is built upon the semi-discrete methodology 
\citep{Zienkiewicz} and the dual Schur domain decomposition 
method \citep{Toselli_DD}. The semi-discrete methodology 
converts the partial differential equations into a 
system of ordinary differential equations. For spatial 
discretization of the problem at hand, one can use either 
the Galerkin formulation or a stabilized formulation, which 
could depend on the relative strengths of transport 
processes and the decay coefficient due to chemical 
reactions. The dual Schur domain decomposition is 
an elegant way to handle decomposition of the 
computational domain into subdomains through 
Lagrange multipliers. 

%==============================================================;
%  Subsection: Governing equations for the decomposed problem  ;
%==============================================================;
\subsection{Domain decomposition and the resulting equations}
In order to facilitate multi-time-step coupling, 
we decompose the computational domain into 
$\mathcal{S}$ non-overlapping subdomains 
such that 
%----------------------------------;
%  Equation: Domain decomposition  ;
%----------------------------------;
\begin{align}
  \overline{\Omega} = \bigcup_{i = 1}^{\mathcal{S}} 
    \overline{\Omega}_i \quad \mathrm{and} \quad  
    \Omega_i \cap \Omega_j = \emptyset \; 
    \mathrm{for} \; i \neq j
\end{align}
where a superposed bar denotes the set closure. 
The meshes in all subdomains are assumed to be 
conforming along the subdomain interface, see 
Figure \ref{Fig:dual_Schur}. 
We shall use signed Boolean matrices to write the 
compatibility constraints along the subdomain interface, 
as they provide a systematic way to write the interface 
constraints as a system of linearly independent equations. 
Moreover, the mathematical structure of the resulting 
equations is suitable for a mathematical analysis. The 
entries of a signed Boolean matrix are either -1, 0, or 1, 
and each row has at most one non-zero entry. However, 
it needs to be emphasized that a signed Boolean matrix is never 
constructed explicitly in a computer implementation, as it 
is computationally not efficient to store such a matrix. 
It should also be noted that signed Boolean matrices can 
handle constraints arising from cross-points, which are 
the points on the subdomain interface that are connected 
to more than two subdomains. For more details on signed 
Boolean matrices see 
\citep{Nakshatrala_Hjelmstad_Tortorelli_IJNME_2008_v75_p1385}.

In a time-continuous setting, the governing 
equations after spatial discretization can 
be written as follows:
%-------------------------------------------;
%  Equation: Continuous GE for first-order  ;
%-------------------------------------------;
\begin{subequations}
  \label{Eqn:DAE_Linear_Model}
  \begin{align}
  \label{Eqn:DAE_Linear_Model_Eq1}
    &\boldsymbol{M}_i \dot{\boldsymbol{c}}_i (t)
    + \boldsymbol{K}_i \boldsymbol{c}_i(t) = 
    \boldsymbol{f}_i (t) + \boldsymbol{C}_i^{\mathrm{T}} 
    \boldsymbol{\lambda} \left( t \right) \quad i = 
    1,\cdots, \mathcal{S} \\
  \label{Eqn:DAE_Linear_Model_Eq2}
    &\sum_{i = 1}^{\mathcal{S}} \boldsymbol{C}_i 
    \boldsymbol{c}_i \left( t \right) = 
    \boldsymbol{0}
  \end{align}
\end{subequations}
where a superposed dot denotes a derivative with respect 
to time, the subscript $i$ denotes the subdomain number, 
the nodal concentration vector of the $i$-th subdomain is 
denoted by $\boldsymbol{c}_i$, the capacity matrix of the 
$i$-th subdomain is denoted by $\boldsymbol{M}_i$, the 
transport matrix of the $i$-th subdomain is denoted by 
$\boldsymbol{K}_i$, $\boldsymbol{f}_i(t)$ is the forcing 
vector of the $i$-th subdomain, $\boldsymbol{\lambda}$ 
denotes the vector of Lagrange multipliers, and 
$\boldsymbol{C}_i$ denotes the signed Boolean matrix 
for the $i$-th subdomain.  
Let the number of degrees-of-freedom in the $i$-th 
subdomain be denoted by $N_i$, and the number 
of degrees-of-freedom on the subdomain interface be 
denoted by $N_{\lambda}$. The size of $\boldsymbol{c}_i$ 
is $N_i \times 1$, and both the capacity and transport 
matrices of the $i$-th subdomain will be of the size 
$N_i \times N_i$. The size of $\boldsymbol{\lambda}$ 
will be $N_{\lambda} \times 1$, and the size of the 
signed Boolean matrix $\boldsymbol{C}_i$ will be 
$N_{\lambda} \times N_i$. 

It is imperative to note that the governing equations 
\eqref{Eqn:DAE_Linear_Model_Eq1}--\eqref{Eqn:DAE_Linear_Model_Eq2}, 
which arise from domain decomposition, form a system of 
differential/algebraic equations (DAEs). For completeness 
and future reference we now present the necessary details 
about differential/algebraic equations.

%================================================;
%  Subsection: Differential/algebraic equations  ;
%================================================;
\subsubsection{Differential/algebraic equations}
\label{Subsec:First_DAEs}
A differential/algebraic equation is an equation involving 
a set of independent variables, an unknown function of the 
independent variables, and derivatives of the functions 
with respect to the independent variables. Clearly, 
ordinary differential equations, and algebraic equations 
form subclasses of differential/algebraic equations. 
In this paper, we are concerned with first-order 
differential/algebraic equations. Mathematically, 
a DAE in first-order form takes the following general 
form:
%---------------------------------;
%  Equation: DAE in general form  ;
%---------------------------------;
\begin{align}
  \label{Eqn:DAE_general_form}
  \boldsymbol{w} \left(\dot{\boldsymbol{x}}(t),
  \boldsymbol{x}\left( t \right), t\right) = 
  \boldsymbol{0} \quad t \in \mathcal{I}
\end{align}
where $t$ is the independent variable, and $\boldsymbol{x}(t)$ 
is the unknown function. 
It is well-known that solving a system of 
differential/algebraic equations numerically 
can be more difficult than solving a system 
of ordinary differential equations \citep{Hairer_Wanner_V2,
Petzold_SIAMJSciStatComp_1982_v3_p367}. 
A notion which is popularly employed to 
measure the difficulty of obtaining numerical solutions 
to a particular DAE is the \emph{differential index}.
The differential index of a DAE is the number 
of times one has to take derivatives of equation 
\eqref{Eqn:DAE_general_form} in order to be able 
to derive an ODE by mere algebraic manipulations. 
It is obvious that a system of ODEs will have 
differential index of zero. A special form of 
DAEs which is of interest to us in this paper 
is the Hessenberg index-2 DAE. It has the 
following mathematical form:
%-------------------------------;
%  Equation: Semi-explicit DAE  ;
%-------------------------------;
\begin{subequations}
  \label{Eqn:Semi_Implicit_DAE}
  \begin{align}
    &\dot{\boldsymbol{x}} = \boldsymbol{p} 
    (\boldsymbol{x},\boldsymbol{y},t) \\
    &\boldsymbol{0} = \boldsymbol{q}(\boldsymbol{x})
  \end{align}
\end{subequations}
which consists of a system of ordinary differential equations 
along with a set of algebraic equations (i.e., constraints). 
This paper concerns with differential/algebraic equations 
of differential index two or lower. Many of the constrained 
mechanical systems can be modeled using DAEs (e.g., see 
\citep{Geradin_Cardona}). In the case of coupling 
algorithms, the compatibility of subdomains along 
the interfaces will appear as an algebraic constraint 
to the ODEs obtained from a finite element discretization.
It is not possible to solve differential/algebraic equations 
analytically unless in some very special cases. Hence, 
one has to resort to numerical solutions. In this paper, 
we shall restrict to time-stepping schemes from the 
trapezoidal family. However, the corresponding 
properties when applied to differential/algebraic 
equations can be different. For a detailed discussion 
on this topic see \citep{Hairer_Wanner_V2}.

%==========================;
%  Subsection: Time discretization   ;
%==========================;
\subsection{Time discretization}
\label{Sec:S3_Proposed_Coupling}
We now construct two multi-time-step coupling methods 
that can handle multiple subdomains, and can allow the 
use of different time-steps, different time-integrators 
and/or different numerical formulation in different 
subdomains. To this end, the time interval of interest 
is divided into non-overlapping intervals whose end 
points will be referred to as \emph{system time-levels}. 
The algebraic compatibility constraints will be 
enforced at the system time-levels. For convenience, 
we shall assume that the system time-levels are uniform. 
The $n$-th system time-level will be denoted by 
$t^{(n)}$ and can be written as follows: 
%-------------------------------;
%  Equation: System time-level  ;
%-------------------------------;
\begin{align}
  t^{(n)} = n \Delta t \quad n = 0, 1, \cdots, \mathcal{N}
\end{align}
where $\Delta t$ is called the \emph{system time-step}.
The numerical time-integration of each subdomain will
advance by the \emph{subdomain time-step}. The subdomain
time-step of the $i$-th subdomain will be denoted by
$\Delta t_i$. Note that $\Delta t \geq \Delta t_i \;
\forall i$. Furthermore, we shall assume that the
ratio between the system and subdomain time-step is
a natural number, and is denoted by $\eta_i$. That is, 
%---------------------------------;
%  Equation: Definition of eta_i  ;
%---------------------------------;
\begin{align}
  \eta_i = \frac{\Delta t}{\Delta t_i}
\end{align}
Figure \ref{Fig:Monolithic_multi_time_step_notation} 
presents a pictorial description of the system and 
subdomain time-steps. In the rest of the paper, we 
will use the following notation to show the value of 
a variable at a time-level:
\begin{align}
  &x^{\left( n + \frac{j}{\eta_i}\right)} = 
  x \left( t^{(n)} + j \Delta t_i\right) \\
  &t^{\left( n + \frac{j}{\eta_i}\right)} = 
  t^{(n)} + j \Delta t_i
\end{align}
Note that because of the enforcement of compatibility 
constraint at system time-levels only, the Lagrange
multipliers can only be calculated at system time-levels. 
We shall linearly interpolate the Lagrange multipliers 
within system time-levels. That is, 
\begin{align}
  \label{Eqn:Lagrange_Multiplier_Approx}
  \boldsymbol{\lambda}^{\left( n + \frac{j+1}{\eta_i}\right)}
  = \left( 1 - \frac{j+1}{\eta_i}\right) \boldsymbol{\lambda}^{(n)}
  + \left( \frac{j+1}{\eta_i}\right) \boldsymbol{\lambda}^{(n+1)}
\end{align}
As discussed earlier, coupling explicit and implicit 
time-stepping schemes is not straightforward in the 
case of first-order transient systems as compared with 
second-order systems. The proposed computational framework 
will employ different compatibility constraints in order to 
enforce continuity and to make an explicit/implicit coupling 
possible. 

%===============================================================;
%  Subsection: Mathematical statements of the proposed methods  ;
%===============================================================;
\subsection{Mathematical statements of the proposed coupling methods}
The compatibility constraints along the subdomain interface 
will be enforced at system time-levels. Mathematically, the time 
discretization of compatibility constraints reads as follows: 
  \begin{alignat}{2}
   % \label{Eqn:Method_d_continuity}
    &\sum_{i = 1}^{\mathcal{S}} \boldsymbol{C}_i \boldsymbol{d}_i^{(n+1)} 
    = \boldsymbol{0} \quad \forall n && \quad \mbox{$d$-continuity method} \\
  %
%    \label{Eqn:Method_Baumgarte}	
    &\sum_{i = 1}^{\mathcal{S}} \boldsymbol{C}_i 
    \left( \boldsymbol{v}_i^{(n+1)}+ \frac{\alpha}{\Delta t}
    \boldsymbol{d}_i^{(n+1)}
    \right) = \boldsymbol{0} \quad \forall n &&\quad 
    \mbox{Baumgarte stabilization}
  \end{alignat}
where $\alpha > 0$ is the Baumgarte stabilization parameter. 
The proposed coupling method based on $d$-continuity
will read as follows:~Find $\left( \boldsymbol{v}^{\left( n + (j+1)/\eta_i\right)}_i, \boldsymbol{d}^{\left( n + (j+1)/\eta_i\right)}_i, 
\boldsymbol{\lambda}^{(n+1)} \right)$ for $n = 1 , ... , \mathcal{N}$; 
$j = 0 , ... , \eta_i - 1$; and $i = 1 , ... , \mathcal{S}$ such that 
we have
\begin{subequations}
\label{Eqn:Coupling_d_continuity}
\begin{align}
	&\boldsymbol{M}_i \boldsymbol{v}_i^{\left( n + \frac{j+1}{\eta_i}
	\right)} + \boldsymbol{K}_i \boldsymbol{d}_i^{\left( n + \frac{j+1}{\eta_i}\right)}= \boldsymbol{f}_i^{\left( n + \frac{j+1}{\eta_i}\right)}+ \boldsymbol{C}_i^{\mathrm{T}} \boldsymbol{\lambda}^{\left( n + \frac{j+1}{\eta_i} \right)} \\
	&\boldsymbol{d}_i^{\left( n + \frac{j+1}{\eta_i}\right)} = 
	\boldsymbol{d}_i^{\left( n + \frac{j}{\eta_i}\right)} + 
	\Delta t_i \left( \left( 1 - \vartheta_i\right)\boldsymbol{v}_i^{\left( n + \frac{j}{\eta_i}\right)} + \vartheta_i \boldsymbol{v}_i^{\left( n + \frac{j+1}{\eta_i}\right)}\right) \\
	&\boldsymbol{\lambda}^{\left( n + \frac{j+1}{\eta_i}\right)} = 
	\left( 1 - \frac{j+1}{\eta_i}\right) \boldsymbol{\lambda}^{(n)} +
	\left( \frac{j+1}{\eta_i}\right) \boldsymbol{\lambda}^{(n+1)} \\
	&\sum_{i = 1}^{S} \boldsymbol{C}_i \boldsymbol{d}_i^{(n + 1)} = \boldsymbol{0}
\end{align}
\end{subequations}
The proposed coupling method based on the Baumgarte 
stabilization will read as follows:~Find $\left( \boldsymbol{v}^{\left( n + (j+1)/\eta_i\right)}_i, \boldsymbol{d}^{\left( n + (j+1)/\eta_i\right)}_i, 
\boldsymbol{\lambda}^{(n+1)} \right)$ for $n = 1 , ... , \mathcal{N}$; 
$j = 0 , ... , \eta_i - 1$; and $i = 1 , ... , \mathcal{S}$ such that 
we have
\begin{subequations}
\label{Eqn:Coupling_Baumgarte}
\begin{align}
	&\boldsymbol{M}_i \boldsymbol{v}_i^{\left( n + \frac{j+1}{\eta_i}
	\right)} + \boldsymbol{K}_i \boldsymbol{d}_i^{\left( n + \frac{j+1}{\eta_i}\right)}= \boldsymbol{f}_i^{\left( n + \frac{j+1}{\eta_i}\right)}+ \boldsymbol{C}_i^{\mathrm{T}} \boldsymbol{\lambda}^{\left( n + \frac{j+1}{\eta_i} \right)} \\
	&\boldsymbol{d}_i^{\left( n + \frac{j+1}{\eta_i}\right)} = 
	\boldsymbol{d}_i^{\left( n + \frac{j}{\eta_i}\right)} + 
	\Delta t_i \left( \left( 1 - \vartheta_i\right)\boldsymbol{v}_i^{\left( n + \frac{j}{\eta_i}\right)} + \vartheta_i \boldsymbol{v}_i^{\left( n + \frac{j+1}{\eta_i}\right)}\right) \\
	&\boldsymbol{\lambda}^{\left( n + \frac{j+1}{\eta_i}\right)} = 
	\left( 1 - \frac{j+1}{\eta_i}\right) \boldsymbol{\lambda}^{(n)} +
	\left( \frac{j+1}{\eta_i}\right) \boldsymbol{\lambda}^{(n+1)} \\
	&\sum_{i = 1}^{S} \boldsymbol{C}_i \left( \boldsymbol{v}_i^{(n + 1)} + 
	\frac{\alpha}{\Delta t} \boldsymbol{d}_i^{(n + 1)}\right) = \boldsymbol{0}
\end{align}
\end{subequations}

Before we perform a systematic theoretical analysis of the 
proposed multi-time-step coupling methods in the next section, 
it needs to be mentioned that the quantity $\partial c/\partial t$ 
in the stabilization terms under the SUPG and GLS stabilized 
formulations (see equations 
\eqref{Eqn:SUPG_Parameter} and \eqref{Eqn:GLS_weak_form}) 
will be evaluated at the weighted time-level $n + (j+\vartheta_i)/\eta_i$ 
for the $i$-th subdomain. This implies that this quantity 
in the stabilization terms for the $i$-th subdomain needs to be 
calculated as follows:
\begin{align}
  \left. \frac{\partial c}{\partial t} \right|_{n+(j+\vartheta_i)/\eta_i} 
  \approx (1 - \vartheta_i) \boldsymbol{v}^{(n+j/\eta_i)} + \vartheta_i 
  \boldsymbol{v}^{(n+(j+1)/\eta_i)} = 
  \frac{\boldsymbol{d}^{(n + (j+1)/\eta_i)}-\boldsymbol{d}^{(n+j/\eta_i)}}{\Delta t_i}
\end{align}
This form of discretization will be crucial in 
proving the stability of the proposed coupling 
methods. 
More details on the implementation of the proposed 
coupling methods can be be found in Appendix.

%************************************************;
%                                                ;
%  NAME                                          ;
%    S4_First_Theoretical.tex                    ;
%                                                ;
%  WRITTEN BY                                    ;
%    Saeed Karimi                                ;
%    Kalyana Babu Nakshatrala                    ;
%                                                ;
%************************************************;
\section{A THEORETICAL STUDY ON THE PROPOSED METHODS}
\label{Sec:S4_First_Theoretical}

%========================;
%  Subsection: Notation  ;
%========================;
\subsection{Notation} 
The jump and average operators over the $i$-th 
subdomain time-step are, respectively, defined 
as follows:
%--------------------------------------------------;
%  Equation: Subdomain average and jump operators  ;
%--------------------------------------------------;
\begin{subequations}
  \begin{align}
    \left[x^{\left( n + \frac{j}{\eta_i} \right)}\right]_i 
    &:= x^{\left( n + \frac{j + 1}{\eta_i}\right)} - 
    x^{\left( n + \frac{j}{\eta_i}\right)} \\
    \left\langle x^{\left(n + \frac{j}{\eta_i} \right)}\right\rangle_i
    &:= \frac{1}{2} \left( x^{\left( n + \frac{j+1}{\eta_i}\right)} + 
    x^{\left( n + \frac{j}{\eta_i} \right)}\right)
  \end{align}
\end{subequations}
One can similarly define the jump and average 
operators over a system time-step as follows:
%-----------------------------------------------;
%  Equation: System average and jump operators  ;
%-----------------------------------------------;
\begin{subequations}
  \label{Eqn:Jum_Average_system}
  \begin{align}
    &\left\llbracket x^{(n)}\right\rrbracket :=
    x^{(n+1)} - x^{(n)} = \sum_{j = 0}^{\eta_i - 1} 
    \left[ x^{\left( n + \frac{j}{\eta_i}\right)}\right]_i \\
    &\left\llangle x^{(n)}\right\rrangle :=
    \frac{1}{2} \left( x^{(n+1)} + x^{(n)}\right) 
  \end{align}
\end{subequations}
Let $\boldsymbol{S}$ be a symmetric matrix, 
then we have the following identity: 
%----------------------------;
%  Equation: Two identities  ;
%----------------------------;
\begin{align}
  \left\llangle \boldsymbol{x} \right\rrangle^{\mathrm{T}} 
  \boldsymbol{S} \llbracket \boldsymbol{x} \rrbracket 
  = \frac{1}{2} \llbracket\boldsymbol{x}^{\mathrm{T}} 
  \boldsymbol{S} \boldsymbol{x}\rrbracket
\end{align}
The trapezoidal family of time-stepping 
schemes applied over a subdomain time-step 
can be compactly written as follows: 
%--------------------------------;
%  Equation: Trapezoidal family  ;
%--------------------------------;
\begin{align}
  \label{Eqn:Trapezoidal_jump_d}
  \left[ \boldsymbol{d}_i^{\left( n + \frac{j}{\eta_i}\right)}\right]_i
  = \Delta t_i \left( \left\langle \boldsymbol{v}_i^{\left( n + \frac{j}{\eta_i}\right)}\right\rangle_i + \left( \vartheta_i - \frac{1}{2}\right) 
  \left[ \boldsymbol{v}_i^{\left( n + \frac{j}{\eta_i}\right)}\right]_i\right)
\end{align}

%==================================;
%  Subsection: Stability analysis  ;			
%==================================;
\subsection{Stability analysis}
Consistency of the proposed coupling methods is 
trivial by construction. Hence, for convergence, 
it is necessary and sufficient to show that the 
proposed coupling methods are stable. We now show 
that both the proposed coupling methods are indeed 
stable using the energy method \citep{Wood}. 
For numerical stability analysis, it is common 
to assume that supply function is zero. Therefore, 
we take $\boldsymbol{f}_i(t) = \boldsymbol{0}$ in 
all the subdomains.
%=========================================;
%  On the nature of the transport matrix  ; 
%=========================================;
Before we can provide stability proofs for the proposed coupling 
methods, we need to present an important property that the 
transport matrices enjoy under the three weak formulations 
that were outlined in the previous section. This property 
will play a crucial role in the stability analysis. We 
provide a proof for the Galerkin weak formulation.

%------------------------------------------;
%  Lemma: Positive Definiteness of sym[K]  ;
%------------------------------------------;
\begin{lemma}
  \label{Thm:Nonsymmetric_K}
  Consider the Galerkin weak formulation given by 
  equation \eqref{Eqn:AD_Galerkin}. If the advection 
  velocity satisfies $\mathrm{div}[\mathbf{v}] 
  \geq 0$, and the diffusivity tensor $\mathbf{D}
  (\mathbf{x})$ is symmetric and positive definite, 
  then the symmetric part of the transport matrix 
  resulting from the finite element discretization 
  will be positive semi-definite.	
\end{lemma}
%----------------------;
%  Proof of the lemma  ;
%----------------------;
\begin{proof}
  Let us denote the spatial operator of the 
  advective-diffusive system as follows: 
  %------------------------;
  %  Equation: Operator L  ;
  %------------------------;
  \begin{align}
    \mathcal{L}[\mathrm{c}] := \mathrm{div}[\mathbf{v} \mathrm{c}] 
    - \mathrm{div}[\mathbf{D}(\mathbf{x}) 
      \mathrm{grad}[\mathrm{c}]]
  \end{align}
  It is easy to show that the adjoint of the 
  spatial operator takes the following form:
  %------------------------------;
  %  Equation: Adjoint operator  ;
  %------------------------------;
  \begin{align}
    \mathcal{L}^{*}[\mathrm{c}] = - \mathbf{v} \cdot \mathrm{grad}[\mathrm{c}] 
    - \mathrm{div}\left[\mathbf{D}^{\mathrm{T}}(\mathbf{x}) 
      \mathrm{grad}[\mathrm{c}]\right]
  \end{align}
  Noting the symmetry of diffusivity tensor, the 
  symmetric part of the spatial operator takes 
  the following form: 
  %--------------------------------;
  %  Equation: Symmetric operator  ;
  %--------------------------------;
  \begin{align}
    \widetilde{\mathcal{L}}[\mathrm{c}] = \frac{\mathcal{L}[\mathrm{c}] 
      + \mathcal{L}^{*}[\mathrm{c}]}{2} = \frac{1}{2} \mathrm{div}
              [\mathbf{v}] \mathrm{c} - \mathrm{div}[\mathbf{D}
                (\mathbf{x})\mathrm{grad}[\mathrm{c}]] 
  \end{align}
  The coefficient (i.e., ``stiffness'') matrix corresponding 
  to the operator $\widetilde{\mathcal{L}}[\mathrm{c}]$ over a finite 
  element $\Omega_e$ can be written as follows: 
  %------------------------------------;
  %  Equation: Local stiffness matrix  ;
  %------------------------------------;
  \begin{align}
    \boldsymbol{K}_{e} = \int_{\Omega_e} \frac{1}{2} 
    \mathrm{div}[\mathbf{v}] \boldsymbol{N}^{\mathrm{T}} 
    (\mathbf{x}) \boldsymbol{N}(\mathbf{x}) \; \mathrm{d} 
    \Omega + \int_{\Omega_e} \boldsymbol{B}(\mathbf{x}) 
    \mathbf{D}(\mathbf{x}) \boldsymbol{B}^{\mathrm{T}}
    (\mathbf{x}) \; \mathrm{d} \Omega 
  \end{align}
  where $\boldsymbol{N}\left( \mathbf{x}\right)$ is 
  the row vector containing shape functions, and 
  $\boldsymbol{B} \left( \mathbf{x}\right)$ is the matrix 
  containing the derivatives of shape functions with 
  respect to $\mathbf{x}$.
  Since $\mathrm{div}[\mathbf{v}] \geq 0$ and 
  $\mathbf{D}(\mathbf{x})$ is positive definite, 
  the matrix $\boldsymbol{K}_e$ will be positive 
  semi-definite. Since $\mathbf{D}(\mathbf{x})$ is 
  symmetric, the matrix $\boldsymbol{K}_e$ is 
  symmetric. The assembly procedure preserves 
  the positive semi-definiteness when the local 
  matrices are mapped to a global matrix. 
\end{proof}
One can similarly show that the symmetric part of the transport 
matrix under the GLS formulation is also positive semi-definite. 
On the other hand, the symmetric part of the transport matrix 
under the SUPG formulation will be positive semi-definite only 
if the diffusivity tensor is constant, and low-order simplicial 
elements (e.g, two-node element, three-node triangle element, 
four-node tetrahedron element) are employed.

%----------------------------------------;
%  Theorem: Stability with d-continuity  ;
%----------------------------------------;
\begin{theorem}[\textsf{Stability of the 
      $d$-continuity coupling method}]
\label{Theorem:First_d_continuity}
Under the proposed multi-time-step method with 
$d$-continuity, the rate variables 
$\boldsymbol{v}_i$ will remain bounded if
$1/2\leq\vartheta_i\leq 1 \; \forall i$.
\end{theorem}
%------------------------;
%  Proof of the theorem  ;
%------------------------;
\begin{proof}
  Using the notation introduced earlier, one can write: 
  \begin{subequations}
    \begin{align}
      \label{Eqn:S4_d_con_Stability_GE}
      &\boldsymbol{M}_i \left[ \boldsymbol{v}_i^{\left( n + \frac{j}{\eta_i}\right)} 
        \right]_i + \boldsymbol{K}_i 
      \left[ \boldsymbol{d}_i^{\left( n + \frac{j}{\eta_i}\right)} \right]_i
      = \frac{1}{\eta_i}\boldsymbol{C}_i^{\mathrm{T}} \left \llbracket \boldsymbol{\lambda}^{\left( n \right)} \right \rrbracket \\
      \label{Eqn:S4_d_con_Stability_Constraint}
      &\sum_{i = 1}^{\mathcal{S}} \boldsymbol{C}_i 
      \left\llbracket \boldsymbol{d}_i^{\left( n \right)} 
      \right\rrbracket = \boldsymbol{0}
\end{align}
\end{subequations}
where interpolation of Lagrange multipliers
using a first-order polynomial is used. 
For convenience, let us denote:
\begin{align}
  \boldsymbol{Q}_i := \boldsymbol{M}_i 
  + 2 \left(\vartheta_i - \frac{1}{2}\right) 
  \Delta t_i \; \mathrm{sym}\left[\boldsymbol{K}_i
    \right] 
\end{align}
Clearly, the matrix $\boldsymbol{Q}_i$ is symmetric, 
as the matrix $\boldsymbol{M}_i$ is symmetric. Since 
the matrix $\boldsymbol{M}_i$ is positive definite, 
the symmetric part of $\boldsymbol{K}_i$ is positive 
semi-definite, $\vartheta_i \geq 1/2$, and $\Delta t_i 
> 0$; one can conclude that the matrix $\boldsymbol{Q}_i$ 
is positive definite. 

Premultiplying both sides of equation 
\eqref{Eqn:S4_d_con_Stability_GE} by $\left[ 
\boldsymbol{d}_i^{\left( n + \frac{j}{\eta_i}\right)}\right]_i$ 
and using equation \eqref{Eqn:Trapezoidal_jump_d}, gives
the following equation:
\begin{align}
  \label{Eqn:S4_d_con_Stability_Int_1}
  \left\langle \boldsymbol{v}_i^{\left( n + \frac{j}{\eta_i}\right)}
  \right \rangle_i^{\mathrm{T}} \boldsymbol{Q}_i
  \left[ \boldsymbol{v}_i^{\left( n + \frac{j}{\eta_i}\right)}\right]_i 
  &+\left(\vartheta_i - \frac{1}{2}\right) 
  \left[ \boldsymbol{v}_i^{\left( n + \frac{j}{\eta_i}\right)}\right]_i^{\mathrm{T}} 
  \left(\boldsymbol{M}_i + (\vartheta_i - \frac{1}{2}) \Delta t_i 
  \; \mathrm{sym}\left[\boldsymbol{K}_i\right]\right)
  \left[ \boldsymbol{v}_i^{\left( n + \frac{j}{\eta_i}\right)}\right]_i \nonumber \\
  &+ \Delta t_i \left \langle \boldsymbol{v}_i^{\left( n + \frac{j}{\eta_i}\right)} 
  \right \rangle_i^{\mathrm{T}} \mathrm{sym}\left[\boldsymbol{K}_i\right] 
  \left \langle \boldsymbol{v}_i^{\left( n + \frac{j}{\eta_i}\right)} \right 
  \rangle_i 
  = \frac{1}{\Delta t} \left \llbracket 
  \boldsymbol{\lambda}^{\left( n \right)} \right \rrbracket^{\mathrm{T}} 
  \boldsymbol{C}_i \left[ \boldsymbol{d}_i^{\left( n + \frac{j}{\eta_i}\right)}
    \right]_i
\end{align}
Since the symmetric part of $\boldsymbol{K}_i$ is 
positive semi-definite, and $\Delta t_i > 0$, we 
have the following inequality: 
\begin{align}
  \left\langle \boldsymbol{v}_i^{\left( n + \frac{j}{\eta_i}\right)}
  \right \rangle_i^{\mathrm{T}} \boldsymbol{Q}_i
  \left[\boldsymbol{v}_i^{\left( n + \frac{j}{\eta_i}\right)}\right]_i 
  &+\left(\vartheta_i - \frac{1}{2}\right) 
  \left[ \boldsymbol{v}_i^{\left( n + \frac{j}{\eta_i}\right)}\right]_i^{\mathrm{T}} 
  \left(\boldsymbol{M}_i + (\vartheta_i - \frac{1}{2}) \Delta t_i 
  \; \mathrm{sym}\left[\boldsymbol{K}_i\right]\right)\left[ \boldsymbol{v}_i^{\left( n + \frac{j}{\eta_i}\right)}\right]_i \nonumber \\
  &\leq \frac{1}{\Delta t} \left \llbracket 
  \boldsymbol{\lambda}^{\left( n \right)} \right \rrbracket^{\mathrm{T}} 
  \boldsymbol{C}_i \left[ \boldsymbol{d}_i^{\left( n + \frac{j}{\eta_i}\right)}
    \right]_i
\end{align}
The matrices $\boldsymbol{M}_i$ and 
$\mathrm{sym} \left[ \boldsymbol{K}_i\right]$ are positive definite
and semidefinite respectively. In addition to that, if one
has $\vartheta_i \geq 1/2 \; \forall i$, then the following
inequality can be derived:
\begin{align}
  \left\langle \boldsymbol{v}_i^{\left( n + \frac{j}{\eta_i}\right)}
  \right \rangle_i^{\mathrm{T}} \boldsymbol{Q}_i 
  \left[ \boldsymbol{v}_i^{\left( n + \frac{j}{\eta_i}\right)}\right]_i 
  \leq \frac{1}{\Delta t} \left \llbracket 
  \boldsymbol{\lambda}^{\left( n \right)} \right \rrbracket^{\mathrm{T}} 
  \boldsymbol{C}_i \left[ \boldsymbol{d}_i^{\left( n + \frac{j}{\eta_i}\right)}
    \right]_i
\end{align}
Summing over all the subdomain time levels
(i.e., summing over $j$),
subdomains (i.e., summing over $i$),  
and using equation \eqref{Eqn:S4_d_con_Stability_Constraint}
will give:
\begin{align}
  \sum_{i=1}^{\mathcal{S}} \sum_{j=0}^{\eta_i - 1}
  \left\langle \boldsymbol{v}_i^{\left( n + \frac{j}{\eta_i}\right)}
  \right \rangle_i^{\mathrm{T}} \boldsymbol{Q}_i 
  \left[ \boldsymbol{v}_i^{\left( n + \frac{j}{\eta_i}\right)}\right]_i 
  \leq 0 
\end{align}
Since the matrices $\boldsymbol{Q}_i$ are symmetric, 
the above inequality can be rewritten as follows:
\begin{align}
  \sum_{i=1}^{\mathcal{S}} \sum_{j=0}^{\eta_i - 1}
  \left[ \left(\boldsymbol{v}_i^{\left( n + \frac{j}{\eta_i}\right)}
  \right)^{\mathrm{T}} \boldsymbol{Q}_i 
    \boldsymbol{v}_i^{\left( n + \frac{j}{\eta_i}\right)}\right]_i 
  \leq 0 
\end{align}
By executing the telescopic summation, 
we obtain the following: 
\begin{align}
  \sum_{i=1}^{\mathcal{S}} 
  \left\llbracket \left(\boldsymbol{v}_i^{(n)}
  \right)^{\mathrm{T}} \boldsymbol{Q}_i 
  \boldsymbol{v}_i^{(n)}\right\rrbracket 
  \leq 0 
\end{align}
This further implies that 
\begin{align}
  \sum_{i = 1}^{\mathcal{S}} {\boldsymbol{v}_i^{(n)}}^{\mathrm{T}} 
  \boldsymbol{Q}_i \boldsymbol{v}_i^{(n)} \leq 
  \sum_{i = 1}^{\mathcal{S}} {\boldsymbol{v}_i^{(n-1)}}^{\mathrm{T}} 
  \boldsymbol{Q}_i \boldsymbol{v}_i^{(n-1)} \leq \cdots 
  \leq \sum_{i = 1}^{\mathcal{S}} {\boldsymbol{v}_i^{(0)}}^{\mathrm{T}} 
  \boldsymbol{Q}_i \boldsymbol{v}_i^{(0)}
\end{align}
Boundedness of $\boldsymbol{v}_i^{(0)}$ 
and positive definiteness of matrices $\boldsymbol{Q}_i$ ($i = 1, \cdots, 
\mathcal{S}$) concludes the boundedness of $\boldsymbol{v}_i^{(n)}$, in
all subdomains and at all time-levels.
\end{proof}

%---------------------------------------------;
%  Remark: Remark about \vartheta_i \geq 1/2  ;
%---------------------------------------------;
\begin{remark}
  One cannot relax the condition $\vartheta_i 
  \geq 1/2$ under the coupling method based 
  on the $d$-continuity method. It should be 
  noted that one would obtain numerical instability 
  if this condition is violated. This will be the case 
  even if one does not employ subcycling 
  \citep{Nakshatrala_Prakash_Hjelmstad_JCP_2009_v228_p7957}. 
  However, the main advantage of employing the coupling 
  method based on the $d$-continuity is that one can choose 
  any system time-step and subdomain time-step, and still 
  achieve numerical stability.  
\end{remark}

We now assess the stability of the proposed coupling 
method based on the Baumgarte stabilization. We are 
able to construct a proof only for the case in which 
the matrices $\boldsymbol{K}_i$ are symmetric. This 
means that the proof does not hold for the case in which 
advection is present. However, the numerical results 
presented in a subsequent section show that the coupling 
method based on the Baumgarte stabilization provide 
stable solutions even in the presence of advection. 
It is therefore a good research problem to theoretically 
assess the stability of the coupling method based on the 
Baumgarte stabilization in the presence of advection. 

%-------------------------------------;
%  Theorem: Stability with Baumgarte  ;
%-------------------------------------;
\begin{theorem}[\textsf{Stability of the proposed 
    method with Baumgarte stabilization}]
\label{Thm:S4_Stability_Baumgarte}
Under the proposed multi-time-step method with Baumgarte 
stabilization, the rate variables $\boldsymbol{v}_i$, will
remain bounded if one chooses $\Delta t_i \leq \Delta 
t_i^{\mathrm{critical}}$ and $\alpha \leq \alpha_{\mathrm{max}}$ 
where 
\begin{subequations}
  \label{Eqn:Baum_Stab}
  \begin{align}
    &\Delta t_i^{\mathrm{critical}} := \left\{ \begin{array}{c}
      \frac{2}{\left( 1 - 2 \vartheta_i\right) \omega_i} \quad
      \text{if} \quad 0 \leq \vartheta_i < 1/2 \\
      +\infty \quad \quad \quad \text{if} \quad 1/2 \leq 
      \vartheta_i \leq 1 \end{array} \right. \\
    &\alpha_{\mathrm{max}} := \left\{ \begin{array}{c}
      \min \left\{ \frac{2 \eta_i}{1 - 2\vartheta_i} : 0 
      \leq \vartheta_i < 1/2\right \}\\
      +\infty \quad \text{if} \quad 1/2 \leq \vartheta_i 
      \leq 1 \; \forall i
    \end{array}\right. 
\end{align}
\end{subequations}
and $\omega_i = \mathrm{max}
\left\{ \omega \; : \; \mathrm{det} \left( \omega 
\boldsymbol{I}_i - \boldsymbol{M}_i^{-1} 
\boldsymbol{K}_i\right) = 0\right\}.$ 
It is assumed that the matrices $\boldsymbol{K}_i \; 
(i = 1, \cdots, \mathcal{S})$ are symmetric and 
positive semi-definite. 
\end{theorem}
%------------------------;
%  Proof of the theorem  ;
%------------------------;
\begin{proof} 
Consider the following equations: 
\begin{subequations}
\begin{align}
\label{Eqn:S4_Baum_Stab_1}
 &\boldsymbol{M}_i \left[ \boldsymbol{v}_i^{\left( n + \frac{j}{\eta_i}\right)} \right]_i
 + \boldsymbol{K}_i \left[ \boldsymbol{d}_i^{\left( n + \frac{j}{\eta_i}\right)} \right]_i
 = \frac{1}{\eta_i}\boldsymbol{C}_i^{\mathrm{T}} \left \llbracket \boldsymbol{\lambda}^{\left( n \right)} \right \rrbracket \\
\label{Eqn:S4_Baum_Stab_2}
 &\sum_{i = 1}^{S} \boldsymbol{C}_i \left( \left \llbracket \boldsymbol{v}_i^{\left( n \right)}\right \rrbracket + \frac{\alpha}{\Delta t} \left \llbracket \boldsymbol{d}_i^{\left( n \right)} \right \rrbracket\right) = \boldsymbol{0}
\end{align}
\end{subequations}
Premultiplying both sides of equation \eqref{Eqn:S4_Baum_Stab_1} 
by $\left[\boldsymbol{v}_i^{\left( n + \frac{j}{\eta_i} \right)}\right]_i 
+ \frac{\alpha}{\Delta t}\left[ \boldsymbol{d}_i^{\left( n + \frac{j}{\eta_i} 
\right)}\right]_i$ we obtain the following:
\begin{align}
 &\left[ \boldsymbol{v}_i^{\left( n + \frac{j}{\eta_i} \right)}\right]_i^{\mathrm{T}}
 \boldsymbol{M}_i \left[\boldsymbol{v}_i^{\left( n + \frac{j}{\eta_i} \right)}\right]_i
 + \frac{\alpha}{\Delta t} \left[ \boldsymbol{d}_i^{\left( n + \frac{j}{\eta_i} \right)}\right]_i^{\mathrm{T}} \boldsymbol{M}_i \left[ \boldsymbol{v}_i^{\left( n + \frac{j}{\eta_i} \right)}\right]_i 
 +\left[ \boldsymbol{v}_i^{\left( n + \frac{j}{\eta_i} \right)}\right]_i^{\mathrm{T}} 
 \boldsymbol{K}_i \left[ \boldsymbol{d}_i^{\left( n + \frac{j}{\eta_i} \right)}\right]_i 
 \nonumber \\
 &\qquad + \frac{\alpha}{\Delta t} \left[ \boldsymbol{d}_i^{\left( n + \frac{j}{\eta_i} \right)}\right]_i^{\mathrm{T}} \boldsymbol{K}_i \left[ \boldsymbol{d}_i^{\left( n + \frac{j}{\eta_i} \right)}\right]_i 
 = \frac{1}{\eta_i} \left \llbracket \boldsymbol{\lambda}^{\left( n \right)}\right \rrbracket ^
 {\mathrm{T}}\boldsymbol{C}_i \left( \left[ \boldsymbol{v}_i^{\left( n + \frac{j}{\eta_i} \right)}\right]_i + \frac{\alpha}{\Delta t}
 \left[ \boldsymbol{d}_i^{\left( n + \frac{j}{\eta_i} \right)}\right]_i \right)
\end{align}
Employing equation \eqref{Eqn:Trapezoidal_jump_d} yields:
\begin{align}
\label{Eqn:S4_Baum_Stability_Int_1}
 &\left[ \boldsymbol{v}_i^{\left( n + \frac{j}{\eta_i} \right) }\right]_i^{\mathrm{T}} 
 \left( \left( 1 + \alpha \left( \vartheta_i - \frac{1}{2} \right)\frac{\Delta t_i}{\Delta t}\right)\boldsymbol{M}_i
 +  \Delta t_i \left( \vartheta_i - \frac{1}{2} \right)\left( 1 + \alpha \left( \vartheta_i - \frac{1}{2} \right)\frac{\Delta t_i}{\Delta t}\right) \boldsymbol{K}_i \right) 
 \left[ \boldsymbol{v}_i^{\left( n + \frac{j}{\eta_i} \right) }\right]_i  \nonumber \\
+ &\left\langle \boldsymbol{v}_i^{\left( n + \frac{j}{\eta_i} \right)} \right\rangle_i ^{\mathrm{T}}
\left( \alpha \frac{\Delta t_i}{\Delta t} \boldsymbol{M}_i + \Delta t_i
\left( 1 + 2 \alpha \left( \vartheta_i - \frac{1}{2}\right) \frac{\Delta t_i}{\Delta t}\right)\boldsymbol{K}_i \right) \left[ \boldsymbol{v}_i^
{\left( n + \frac{j}{\eta_i} \right)} \right]_i 
+ \alpha \frac{\Delta t_i^{2}}{\Delta t} \left\langle \boldsymbol{v}_i^{\left( n + \frac{j}{\eta_i} \right) }\right\rangle_i^{\mathrm{T}} \boldsymbol{K}_i \left\langle \boldsymbol{v}_i^{\left( n + \frac{j}{\eta_i} \right) }\right\rangle_i \nonumber \\ 
= & \frac{1}{\eta_i}\left \llbracket \boldsymbol{\lambda}^{\left(n\right)} \right \rrbracket ^{\mathrm{T}} \boldsymbol{C}_i
\left( \left[ \boldsymbol{v}_i^{\left( n + \frac{j}{\eta_i} \right)}\right]_i +
 \frac{\alpha}{\Delta t}\left[ \boldsymbol{d}_i^{\left( n + \frac{j}{\eta_i} 
\right)}\right]_i \right)
\end{align}
Note that the parameters $\alpha$, $\Delta t_i$, and $ 
\Delta t$ are strictly positive. The matrices $\boldsymbol{K}_i$ 
are assumed to be positive semi-definite. Thus, we have the following
inequality: 
\begin{align}
  \left[ \boldsymbol{v}_i^{\left( n + \frac{j}{\eta_i} \right) }
    \right]_i^{\mathrm{T}} \boldsymbol{P}_i
  \left[ \boldsymbol{v}_i^{\left( n + \frac{j}{\eta_i} \right) }\right]_i  
  &+ \left\langle \boldsymbol{v}_i^{\left( n + \frac{j}{\eta_i} \right)} 
  \right\rangle_i ^{\mathrm{T}} \boldsymbol{U}_i
  \left[ \boldsymbol{v}_i^
    {\left( n + \frac{j}{\eta_i} \right)} \right]_i \nonumber \\
  &\leq \left \llbracket \boldsymbol{\lambda}^{\left(n\right)} 
\right \rrbracket ^{\mathrm{T}} \boldsymbol{C}_i \left(\left[ 
  \boldsymbol{v}_i^{\left( n + \frac{j}{\eta_i} \right)}\right]_i +
\frac{\alpha}{\Delta t}\left[ \boldsymbol{d}_i^{\left( n + \frac{j}{\eta_i} 
\right)}\right]_i \right)
\end{align}
where 
\begin{subequations}
  \begin{align}
    &\boldsymbol{P}_i := \left(\eta_i + \alpha 
    \left(\vartheta_i - \frac{1}{2}\right)\right)\boldsymbol{M_i}
    +  \Delta t_i \left(\vartheta_i - \frac{1}{2}\right) 
    \left(\eta_i + \alpha \left( \vartheta_i - \frac{1}{2}\right)\right) 
    \boldsymbol{K}_i \\
    &\boldsymbol{U}_i := \alpha \boldsymbol{M}_i + 
    \Delta t_i \left(\eta_i + 2 \alpha \left(\vartheta_i 
    - \frac{1}{2}\right)\right)\boldsymbol{K}_i
  \end{align}
\end{subequations}
Summing over all the subdomains (i.e., summing over $i$) 
and subdomain time-levels (i.e., summing 
over $j$), gives the following inequality: 
\begin{align}
  \label{Eqn:Baum_Stability_Int_2}
  \sum_{i = 1}^{\mathcal{S}} \sum_{j = 0}^{\eta_i - 1} 
  \left[\boldsymbol{v}_i^{\left( n + \frac{j}	
      {\eta_i} \right)} \right]_i^{\mathrm{T}} \boldsymbol{P}_i
  \left[ \boldsymbol{v}_i^{\left( n + \frac{j}{\eta_i} \right)}\right]_i 
  &+  \sum_{i = 1}^{\mathcal{S}} \sum_{j = 0}^{\eta_i - 1} 
  \left\langle \boldsymbol{v}_i^{\left( n + \frac{j}{\eta_i} \right)} 
  \right\rangle_i ^{\mathrm{T}} \boldsymbol{U}_i \left[ 
    \boldsymbol{v}_i^{\left( n + \frac{j}{\eta_i} \right)} \right]_i \nonumber \\
  &\leq \left \llbracket \boldsymbol{\lambda}^{\left( n \right)} 
  \right \rrbracket^{\mathrm{T}} \sum_{i = 1}^{\mathcal{\mathcal{S}}} 
  \boldsymbol{C}_i \left( \left \llbracket \boldsymbol{v}_i^{\left( n \right)}
  \right \rrbracket + \frac{\alpha}{\Delta t} \left \llbracket 
  \boldsymbol{d}_i^{\left( n \right)} \right \rrbracket \right)
\end{align}  
The compatibility condition along the subdomain 
interface in the form given by equation 
\eqref{Eqn:Baum_Stability_Int_2} implies 
that 
\begin{align}
  \sum_{i = 1}^{\mathcal{S}} \sum_{j = 0}^{\eta_i - 1} 
  \left[\boldsymbol{v}_i^{\left( n + \frac{j}	
      {\eta_i} \right)} \right]_i^{\mathrm{T}} \boldsymbol{P}_i
  \left[ \boldsymbol{v}_i^{\left( n + \frac{j}{\eta_i} \right)}\right]_i 
  +  \sum_{i = 1}^{\mathcal{S}} \sum_{j = 0}^{\eta_i - 1} 
  \left\langle \boldsymbol{v}_i^{\left( n + \frac{j}{\eta_i} \right)} 
  \right\rangle_i ^{\mathrm{T}} \boldsymbol{U}_i \left[ 
    \boldsymbol{v}_i^{\left( n + \frac{j}{\eta_i} \right)} \right]_i 
  \leq 0
\end{align}  
From the hypothesis of the theorem, it is easy 
to show that the matrix $\boldsymbol{P}_i$ is 
positive semi-definite. 
% To wit, 
% \begin{align}
%  \boldsymbol{P}_i &:= \left(\eta_i + \alpha 
%  \left(\vartheta_i - \frac{1}{2}\right)\right)\boldsymbol{M_i}
%  +  \Delta t_i \left(\vartheta_i - \frac{1}{2}\right) 
%  \left(\eta_i + \alpha \left( \vartheta_i - \frac{1}{2}\right)\right) 
%  \boldsymbol{K}_i \\
%  %
%  &= \underbrace{\left(\eta_i + \alpha \left(\vartheta_i 
%    - \frac{1}{2}\right)\right)}_{\mbox{$\geq 0$ due to equation \eqref{}}}
%  \left(\boldsymbol{M_i} +  \Delta t_i \left(\vartheta_i - \frac{1}{2}\right) 
%  \boldsymbol{K}_i \right)
% \end{align}
%
This implies that we have the following inequality:
\begin{align}
   \sum_{i = 1}^{\mathcal{S}} \sum_{j = 0}^{\eta_i - 1} 
  \left\langle \boldsymbol{v}_i^{\left( n + \frac{j}{\eta_i} \right)} 
  \right\rangle_i ^{\mathrm{T}} \boldsymbol{U}_i \left[ 
    \boldsymbol{v}_i^{\left( n + \frac{j}{\eta_i} \right)} \right]_i 
  \leq 0
\end{align}  
It is easy to check that $\boldsymbol{U}_i$ 
is symmetric, which implies the following: 
\begin{align}
  \sum_{i = 1}^{\mathcal{S}} \sum_{j = 0}^{\eta_i - 1} 
  \left\langle \boldsymbol{v}_i^{\left( n + \frac{j}{\eta_i} \right)} 
  \right\rangle_i ^{\mathrm{T}} \boldsymbol{U}_i \left[ 
    \boldsymbol{v}_i^{\left( n + \frac{j}{\eta_i} \right)} \right]_i 
  &= \sum_{i = 1}^{\mathcal{S}} \sum_{j = 0}^{\eta_i - 1} 
   \frac{1}{2}\left[{\boldsymbol{v}_i^{\left( n + \frac{j}{\eta_i} \right)}}^{\mathrm{T}} 
     \boldsymbol{U}_i \boldsymbol{v}_i^{\left( n + \frac{j}{\eta_i} \right)} \right]_i 
     \nonumber \\
  &= \frac{1}{2} \sum_{i = 1}^{\mathcal{\mathcal{S}}} 
  \left\llbracket{\boldsymbol{v}_i^{\left( n \right)}}^{\mathrm{T}} 
  \boldsymbol{U}_i\boldsymbol{v}_i^{\left( n \right)} \right\rrbracket 
  \leq 0 \quad \forall n
\end{align}
This further implies that 
\begin{align}
  \sum_{i = 1}^{\mathcal{S}} {\boldsymbol{v}_i^{(n)}}^{\mathrm{T}} 
  \boldsymbol{U}_i \boldsymbol{v}_i^{(n)} \leq 
  \sum_{i = 1}^{\mathcal{S}} {\boldsymbol{v}_i^{(n-1)}}^{\mathrm{T}} 
  \boldsymbol{U}_i \boldsymbol{v}_i^{(n-1)} \leq \cdots \leq
  \sum_{i = 1}^{\mathcal{S}} {\boldsymbol{v}_i^{(0)}}^{\mathrm{T}} 
  \boldsymbol{U}_i \boldsymbol{v}_i^{(0)}
\end{align}
Since the matrices $\boldsymbol{U}_i \; (i= 1, 
\cdots, \mathcal{S})$ are positive definite, and the 
initial rates $\boldsymbol{v}_i^{(0)}$ are bounded, 
one can conclude that the rate variables will remain 
bounded at all time-levels. 
\end{proof}

%=============================================================;
%  Subsection: Drift in continuity of temperatures and rates  ;
%=============================================================;
\subsection{Bounds on drifts in concentrations and rate variables}
A well-known phenomenon appearing in numerical solutions 
of DAEs is the drift in the compatibility/constraint 
equations \citep{Hairer_Wanner_V2}). In our case, the 
drift will manifest as discontinuity in the primary 
and/or rate variables along the subdomain interface. 
The drifts will be different for two proposed 
coupling methods, as they differ in handling 
compatibility conditions along the subdomain 
interface. 
Herein, we shall ignore subcycling (i.e., $\eta_i = 1\; 
\forall i$), and assume that $\vartheta_i = \vartheta \; 
\forall i$. The following notation is employed:
\begin{align}
  \label{Eqn:S4_Drift_Notation}
  \boldsymbol{d}_{\mathrm{drift}}^{(n)} := 
  \sum_{i=1}^{\mathcal{S}} \boldsymbol{C}_i 
  \boldsymbol{d}_i^{(n)}\; , \quad
  \boldsymbol{v}_{\mathrm{drift}}^{(n)} := 
  \sum_{i = 1}^{\mathcal{S}} \boldsymbol{C}_i 
  \boldsymbol{v}_i^{(n)}
\end{align}
Under the $d$-continuity coupling method, 
by construction of the method, there is 
no drift in the primary variable (i.e., 
concentration) along the subdomain 
interface at all system time levels. 
The drift in the rate satisfy the 
following recursive relation:
\begin{align}
\label{Eqn:d_con_drift}
	\boldsymbol{v}_{\mathrm{drift}}^{\left( n + 1 \right)} = 
	\left( 1 - \frac{1}{\vartheta} \right) \boldsymbol{v}_
	{\mathrm{drift}}^{\left( n \right)} \quad \forall n > 1
\end{align}
Note that if the implicit Euler method (i.e., $\vartheta = 1$) 
is employed then the drifts at system time-levels will be zero 
in both concentrations and rate variables.	

Under the proposed coupling method with Baumgarte 
stabilization, the following recursive relations 
hold:
\begin{subequations}
\label{Eqn:S4_Drift_Measurement}
\begin{align}
	&\boldsymbol{d}_{\mathrm{drift}}^{\left( n + 1 \right)} = \frac{1}{1 + \alpha \vartheta} \boldsymbol{d}_{\mathrm{drift}}^{\left( n \right)} + 
	\frac{\Delta t \left( 1 - \vartheta \right)}{1 + \alpha \vartheta} \boldsymbol{v}_{\mathrm{drift}}^{\left( n \right)} \quad \forall n > 1\\ 
	&\boldsymbol{v}_{\mathrm{drift}}^{\left( n + 1 \right)} = 
	- \frac{\alpha}{\Delta t \left( 1 + \alpha \vartheta \right)} 
	\boldsymbol{d}_{\mathrm{drift}}^{\left( n \right)} - \frac{\alpha \left( 1 - \vartheta \right)}{1 + \alpha \vartheta} \boldsymbol{v}_{\mathrm{drift}}^{\left( n \right)} \quad \forall  n > 1
\end{align}
\end{subequations}
which imply that choosing larger $\alpha$ will 
decrease drifts in concentration.
It should be noted that subcycling, and mixed methods 
can have adverse effects on the drifts. That is, the 
drifts can be worse than predictions made by the 
above bounds. However, the above relations can be 
valuable to check a computer implementation, and 
can show a general trend of the drifts in the numerical 
time integration process. In a subsequent section, 
some numerical results are presented to corroborate 
the aforementioned theoretical predictions.

%=============================================;
%  Subsection: On influence of perturbations  ;
%=============================================;
\subsection{Influence of perturbations}
In this section, we will study the propagation of 
perturbations over a system time-step. This analysis
will help us to better understand how perturbations in 
input (in this case, previous time-level) will affect 
the solution at the next time-level. In the following
theorem, we will consider application of the proposed
method to non-linear DAEs of the form:
\begin{align}
  &\boldsymbol{M}_i \dot{\boldsymbol{c}}_i \left( t \right) =
  \boldsymbol{h}_i \left( \boldsymbol{c}_i \left( t\right),
  t \right) + \boldsymbol{C}_i^{\mathrm{T}} \boldsymbol{\lambda} \left( t \right) \quad \forall i\\
  &\sum_{i = 1}^{\mathcal{S}} \boldsymbol{C}_i \boldsymbol{c}_i \left( t \right) = \boldsymbol{0}
\end{align}
\begin{theorem}
 	Let $\left( \widehat{\boldsymbol{v}}_i^{\left( n + (j+1)/\eta_i\right)},
	\widehat{\boldsymbol{d}}_i^{\left( n + (j+1)/\eta_i \right)},
	\widehat{\lambda}^{(n+1)}\right)$ with $j = 1 , \cdots , \eta_i - 1 $
	and $i = 1, \cdots, \mathcal{S}$ be the solution of the 
	following system
	\begin{subequations}
	\label{Eqn:IoP_perturbed}
	\begin{align}
		& \widehat{\boldsymbol{v}}_i^{\left( n + \frac{j+1}{\eta_i}\right)} = 
		\boldsymbol{M}_{i}^{-1} \boldsymbol{h}_i 									\left( \widehat{\boldsymbol{d}}_i^{\left( n + \frac{j+1}{\eta_i}			\right)} , t^{\left( n + \frac{j+1}{\eta_i}\right)}\right) + 
		\boldsymbol{M}_{i}^{-1}\boldsymbol{C}_i^{\mathrm{T}} 						\widehat{\boldsymbol{\lambda}}^{\left( n + \frac{j+1}{\eta_i}\right)} 		\\
                \label{Eqn:IoP_perturbed_d}
		& \widehat{\boldsymbol{d}}_i^{\left( n + \frac{j+1}{\eta_i} \right)} 		= \widehat{\boldsymbol{d}}_i^{\left( n + \frac{j+1}{\eta_i}\right)} + 
		\Delta t_i \left( 1 - \vartheta_i\right) \widehat{\boldsymbol{v}}			_i^{\left( n + \frac{j}{\eta_i}\right)} + \Delta t_i \vartheta_i 			\widehat{\boldsymbol{v}}_i^{\left( n + \frac{j+1}{\eta_i}\right)} + 		\Delta t_i \boldsymbol{\varepsilon}_{d_i} \\
		&\widehat{\boldsymbol{\lambda}}^{\left( n + \frac{j+1}{\eta_i}				\right)} = \left( 1 - \frac{j+1}{\eta_i}\right) 							\widehat{\boldsymbol{\lambda}}^{(n)} + 
		\left( \frac{j+1}{\eta_i} \right) \widehat{\boldsymbol{\lambda}}^{(
		+1)} + \Delta t \boldsymbol{\Delta}_{\lambda} \\
		& \underbrace{\sum_{i = 1}^{\mathcal{S}} \boldsymbol{C}_i 					\widehat{\boldsymbol{d}}_i^{(n+1)} = \boldsymbol{\varepsilon}				_{\lambda}}_{d-\mathrm{continuity}} \quad \text{or} \quad 				\underbrace{\sum_{i = 1}^{\mathcal{S}}\boldsymbol{C}_i 						\left( \widehat{\boldsymbol{v}}_i^{(n+1)} + 								\frac{\alpha}{\Delta t} \widehat{\boldsymbol{d}}_i^{(n+1)}\right) = 		\frac{1}{\Delta t}\boldsymbol{\varepsilon}_{\lambda}}_{\mathrm{Baumgarte \;			stabilization}}
	\end{align}
	\end{subequations}
	in which we have assumed that 
	\begin{align}
	\boldsymbol{\Delta}_{\lambda} = O(\Delta t), \quad
	\boldsymbol{\varepsilon}_{d_i} = O(\Delta t_i), \quad
	\boldsymbol{\varepsilon}_{\lambda} = O(\Delta t^2)
    \end{align}
       Furthermore,
	\begin{align}
	  \widehat{\boldsymbol{v}}_i^{\left( n\right)} - \boldsymbol{v}_i^{\left( n\right)} = O(\Delta t_i),
		 \quad \widehat{\boldsymbol{d}}_i^{\left( n\right)} - \boldsymbol{d}_i^{\left( n\right)} 
                 = O(\Delta t_i^2), \quad \widehat{\boldsymbol{\lambda}}^{(n)} - \boldsymbol{\lambda}^{(n)} 
                 = O(\Delta t)
	\end{align}
	Let the functions $\boldsymbol{M}_i^{-1}\boldsymbol{h}_i \; 
        (i = 1, \cdots, \mathcal{S})$ be Lipschitz continuous, 	then the following inequalities will hold:
	\begin{subequations}
	  \label{Eqn:IoP_d_continuity}
	  \begin{align}
	    &\left\| \delta \boldsymbol{d}_i^{(n+1)}\right\| \leq \mathcal{C}_{d} 
	    \left( \sum_{l = 1}^{\mathcal{S}} \left( \left\| \delta \boldsymbol{d}_l^{(n)}\right\| 
            + \Delta t \left\| \boldsymbol{\varepsilon}_{d_l}\right\| \right) 
            + \Delta t \left\| \delta \boldsymbol{\lambda}^{(n)}\right\| 
            + \phi \left\| \boldsymbol{\varepsilon}_{\lambda}\right\| 
            + \Delta t^2 \left\|\boldsymbol{\Delta}_{\lambda}\right\|\right) \\
	    &\left\| \delta \boldsymbol{v}_i^{(n+1)}\right\| \leq 
            \mathcal{C}_{v} \left(\sum_{l = 1}^{\mathcal{S}} 
            \left(\frac{1}{\Delta t}\left\| \delta \boldsymbol{d}_l^{(n)}
            \right\| 
            + \left\| \boldsymbol{\varepsilon}_{d_l}\right\| \right) 
            + \left\| \delta \boldsymbol{\lambda}^{(n)}\right\| 
            + \frac{\phi}{\Delta t}\left\| \boldsymbol{\varepsilon}_{\lambda}\right\| 
            + \Delta t \left\| \boldsymbol{\Delta}_{\lambda}\right\|\right) \\
	    &\left\| \delta \boldsymbol{\lambda}^{(n+1)}\right\| 
            \leq \mathcal{C}_{\lambda} \left( \sum_{l = 1}^{\mathcal{S}} 
            \left( \frac{1}{\Delta t}\left\|\delta \boldsymbol{d}_l^{(n)}\right\| 
            + \left\| \boldsymbol{\varepsilon}_{d_l}\right\| \right) 
            + \left\| \delta \boldsymbol{\lambda}^{(n)}\right\| 
            + \frac{\phi}{\Delta t}\left\| \boldsymbol{\varepsilon}_{\lambda}\right\| 
            + \Delta t \left\| \boldsymbol{\Delta}_{\lambda}\right\|\right)
	  \end{align}
	\end{subequations}
	where $\mathcal{C}_{d}$, $\mathcal{C}_{v}$, 
	$\mathcal{C}_{\lambda}$ are constants, $\delta 
        \Box = \widehat{\Box} - \Box$, and 
        \begin{align}
          \phi = \left\{\begin{array}{ll}
          1 & d-\mathrm{continuity \; method} \\
          \Delta t & \mathrm{Baumgarte \; stabilization \; method} 
          \end{array} \right.
        \end{align}
\end{theorem}
\begin{proof}
From equation \eqref{Eqn:IoP_perturbed} we can write:
\begin{align}
\label{Eqn:IoP_temp_1}
	\delta \boldsymbol{v}_i^{\left( n + \frac{j+1}{\eta_i}\right)}
	&= \boldsymbol{M}_i^{-1} \left( \boldsymbol{h}_i \left( \widehat{\boldsymbol{d}}_i^{\left( n + \frac{j+1}{\eta_i}\right)} , t^{\left( n + \frac{j+1}{\eta_i}\right)}\right) - \boldsymbol{h}_i \left( \boldsymbol{d}_i^{\left( n + \frac{j+1}{\eta_i}\right)} , t^{\left( n + \frac{j+1}{\eta_i}\right)}\right)\right)  \nonumber \\ &+ \boldsymbol{M}_i^{-1} \boldsymbol{C}_i^{\mathrm{T}} \left( \left( 1 - \frac{j+1}{\eta_i}\right) \delta \boldsymbol{\lambda}^{(n)} + \left( \frac{j+1}{\eta_i}\right) \delta\boldsymbol{\lambda}^{(n+1)} + \Delta t \boldsymbol{\Delta}_{\lambda}\right) \quad \forall i
\end{align}
Lipschitz continuity of  
functions $\boldsymbol{M}_i^{-1}\boldsymbol{h}_i$ and 
$\boldsymbol{M}_i^{-1} \boldsymbol{C}_i^{\mathrm{T}}$ 
can be used to obtain the following inequalities:  
\begin{subequations}
\begin{align} 
&\left\| \boldsymbol{M}_i^{-1}\boldsymbol{h}_i \left( \widehat{\boldsymbol{d}}_i^{\left( n + \frac{j+1}{\eta_i}\right)}, t^{\left( n + \frac{j+1}{\eta_i}\right)}\right)
- \boldsymbol{M}_i^{-1}\boldsymbol{h}_i \left( \boldsymbol{d}_i^{\left( n + \frac{j+1}{\eta_i}\right)}, t^{\left( n + \frac{j+1}{\eta_i}\right)}\right)
\right\| \leq \mathcal{C}_i^{h} \left\| \delta 
\boldsymbol{d}_i^{\left( n + \frac{j+1}{\eta_i}\right)}\right\| \\
&\left \|  \boldsymbol{M}_i^{-1} \boldsymbol{C}_i^{\mathrm{T}} 
\left( \widehat{\boldsymbol{\lambda}}^{\left( n + \frac{j+1}{\eta_i}\right)} - \boldsymbol{\lambda}^{\left( n + \frac{j+1}{\eta_i}\right)}\right) \right \| \leq \mathcal{C}_{i}^{\lambda} \left\| \delta \boldsymbol{\lambda}^{\left( n + \frac{j+1}{\eta_i}\right)}\right\|
\end{align}
\end{subequations}
By taking norms of both sides of equation \eqref{Eqn:IoP_temp_1}, 
and applying the triangle inequality, we obtain the following: 
\begin{align}
\label{Eqn:IoP_deltav_1}
	\left\| \delta \boldsymbol{v}_i^{\left( n + \frac{j+1}{\eta_i}\right)}\right\|
	&\leq \mathcal{C}_i^{h} \left\| \delta \boldsymbol{d}_i^{\left( n + \frac{j+1}{\eta_i}\right)}\right\| + \left( 1 - \frac{j+1}{\eta_i}\right) \mathcal{C}_i^{\lambda} \left\| \delta \boldsymbol{\lambda}^{(n)}\right\| + 
	\left( \frac{j+1}{\eta_i}\right) \mathcal{C}_i^{\lambda} \left\| \delta \boldsymbol{\lambda}^{(n+1)}\right\| + \Delta t \mathcal{C}_i^{\lambda} \left\| \boldsymbol{\Delta}_{\lambda}\right\| \nonumber \\
	&\leq \mathcal{C}_i^{h} \left\| \delta \boldsymbol{d}_i^{\left( n + \frac{j+1}{\eta_i}\right)}\right\| + \mathcal{C}_i^{\lambda} \left\| \delta \boldsymbol{\lambda}^{(n)}\right\| + 
	 \mathcal{C}_i^{\lambda} \left\| \delta \boldsymbol{\lambda}^{(n+1)}\right\| + \Delta t \mathcal{C}_i^{\lambda} \left\| \boldsymbol{\Delta}_{\lambda}\right\|
\end{align}
Note that $0 \leq (j+1)/\eta_i \leq 1 \; \forall j$. 
Using equation \eqref{Eqn:IoP_perturbed_d} one can 
obtain the following inequality: 
\begin{align}
\label{Eqn:IoP_deltad_1}
	\left\| \delta \boldsymbol{d}_i^{\left( n + \frac{j+1}{\eta_i}\right)}\right\|
	&\leq \left\| \delta \boldsymbol{d}_i^{\left( n + \frac{j}{\eta_i}\right)}\right\| + \Delta t_i \left( 1 - \vartheta_i\right)\left\| \delta \boldsymbol{v}_i^{\left( n + \frac{j}{\eta_i}\right)}\right\| + \Delta t_i \vartheta_i \left\| \delta \boldsymbol{v}_i^{\left( n + \frac{j+1}{\eta_i}\right)}\right\| + \Delta t_i \left\| \boldsymbol{\varepsilon}_{d_i}\right\|
	\nonumber \\
	& \leq \left\| \delta \boldsymbol{d}_i^{\left( n + \frac{j}{\eta_i}\right)}\right\| + \Delta t_i \left\| \delta \boldsymbol{v}_i^{\left( n + \frac{j}{\eta_i}\right)}\right\| + \Delta t_i \left\| \delta \boldsymbol{v}_i^{\left( n + \frac{j+1}{\eta_i}\right)}\right\| + \Delta t_i \left\| \boldsymbol{\varepsilon}_{d_i}\right\|
\end{align}
Inequalities \eqref{Eqn:IoP_deltav_1} and 
\eqref{Eqn:IoP_deltad_1} imply the following:
\begin{align}
\left( 1 - \Delta t_i \mathcal{C}_i^{h}\right) \left\| \delta \boldsymbol{d}_i^{\left( n + \frac{j+1}{\eta_i}\right)}\right\| & \leq 
\left( 1 + \Delta t_i \mathcal{C}_i^{h}\right) \left\| \delta \boldsymbol{d}_i^{\left( n + \frac{j}{\eta_i}\right)}\right\| + 2 \Delta t_i \mathcal{C}_i^{\lambda} \left\| \delta \boldsymbol{\lambda}^{(n)}\right\| + 
2 \Delta t_i \mathcal{C}_i^{\lambda} \left\| \delta \boldsymbol{\lambda}^{(n+1)}\right\| \nonumber \\ &+ 2 \Delta t_i \Delta t \mathcal{C}_i^{\lambda} \left\| \boldsymbol{\Delta}_{\lambda}\right\| + \Delta t_i \left\| \boldsymbol{\varepsilon}_{d_i}\right\|
\end{align}
We shall assume that the subdomain time-steps 
$\Delta t_i$ are sufficiently small such that 
$1 - \Delta t_i \mathcal{C}_i^{h} > 0$ holds. 
Then, the propagation of perturbations over a 
subdomain time-step will satisfy the following
inequality:
\begin{align}
\label{Eqn:IoP_deltad_2}
	\left\| \delta \boldsymbol{d}_i^{\left( n + \frac{j+1}{\eta_i}\right)}\right\| & \leq 
\frac{ 1 + \Delta t_i \mathcal{C}_i^{h}}{1 - \Delta t_i \mathcal{C}_i^{h}} \left\| \delta \boldsymbol{d}_i^{\left( n + \frac{j}{\eta_i}\right)}\right\| + \frac{2 \Delta t_i \mathcal{C}_i^{\lambda}}{1 - \Delta t_i \mathcal{C}_i^{h}}  \left\| \delta \boldsymbol{\lambda}^{(n)}\right\| + 
\frac{2 \Delta t_i \mathcal{C}_i^{\lambda}}{1 - \Delta t_i \mathcal{C}_i^{h}} \left\| \delta \boldsymbol{\lambda}^{(n+1)}\right\| \nonumber \\ &+ \frac{2 \Delta t_i \Delta t \mathcal{C}_i^{\lambda}}{1 - \Delta t_i \mathcal{C}_i^{h}}  \left\| \boldsymbol{\Delta}_{\lambda}\right\| + \frac{\Delta t_i}{1 - \Delta t_i \mathcal{C}_i^{h}} \left\| \boldsymbol{\varepsilon}_{d_i}\right\|
\end{align}
Applying the above inequality in a recursive manner, 
the following inequality can be obtained over a 
system time-step: 
\begin{align}
\label{Eqn:IoP_deltad}
\left\| \delta \boldsymbol{d}_i^{\left( n + 1 \right)}\right\|
	&\leq \left( \frac{1 + \Delta t_i \mathcal{C}_i^{h}}{1 - \Delta t_i \mathcal{C}_i^{h}}\right)^{\eta_i} \left\| \delta \boldsymbol{d}_i^{(n)}\right\| + 
	\left\{ \sum_{k = 0}^{\eta_i - 1} \left( \frac{1 + \Delta t_i \mathcal{C}_i^{h}}{1 - \Delta t_i \mathcal{C}_i^{h}}\right)^{k}\right\} \left( \frac{2 \Delta t_i \mathcal{C}_i^{\lambda}}{1 - \Delta t_i \mathcal{C}_i^{h}} \left\| \delta \boldsymbol{\lambda}^{(n)}\right\| \right. \nonumber \\ & \left. + \frac{2 \Delta t_i \mathcal{C}_i^{\lambda}}{1 - \Delta t_i \mathcal{C}_i^{h}} \left\| \delta \boldsymbol{\lambda}^{(n + 1)}\right\|  + \frac{2 \Delta t_i \Delta t \mathcal{C}_i^{\lambda}}{1 - \Delta t_i \mathcal{C}_i^{h}} \left\| \boldsymbol{\Delta}_{\lambda}\right\| + 
	\frac{\Delta t_i}{1 - \Delta t_i \mathcal{C}_i^{h}} \left\| \boldsymbol{\varepsilon}_{d_i}\right\| \right)
\end{align}
Similarly, one can derive the following inequality 
for the rate variables:
\begin{align}
\label{Eqn:IoP_deltav}
	\left\| \delta \boldsymbol{v}_i^{(n+1)}\right\| &\leq \mathcal{C}_i^{h} \left( \frac{1 + \Delta t_i \mathcal{C}_i^{h}}{1 - \Delta t_i \mathcal{C}_i^{h}}\right)^{\eta_i} \left\| \delta \boldsymbol{d}_i^{(n)}\right\|  \nonumber \\
	&+ \left\{ \mathcal{C}_i^{h} \left\{ \sum_{k = 0}^{\eta_i - 1}\left( \frac{1 + \Delta t_i \mathcal{C}_i^{h}}{1 - \Delta t_i \mathcal{C}_i^{h}} \right)^{k} \right\}\frac{2 \Delta t_i \mathcal{C}_i^{\lambda}}{1 - \Delta t_i \mathcal{C}_i^{h}} + \mathcal{C}_i^{\lambda} \right\} \left( \left\| \delta \boldsymbol{\lambda}^{(n)}\right\| + \left\| \delta \boldsymbol{\lambda}^{(n+1)}\right\|\right) \nonumber \\
&+ \left\{ \mathcal{C}_i^{h} \left\{ \sum_{k = 0}^{\eta_i - 1}\left( \frac{1 + \Delta t_i \mathcal{C}_i^{h}}{1 - \Delta t_i \mathcal{C}_i^{h}} \right)^{k} \right\} \frac{2 \Delta t_i \Delta t \mathcal{C}_i^{\lambda}}{1 - \Delta t_i \mathcal{C}_i^{h}} + \Delta t \mathcal{C}_i^{\lambda} \right\} \left\| \boldsymbol{\Delta}_{\lambda}\right\| \nonumber \\
&+\mathcal{C}_i^{h} \left\{ \sum_{k = 0}^{\eta_i - 1}\left( \frac{1 + \Delta t_i \mathcal{C}_i^{h}}{1 - \Delta t_i \mathcal{C}_i^{h}} \right)^{k} \right\} \frac{\Delta t_i}{1 - \Delta t_i \mathcal{C}_i^{h}} \left\| \boldsymbol{\varepsilon}_{d_i}\right\|
\end{align}
From the perturbed constraint equations, we get the 
following inequality for the $d$-continuity method: 
\begin{align}
	\left\| \boldsymbol{\varepsilon}_{\lambda}\right\| = \left\| \sum_{i = 1}^{\mathcal{S}} \boldsymbol{C}_i \boldsymbol{d}_i^{(n+1)}\right\| \leq \sum_{i = 1}^{\mathcal{S}} \left\| \boldsymbol{d}_i^{(n+1)}\right\|
\end{align}
Similarly, the following inequality can be 
derived for the coupling method based on the 
Baumgarte stabilization:
\begin{align}
	\left\| \boldsymbol{\varepsilon}_{\lambda}\right\| =
	\left\| \sum_{i = 1}^{\mathcal{S}} \boldsymbol{C}_i \left( \boldsymbol{v}_i^{(n+1)} + \frac{\alpha}{\Delta t} \boldsymbol{d}_i^{(n+1)}\right)\right\| \leq \sum_{i = 1}^{\mathcal{S}} \left( \left\| \boldsymbol{v}_i^{(n+1)}\right\| + \frac{\alpha}{\Delta t} \left\| \boldsymbol{d}_i^{(n+1)}\right\|\right)
\end{align}
By substituting inequalities \eqref{Eqn:IoP_deltad} 
and \eqref{Eqn:IoP_deltav} in the above inequalities, one 
can obtain the desired inequality for $\left\| \delta 
\boldsymbol{\lambda}^{(n+1)}\right\|$. By substituting 
the resulting inequality in \eqref{Eqn:IoP_deltad} 
and \eqref{Eqn:IoP_deltav}, one can obtain the desired 
inequalities for $\left\| \delta \boldsymbol{d}_i^{(n+1)}\right\|$ 
and $\left\| \delta \boldsymbol{v}_i^{(n+1)}\right\|$.
\end{proof}

%--------------------------------------;
%  Remark: Remark about perturbations  ;
%--------------------------------------;
\begin{remark}
  The difference in the order of the perturbation in
  the algebraic constraints in \eqref{Eqn:IoP_perturbed}
  arises due to the difference in the differential index 
  of the governing DAEs. That is, the $d$-continuity 
  method form a system of DAEs of index 2, whereas the 
  coupling method based on the Baumgarte stabilization 
  form a system of DAEs of index 1. One can also decide 
  on the order of perturbations based on dimensional 
  analysis and consistency. 
\end{remark}

%************************************************;
%                                                ;
%  NAME                                          ;
%    S5_Monolithic_NR_First.tex                  ;
%                                                ;
%************************************************;
\section{BENCHMARK PROBLEMS FOR VERIFICATION}
\label{Sec:S6_First_Benchmark}
In this section, several benchmark problems are solved to 
illustrate the accuracy of the proposed coupling methods, 
to verify numerically the theoretical predictions, and 
to check the computer implementation. 

%===============================================;
%  Subsection: Split degree-of-freedom problem  ;
%===============================================;
\subsection{Split degree-of-freedom problem}
\label{Sec:S5_First_SDOF}
The governing equations of the coupled system 
that is shown in Figure \ref{Fig:SDOF_problem} 
take the following form:
%-------------------------------------;
%  Equation: Split degree-of-freedom  ;
%-------------------------------------;
\begin{subequations}
  \label{Eqn:SDOF_GE}
  \begin{align}
    \label{Eqn:SDOF_Subdomain_A}
    &m_1 \dot{\mathrm{c}}_1(t) + k_1 \mathrm{c}_1(t) = f_1(t) + \lambda(t) \\
    &m_2 \dot{\mathrm{c}}_2(t) + k_2 \mathrm{c}_2(t) = f_2(t) - \lambda(t) \\
    \label{Eqn:SDOF_Constraint}
    &\mathrm{c}_1(t)  - \mathrm{c}_2(t) = 0
  \end{align}
\end{subequations}
where $\lambda(t)$ is the Lagrange multiplier. 
The following parameters have been used 
in this numerical simulation:
%--------------------------------;
%  Equation: Problem parameters  ;
%--------------------------------;
\begin{align}
  m_1 = 100, \; m_2 = 1, \; k_1 = 1, \; 
  k_2 = 100, \; f_1 = f_2 = 0
\end{align}
We shall solve the DAEs given by equations 
\eqref{Eqn:SDOF_Subdomain_A}--\eqref{Eqn:SDOF_Constraint} 
using the proposed multi-time-step coupling 
methods, subject to the initial 
condition $\mathrm{c}_1(t = 0) = \mathrm{c}_2(t = 0) = 1$. 

%======================================;
%  Subsubsection: d-continuity method  ;
%======================================;
\subsubsection{Performance of the $d$-continuity method}
Figure \ref{Fig:SDOF_d_continuity} shows the results
of numerical solution to \eqref{Eqn:SDOF_GE} using the
proposed coupling method with $d$-continuity. Implicit 
Euler method (i.e., $\vartheta_1 = 1$) is used to integrate the first 
subdomain, and the second subdomain is integrated using 
the midpoint rule (i.e., $\vartheta_2 = 1/2$). The results are
shown for several different choices of system and subdomain
time-steps (see Table \ref{Tbl:SDOF_d_con}). As shown earlier, 
the proposed method is stable under $d$-continuity if 
$\vartheta_i \geq 1/2$ in all subdomains. Enforcing 
$d$-continuity, assures the continuity of primary 
variable (which will be the concentration in this paper) 
along the interface at all system time-levels. The 
proposed methods shows good compatibility with 
the exact solution.
\begin{table}
\caption{Split degree-of-freedom problem: Time-integration parameters for 
the $d$-continuity method. \label{Tbl:SDOF_d_con}}
\begin{tabular}{|c | c c c c c|}
\hline
Case 	& $\Delta t$ & $\Delta t_1$ & $\Delta t_2$ & $\vartheta_1$ & $\vartheta_2$ \\ \hline
1				& 0.5 		 & 0.25         & 0.5   	   & 1 		& 1/2 \\ 
2				& 0.5		 & 0.05 		& 0.1 		   & 1 		& 1/2 \\
3 				& 0.1 		 & 0.05 		& 0.1 		   & 1 		& 1/2 \\ \hline
\end{tabular}
\end{table}
%==========================================;
%  Subsubsection: Baumgarte stabilization  ;
%==========================================;
\subsubsection{Performance of the Baumgarte stabilization}
Baumgarte stabilization allows coupling explicit and 
implicit time-integrators in different subdomains. Midpoint
rule is employed in the first subdomain (i.e., $\vartheta_1 = 1/2$).
In this problem explicit Euler method is used in the second
subdomain (i.e., $\vartheta_2 = 0$). As it can be seen in 
Figure \ref{Fig:SDOF_Baumgarte}, choice of system time-step
$\Delta t$, and Baumgarte stabilization parameter $\alpha$,
influence the accuracy of the numerical result (see Table 
\ref{Tbl:SDOF_Baumgarte} for the values of integration parameters). 
The drift in the primary variables, $u_1$ and $u_2$, is nonzero.
One can observe in Figure \ref{Fig:SDOF_Baumgarte}, that
increasing the Baumgarte stabilization parameter $\alpha$,
or decreasing the system time-step $\Delta t$ can improve
the accuracy. Figure \ref{Fig:SDOF_Error} shows the absolute error
at time $t = 1$ vs. the system time-step. These 
figures show that despite subcycling (and using 
linear interpolation for Lagrange multipliers),
the convergence rate remains close to that of 
the midpoint rule (which was used in all subdomains).
\begin{table}
\caption{Split degree-of-freedom problem: Time-integration parameters for 
the Baumgarte stabilization method. \label{Tbl:SDOF_Baumgarte}}
\begin{tabular}{|c | c c c c c c|}
\hline
Case 	& $\Delta t$ & $\alpha$ & $\Delta t_1$ & $\Delta t_2$ & $\vartheta_1$ & $\vartheta_2$ \\ \hline
1				& 0.5 	& 1.0	 & 0.1      & 0.02   	   & 1/2 		& 0 \\ 
2				& 0.1	& 1.0	 & 0.1 		& 0.02 		   & 1/2 		& 0 \\
3 				& 0.5 	& 25.0	 & 0.1 		& 0.02 		   & 1/2 		& 0 \\ \hline
\end{tabular}
\end{table}

%=======================================;
%  Subsection: One-dimensional problem  ;
%=======================================;
\subsection{One-dimensional problem}
We will consider an unsteady diffusion with 
decay in one-dimension, which is an extension 
of the steady-state version considered in 
\citep{Farrell_Hemker_Shishkin_JCM_1995}. 
The governing equations can be written as 
follows: 
%------------------------------------------------;
%  Equation: Governing equations for 1D problem  ;
%------------------------------------------------;
\begin{subequations}
  \label{Eqn:1D_Perturbed_Problem}
  \begin{alignat}{2}
    &\frac{\partial \mathrm{c}}{\partial t} + \mathrm{c} 
    - \varepsilon^2 \frac{\partial^2 \mathrm{c}}{\partial \mathrm{x}^2}
    = 1 \quad && \mathrm{x} \in (0,1), \; t \in (0,T] \\
	&\mathrm{c}(\mathrm{x}=0,t) = \mathrm{c}(\mathrm{x}=1,t) = 0 
        \quad &&t \in (0,T] \\
	  &\mathrm{c}(\mathrm{x},t = 0) = 0 \quad &&\mathrm{x} \in (0,1)
  \end{alignat}
\end{subequations}
It is well-known that the solution of this singularly
perturbed problem will exhibit boundary layers for 
small values of $\varepsilon$. Herein, we have taken 
$\varepsilon = 0.01$.
We shall demonstrate the benefits of using the 
proposed multi-time-step coupling methods to 
problems in which the behavior of the solution 
can be very different in various regions of 
the computational domain.

The domain is decomposed into three subdomains, as 
shown in Figure \ref{Fig:Hemker_1D}. Subdomains 1 
and 3 are the regions in which the boundary layers 
will appear. Note that these subdomains are meshed 
using much finer elements than subdomain 2. For 
time-integration variables, see tables 
\ref{Tbl:Hemker_1D_d_con} and \ref{Tbl:Hemker_1D_Baumgarte}.
The numerical results obtained using the 
proposed multi-time-step coupling methods are shown 
in Figures \ref{Fig:1D_Hemker_d_continuity} and 
\ref{Fig:1D_Hemker_Baumgarte_Set1}. Results are 
in good agreement with the exact solution, and 
the boundary layers are captured accurately by 
the proposed coupling methods. 
The drifts in concentrations and rate variables
are plotted in figures \ref{Fig:1D_Hemker_d_continuity_drift} 
and \ref{Fig:1D_Hemker_Baumgarte_drift}.
This numerical experiment illustrates the 
following attractive features of the proposed 
coupling methods: 
\begin{enumerate}[(a)]
\item The system time-step can be much larger 
  than subdomain time-steps. 
\item For fixed subdomain time-steps, smaller 
  system time-step will result in better accuracy. 
\item Under the coupling method based on the 
  Baumgarte stabilization and fixed subdomain 
  time-steps, decreasing system time-step and/or 
  increasing the Baumgarte stabilization parameter 
  will result in improved accuracy.
\item Utilizing smaller time-steps in individual 
  subdomains improves the accuracy of results in 
  the respective subdomain.
\end{enumerate}

\begin{table}
\caption{One-dimensional problem: Time-integration parameters for the
$d$-continuity method.\label{Tbl:Hemker_1D_d_con}}
\begin{tabular}{|c | c c c c c c c|}
\hline
Case 	& $\Delta t$ & $\Delta t_1$ & $\Delta t_2$ & $\Delta t_3$ & $\vartheta_1$ & $\vartheta_2$ & $\vartheta_3$ \\ \hline
1		& 0.25 & 0.05 & 0.25 & 0.05 & 1/2 & 1 & 1/2 \\ 
2 		& 0.25 & 0.05 & 0.01 & 0.05 & 1/2 & 1 & 1/2 \\ 
3		& 0.1  & 0.1  & 0.1  & 0.1  & 1/2 & 1/2 & 1/2 \\\hline
\end{tabular}
\end{table}

\begin{table}
\caption{One-dimensional problem: Time-integration parameters for the
Baumgarte stabilization method.\label{Tbl:Hemker_1D_Baumgarte}}
\begin{tabular}{|c | c c c c c c c c|}
\hline
Case 	& $\Delta t$ & $\alpha$ &$\Delta t_1$ & $\Delta t_2$ & $\Delta t_3$ & 
$\vartheta_1$ & $\vartheta_2$ & $\vartheta_3$ \\ \hline
1	& 0.25 & 1 & 0.125   & 0.25 & 0.125   & 1/2 & 0 & 1/2 \\ 
2 	& 0.25 & 5 & 0.125   & 0.05 & 0.125   & 1/2 & 0 & 1/2 \\ 
3 	& 0.25 & 5 & 0.00125 & 0.25 & 0.00125 & 0   & 1 & 0   \\
4 	& 0.25 & 1 & 0.0025  & 0.25 & 0.0025  & 0   & 1 & 0   \\ 
5   & 0.1  & 1 & 0.1     & 0.1  & 0.1     & 1/2 & 1/2 & 1/2 \\\hline
\end{tabular}
\end{table}

%=======================================;
%  Subsection: Two-dimensional problem  ;
%=======================================;
\subsection{Two-dimensional problem}
A transient version of the well-known problem 
proposed by Hemker \citep{Hemker_JCAM_1996} 
will be considered. The governing equations take 
the following form:
%---------------------------------------------------;
%  Equation: Governing equations of Hemker problem  ;
%---------------------------------------------------;
\begin{subequations}
  \label{Eqn:Hemker_2D}
  \begin{alignat}{2}
    &\frac{\partial \mathrm{c}}{\partial t} + 
    \frac{\partial \mathrm{c}}{\partial \mathrm{x}} - \varepsilon 
    \mathrm{div} \left[ \mathrm{grad}[\mathrm{c}]\right] = 0 
    \quad &&\mathrm{in} \; \Omega\\
    &\mathrm{c}(\mathrm{x},\mathrm{y},t) = 1 \quad &&\mathrm{on} \; 
    \Gamma^{\mathrm{D}}_{1} \\
    &\mathrm{c}(\mathrm{x},\mathrm{y},t) = 0 \quad &&\mathrm{on}\; 
    \Gamma^{\mathrm{D}}_{2} \\
    -&\varepsilon \mathrm{grad} [\mathrm{c}]\cdot 
    \widehat{\mathbf{n}}(\mathbf{x}) = 0 
    \quad &&\mathrm{on} \; \Gamma^{\mathrm{N}} \\
    &\mathrm{c}(\mathrm{x},\mathrm{y},t = 0) = 0 \quad 
    &&\mathrm{in} \; \Omega
  \end{alignat}
\end{subequations}
Computational domain, 
mesh, and domain decomposition are shown in Figures 
\ref{Fig:Hemker_2D} and \ref{Fig:2D_Hemker_mesh}. 
In this problem, the advection velocity is $\mathbf{v} = 
(1,0)$, and $\varepsilon = 0.01$.
The problem at hand is a singularly perturbed
equation and is known to exhibit both boundary 
and interior layers. Furthermore, the standard 
Galerkin formulation is known to produce spurious 
oscillations for small values of $\varepsilon$ 
\citep{Gresho_Sani_v1}.

The numerical results obtained using the Galerkin 
weak formulation are shown in Figure 
\ref{Fig:2D_Hemker_Galerkin}. As expected, spurious 
oscillations occur at the vicinity of the circle. 
The minimum value observed in both cases is -0.439. 
The spurious oscillations and the violation of the non-negative 
constraint is because of using the Galerkin weak formulation, 
and is \emph{not} due to the use of proposed multi-time-step 
coupling methods. To corroborate this claim, Figure 
\ref{Fig:2D_Hemker_GLS_SUPG_Galerkin} shows the results 
where tailored weak formulations are employed in different 
subdomains. The GLS formulation is used in subdomain 1, the 
SUPG formulation is employed in subdomain 2, and the Galerkin 
formulation in subdomain 3. There are no spurious oscillations 
and the minimum value observed is -0.062. 
In Figure \ref{Fig:2D_Hemker_drift}, the $\infty$-norm of
drift of concentrations from compatibility constraints
is shown. There is no noticeable drift and in the 
case of Baumgarte stabilization method, the the 
drifts are controlled.
Time integration
parameters are given in tables \ref{Tbl:2D_Hemker_1} and 
\ref{Tbl:2D_Hemker_2}.
This example demonstrates choice of disparate time-steps,
and different numerical formulations in different
spatial regions of the computational domain. 

\begin{table}
\centering
\caption{Two-dimensional transient Hemker problem:~Time-integration
parameters for results using the standard Galerkin method. 
\label{Tbl:2D_Hemker_1}}
\begin{tabular}{| c | c c c c c c c c |}
\hline
Compatibility condition & $\Delta t$ & $\alpha$ & $\Delta t_1$ & $\Delta t_2$ & $\Delta t_3$ & $\vartheta_1$ & $\vartheta_2$ & $\vartheta_3$ \\ \hline
$d$-continuity method & 0.1 & & 0.001 & 0.01 & 0.1 & 1/2 & 1 & 1 \\
Baumgarte stabilization & 0.2 & 1 & 0.01 & 0.05 & 0.02 & 1/2 & 1 & 0 \\ \hline
\end{tabular}
\end{table}

\begin{table}
\centering
\caption{Two-dimensional transient Hemker problem:~Time-integration
parameters for results using the GLS-SUPG-Galerkin formulations. 
\label{Tbl:2D_Hemker_2}}
\begin{tabular}{| c | c c c c c c c c |}
\hline
Compatibility condition & $\Delta t$ & $\alpha$ & $\Delta t_1$ & $\Delta t_2$ & $\Delta t_3$ & $\vartheta_1$ & $\vartheta_2$ & $\vartheta_3$ \\ \hline
$d$-continuity method & 0.2 & & 0.001 & 0.005 & 0.2 & 1/2 & 1 & 1 \\
Baumgarte stabilization & 0.2 & 1 & 0.001 & 0.005 & 0.02 & 1 & 1/2 & 0 \\ \hline
\end{tabular}
\end{table}

%************************************************;
%                                                ;
%  NAME                                          ;
%    S8_First_Fast.tex                           ;
%                                                ;
%************************************************;
\section{MULTI-TIME-STEP TRANSIENT ANALYSIS OF 
  A TRANSPORT-CONTROLLED BIMOLECULAR REACTION}
\label{Sec:First_Fast}
In this section, we shall apply the proposed multi-time-step 
coupling methods to a transport-controlled bimolecular 
reaction. This problem is of tremendous practical 
importance in areas such as transverse mixing-limited 
chemical reactions in groundwater and aquifers, and 
mixing-controlled bioreactive transport in heterogeneous 
porous media arising in bioremediation. We shall now 
document the most important equations of the mathematical 
model. A more detailed discussion about the model can be 
found in \citep{Nakshatrala_Mudunuru_Valocchi_JCP_2013}, 
which however did not address multi-time-step coupling methods. 

%==================================;
%  Subsection: Mathematical model  ;
%==================================;
\subsection{Mathematical model}
Consider the following irreversible chemical reaction:
%----------------------------------;
%  Equation: Bimolecular reaction  ;
%----------------------------------;
\begin{align}
  \label{Eqn:Bimolecular_reaction}
  n_A A + n_B B \rightarrow n_C C 
\end{align}
where $A$, $B$ and $C$ are the chemical species 
participating in the reaction, and $n_A$, $n_B$ and 
$n_C$ are their respective (positive) stoichiometry 
coefficients. The fate of the reactants and the 
product are governed by coupled system of transient 
advection-diffusion-reaction equations.
We shall assume the part of the boundary on which 
the Dirichlet boundary condition is enforced to 
be the same for the reactants and the product. 
Likewise is assumed for the Neumann boundary 
conditions. 
One can then find two invariants that are 
unaffected by the underlying reaction, which 
can be obtained via the following linear 
transformation:
%------------------------;
%  Equation: Invariants  ;
%------------------------;
\begin{subequations}
  \label{Eqn:Bimolecular_transformation}
  \begin{align}
    \mathrm{c}_F := \mathrm{c}_A + \left( \frac{n_A}{n_C} \right) \mathrm{c}_C \\
    \mathrm{c}_G := \mathrm{c}_B + \left( \frac{n_B}{n_C} \right) \mathrm{c}_C
  \end{align}
\end{subequations}
The evolution of these invariants is given by the 
following uncoupled transient advection-diffusion 
equations: 
%----------------------------------------;
%  Equation: AD equation for invariants  ;
%----------------------------------------;
\begin{subequations}
  \label{Eqn:Bimolecular_reaction_PDE_uncoupled}
  \begin{alignat}{2}
    \label{Eqn:Bimolecular_uncoupled_PDE}
    &\frac{\partial \mathrm{c}_i}{\partial t} + \mathrm{div} 
    \left[\mathbf{v} \mathrm{c}_i - \mathbf{D}(\mathbf{x}) 
      \mathrm{grad}[\mathrm{c}_i]\right] = f_i(\mathbf{x},t) 
    \quad &&\mathrm{in} \; \Omega \times \mathcal{I} \\
    &\mathrm{c}_i(\mathbf{x},t) = \mathrm{c}_i^{\mathrm{p}}(\mathbf{x},t) 
    := \mathrm{c}_j^{\mathrm{p}}(\mathbf{x},t) + \left( \frac{n_j}{n_C}\right) 
    \mathrm{c}_C^{\mathrm{p}}(\mathbf{x},t) \quad && \mathrm{on} \; 
    \Gamma^{\mathrm{D}} \times \mathcal{I} \\
    -&\widehat{\mathbf{n}}(\mathbf{x}) \cdot \mathbf{D}(\mathbf{x}) 
    \mathrm{grad}[\mathrm{c}_i] = h_{i}^{\mathrm{p}}(\mathbf{x},t)
    := h_j^{\mathrm{p}}(\mathbf{x},t) + \left( \frac{n_j}{n_C}\right)
    h_{C}^{\mathrm{p}}(\mathbf{x},t) \quad &&\mathrm{on} \; 
    \Gamma^{\mathrm{N}} \times \mathcal{I} \\
    \label{Eqn:Bimolecular_uncoupled_initial}
    &\mathrm{c}_i(\mathbf{x},t = 0) = \mathrm{c}_i^{0}(\mathbf{x}):=\mathrm{c}_j^{0}(\mathbf{x}) 
    + \left( \frac{n_j}{n_C}\right) \mathrm{c}_C^{0}(\mathbf{x}) \quad 
    && \mathrm{in} \; \Omega
  \end{alignat}
\end{subequations}
where $i = F$ or $G$.
We shall restrict to fast bimolecular reactions. 
That is, the time-scale of the chemical reaction 
is much smaller than the time-scale of the transport 
processes. For such situations, one can assume that 
the chemical species $A$ and $B$ cannot coexist at 
a spatial point and for a given instance of time. 
This implies that the concentrations of the reactants 
and the product can be obtained from the concentrations 
of the invariants through algebraic manipulations. To wit, 
\begin{subequations}
  \label{Eqn:Bimolecular_max_relations}
  \begin{align}
  \label{Eqn:Bimolecular_max_A}
    &\mathrm{c}_A \left( \mathbf{x}, t\right) = \mathrm{max} \left\{ \mathrm{c}_F \left( \mathbf{x},t\right) - \left( \frac{n_A}{n_B}\right) \mathrm{c}_G \left( \mathbf{x},t\right), 0\right\} \\
	&\mathrm{c}_B \left( \mathbf{x}, t\right) = \left( \frac{n_B}{n_A}\right) \mathrm{max} \left\{ -\mathrm{c}_F \left( \mathbf{x},t\right) + \left( \frac{n_A}{n_B}\right) \mathrm{c}_G \left( \mathbf{x},t\right), 0\right\} \\
  \label{Eqn:Bimolecular_max_C}
	&\mathrm{c}_C \left( \mathbf{x},t\right) = \left( \frac{n_C}{n_A} \right) \left( \mathrm{c}_F \left( \mathbf{x},t\right) - \mathrm{c}_A \left( \mathbf{x},t\right)\right)
\end{align}
\end{subequations}
Note that the solution procedure is still nonlinear, 
as the $\mathrm{max}\{\cdot,\cdot\}$ operator is 
nonlinear. 

We shall employ the proposed multi-time-step 
computational framework to solve equations 
\eqref{Eqn:Bimolecular_uncoupled_PDE}--\eqref{Eqn:Bimolecular_uncoupled_initial} to obtain concentrations of the invariants. 
Using the calculated values, we then find 
the concentrations for the reactants and 
the product using equations 
\eqref{Eqn:Bimolecular_max_A}--\eqref{Eqn:Bimolecular_max_C}. 
The Galerkin formulation is employed in all 
subdomains. The negative values for the 
concentration are clipped at every subdomain 
time-step in the numerical simulations.

%=============================================;
%  Subsection: Bimolecular without advection  ;
%=============================================;
\subsection{Numerical results for a diffusion-controlled 
bimolecular reaction}
Consider a reaction chamber with $L_x = L_y = 1$, 
as shown in Figure \ref{Fig:Bimolecular_Baumgarte}. 
The computational domain is meshed using 5442 four-node
quadrilateral elements. As shown in this figure, the domain 
is decomposed into four non-contiguous subdomains using 
METIS \citep{METIS_paper}. The diffusivity tensor is taken as 
follows:
%------------------------------;
%  Equation: Diffusion tensor  ;
%------------------------------;
\begin{align}
  \mathbf{D} \left( \mathrm{x}, \mathrm{y}\right) = 
  \left[ \begin{array}{c c}
      \gamma \mathrm{x}^2 + \mathrm{y}^2 & - \left( 1 - \gamma \right) 
      \mathrm{x}\mathrm{y} \\
      - \left( 1 - \gamma \right) \mathrm{x}\mathrm{y} & \mathrm{x}^2 + 
      \gamma \mathrm{y}^2
    \end{array} \right]
\end{align}
where $\gamma = 0.001$.
Baumgarte stabilization is employed to enforce 
compatibility along the subdomain interfaces with 
$\alpha = 100$. Implicit Euler 
method is employed in subdomains 1 and 3, 
and midpoint rule is employed in subdomains 
2 and 4. The system time-step is taken as 
$\Delta t = 10^{-3}$, and the subdomain time-steps 
are taken as $\Delta t_1 = \Delta t_3 = 5 \times 
10^{-4}$, and $\Delta t_2 = \Delta t_4 = 10^{-3}$. 

Numerical results for the concentrations of the invariants, 
reactants and product are shown in Figures 
\ref{Fig:Bimolecular_Baumgarte_set1_FandG} and 
\ref{Fig:Bimolecular_Baumgarte_set1}.
As discussed earlier, Baumgarte stabilized coupling method 
can result in drift in the primary variable but it can be 
controlled using the stabilization parameter $\alpha$. 
Equation \eqref{Eqn:S4_Drift_Measurement} can serve as 
a valuable tool assessing the overall behavior of drifts 
with respect to system time-step, and the Baumgarte 
stabilization parameter $\alpha$. Drifts for several 
choices of $\alpha$ and $\Delta t$ are shown in Figure 
\ref{Fig:Bimolecular_Baumgarte_Drift}. Note that
equation \eqref{Eqn:S4_Drift_Measurement} assumes 
no subcycling, and no mixed time-integration.

%==========================================;
%  Subsection: Bimolecular with advection  ;
%==========================================;
\subsection{Numerical results for a fast 
 bimolecular reaction with advection}
Consider a reaction chamber with $L_x = 4$ 
and $L_y = 1$, as shown in Figure 
\ref{Fig:Bimolecular_Problem_with_advection_a}. 
The computational domain is meshed using three-node 
triangular elements, and METIS \citep{METIS_paper} is 
employed to decomposed the domain into four non-contiguous 
subdomains using , as shown in Figure 
\ref{Fig:Bimolecular_Problem_with_advection_b}. 
There are 4151 interface constraints to ensure 
continuity of concentration along the subdomain 
interface. 
The diffusivity tensor is taken as follows:
%--------------------------------;
%  Equation: Diffusivity tensor  ;
%--------------------------------;
\begin{align}
  \mathbf{D}(\mathbf{x}) = \alpha_T \| \mathbf{v} \| \mathbf{I} +
  \frac{\alpha_L - \alpha_T}{\| \mathbf{v} \|} \mathbf{v} \otimes 
  \mathbf{v}
\end{align}
where $\mathbf{I}$ is the second-order identity tensor, 
$\otimes$ is the tensor product, $\|\cdot\|$ is the 2-norm, 
$\mathbf{v}(\mathbf{x})$ is the velocity, and $\alpha_L$ and 
$\alpha_T$ are, respectively, the longitudinal and transverse 
diffusivities. This form of diffusivity tensor is commonly 
employed in subsurface hydrology \citep{Pinder_Celia}. 
We define the velocity through the following stream function:
%-----------------------------;
%  Equation: Stream function  ;
%-----------------------------;
\begin{align}
  \label{Eqn:Bimolecular_Stream}
  \psi(\mathrm{x},\mathrm{y}) = - \mathrm{y} - 
  \sum_{k = 1}^{3} A_{k} \mathrm{cos}\left( \frac{p_k \pi \mathrm{x}}{L_x} 
  - \frac{\pi}{2}\right) \mathrm{sin} \left( \frac{q_k \pi \mathrm{y}}{L_y}\right)
\end{align}
The components of the advection velocity 
can then be calculated as follows:
%---------------------------------;
%  Equation: Velocity components  ;
%---------------------------------;
\begin{align}
  \mathrm{v}_{x}(\mathrm{x},\mathrm{y}) = 
  - \frac{\partial \psi}{\partial \mathrm{y}}, 
  \quad \mathrm{v}_{y}(\mathrm{x},\mathrm{y}) 
  = + \frac{\partial \psi}{\partial \mathrm{x}}
\end{align}
The following parameters are used in the numerical simulation: 
%------------------------;
%  Equation: Parameters  ;
%------------------------;
\begin{align}
  p_1 = 4, \; p_2 = 5, \; p_3 = 10, \; 
  q_1 = 1, \; q_2 = 5, \; q_3 = 10, \; 
  A_1 = 0.08, \; A_2 = 0.02, \; A_3 = 0.01
\end{align}
The diffusivities are taken as $\alpha_L = 1$ and 
$\alpha_T = 10^{-4}$, and the prescribed concentrations 
on the boundary are taken as $c_A^{\mathrm{p}} = 1.0$ and 
$c_B^{\mathrm{p}} = 1.5$.

The numerical results for the concentration of the 
product at various time levels obtained using the 
$d$-continuity coupling method are shown in Figure 
\ref{Fig:Bimolecular_d_Continuity}, and there is no 
drift along the subdomain interface, which is expected 
under the proposed $d$-continuity method. 

The above numerical examples clearly demonstrate that the 
proposed multi-time-step coupling methods can handle 
any decomposition of the computational domain: either 
the subdomains are contiguous or non-contiguous; whether 
the decomposition is based on the physics of the problem 
or based on numerical performance; or whether the decomposition 
is done manually by the user or obtained from a graph-partitioning 
software package.

%************************************************;
%                                                ;
%  NAME                                          ;
%    S9_First_CR.tex                             ;
%                                                ;
%************************************************;
\section{CONCLUDING REMARKS}
\label{Sec:First_CR}
We presented a stable multi-time-step computational 
framework for transient advective-diffusive-reactive 
systems. The computational domain can be divided 
into an arbitrary number of subdomains. Different 
time-stepping schemes under the trapezoidal family 
can be used in different subdomains. Different time-steps 
and different numerical formulations can be employed in 
different subdomains. Unlike many of the prior works on 
multi-time-step methods (e.g., staggered schemes proposed 
in \citep{Piperno_IJNMF_1997_v25_p1207}), no preferential 
treatment is given to the subdomain with the largest subdomain 
time-step, and thereby eliminating the associated 
subdomain-dependent solutions. 

Under the framework, we proposed two different 
monolithic coupling methods, which differ in the way 
compatibility conditions are enforced along the 
subdomain interface. 
Under the first method (i.e., $d$-continuity method), 
the continuity of the primary variable is enforced 
along the subdomain interface at every system time-step. 
An attractive feature of the $d$-continuity method is 
that, by construction, there is no drift in the primary 
variable along the subdomain interface. However, one 
cannot couple explicit and implicit schemes under the 
$d$-continuity method. But this method has good stability 
characteristics. 
The second method is based on an extension of the classical 
Baumgarte stabilization \citep{Baumgarte_CMAME_1972_v1_p1,
Nakshatrala_Prakash_Hjelmstad_JCP_2009_v228_p7957} to 
first-order transient systems. Under this method one can couple 
explicit and implicit schemes. However, there can be drift in the 
primary variable along the subdomain interface. But this drift is 
bounded and small, which we have shown both theoretically 
and numerically. 
The other salient features of the proposed coupling 
methods are as follows: There is no limitation on the 
number of subdomains or on the subcycling ratios 
$\eta_i$. Since no preference is given to any 
subdomain, the numerical solutions under the 
proposed coupling methods will not be affected 
by the way the computational domain is decomposed 
into subdomains. This is also evident from the numerical 
results presented in this paper. The coupling methods 
are shown to be stable, which has been illustrated both 
mathematically and numerically. 

Based on the above discussion, we shall make the 
following two recommendations for a multi-time-step 
analysis of first-order transient systems:
\begin{enumerate}[(i)]
\item If it is not needed to couple explicit/implicit 
	time integrators,
  but one just wants to use different time-steps and different 
  numerical formulations in different regions, then it is 
  recommended to use the proposed $d$-continuity method. 
  If one wants to couple explicit and implicit schemes, 
  then one has to use the proposed coupling method based on 
  Baumgarte stabilization. 
\item Accuracy can be improved by decreasing the system time-step. 
\end{enumerate}
A possible research work can be towards the implementation 
of the proposed multi-time-step coupling methods in a parallel 
computing environment and on graphical processing units (GPUs).

%************************************************;
%                                                ;
%  NAME                                          ;
%    Proposed_coupling.tex                       ;
%                                                ;
%  WRITTEN BY                                    ;
%    Saeed Karimi                                ;
%    Kalyana Babu Nakshatrala                    ;
%                                                ;
%************************************************;
\section*{APPENDIX: A COMPACT NOTATION FOR COMPUTER IMPLEMENTATION}
\label{Sec:Appendix}
Herein, we present a compact matrix form for the 
proposed coupling methods. In fact, we present 
in a more general setting by considering nonlinear 
first-order transient DAEs of the following form:
\begin{align}
\label{Eqn:Appendix_Gov}
	&\boldsymbol{M}_i \dot{\boldsymbol{c}}_i 
	\left( t \right) =
	\boldsymbol{h}_i \left( \boldsymbol{c}_i \left( t\right),
	t\right) + \boldsymbol{C}_i^{\mathrm{T}} \boldsymbol{\lambda}\left( t \right) \quad \forall i \\
	&\sum_{i = 1}^{\mathcal{S}} \boldsymbol{C}_i \boldsymbol{c}_i\left( t \right) = \boldsymbol{0}
\end{align}
We will employ a Newton-Raphson-based approach to 
solve the given system of equation.
Other methods of solving nonlinear equations (e.g., 
Picard method) can also be utilized. However for 
simplicity of the presentation, we shall ignore 
further details on solution techniques for solving 
nonlinear equations. 

The resulting system of equations will have 
to be solved in iterations until some suitable 
convergence condition is met. Let us denote 
the differentiation operator with respect to 
$\boldsymbol{u}$ by $\mathrm{D}$ (i.e., $\mathrm{D}:= \partial / \partial\boldsymbol{u}$). Let $\Box^{\left( n + (j+1)/\eta_i\right)}_{i,\nu}$ 
denote the nodal values of $\Box$ in the $i$-th 
subdomain, at time-level $n + \frac{j+1}{\eta_i}$, 
and after $\nu$ iterations.
The following notation will be useful: 
\begin{align}
	\label{Eqn:Notation_D}
  &\boldsymbol{h}_{i,\nu}^{\left( n + \frac{j+1}{\eta_i}\right)} = 
  \boldsymbol{h}_i \left( \boldsymbol{d}_{i,\nu}^{\left( n + \frac{j+1}{\eta_i}\right)}, t^{\left( n + \frac{j+1}{\eta_i}\right)}\right) \\
  &\mathrm{D} \boldsymbol{h}_{i,\nu}^{\left( n + \frac{j+1}{\eta_i}\right)}= 
  \mathrm{D} \boldsymbol{h}_i \left( \boldsymbol{d}_{i,\nu}^{\left( n + \frac{j+1}{\eta_i}\right)}, t^{\left( n + \frac{j+1}{\eta_i}\right)}\right)
\end{align} 
The unknowns at all subdomain time-levels for the 
$i$-th subdomain and for a given Newton-Raphson 
iteration number $\nu$ can be grouped as follows:
\begin{align}
  \label{Eqn:S3_Notation_X}
  \boldsymbol{X}_{i,\nu+1}^{\left(n + \frac{j}{\eta_i}\right)} = \left[ 
    \begin{array} {c}
      \boldsymbol{v}_{i,\nu+1}^{\left( n + \frac{j}{\eta_i}\right)} \\
      \boldsymbol{d}_{i,\nu+1}^{\left( n + \frac{j}{\eta_i}\right)}
    \end{array} \right] ,\quad
  \mathbb{X}_{i,\nu+1}^{\left(n + 1\right)} = \left[ 
    \begin{array} {c}
      \boldsymbol{X}_{i,\nu+1}^{\left(n + \frac{1}{\eta_i} \right)} \\
      \boldsymbol{X}_{i,\nu+1}^{\left(n + \frac{2}{\eta_i} \right)} \\
      \vdots \\
      \boldsymbol{X}_{i,\nu+1}^{\left(n + 1\right)}
    \end{array} \right], \quad
  \mathbb{X}_{\nu+1}^{\left(n + 1\right)} = \left[ 
    \begin{array} {c}
      \mathbb{X}_{1,\nu+1}^{\left(n + 1 \right)} \\
      \mathbb{X}_{2,\nu+1}^{\left(n + 1 \right)} \\
      \vdots \\
      \mathbb{X}_{\mathcal{S},\nu+1}^{\left(n + 1\right)}
    \end{array} \right]
\end{align}
Using equation \eqref{Eqn:Notation_D} and the trapezoidal 
time-stepping schemes, the following linearized matrices 
can be constructed: 
\begin{align}
\label{Eqn:S3_Left_Right_C}
 \mathbb{L}_{i,\nu}^{\left( n + \frac{j+1}{\eta_i}\right)} = \left[ \begin{array} {c c}
 		\boldsymbol{M}_i & -\mathrm{D} \boldsymbol{h}_{i,\nu}^{\left( n + \frac{j+1}{\eta_i}\right)} \\
		-\vartheta_i \Delta t_i \boldsymbol{I}_i & \boldsymbol{I}_i 
 \end{array} \right], \quad
 \mathbb{R}_i = \left[ \begin{array} {c c}
 		\mathcal{O}_i & \mathcal{O}_i \\
		\left( 1 -\vartheta_i \right) \Delta t_i \boldsymbol{I}_i & \boldsymbol{I}_i 
 \end{array} \right], \quad
 \mathbb{C}_i = \left[ \boldsymbol{C}_i \quad \mathrm{O}_i \right]
\end{align}
where $\mathcal{O}_i$ and $\boldsymbol{I}_i$ are, 
respectively, zero matrix and identity matrix of 
size $N_i \times N_i$, and the matrix $\mathrm{O}_i$ 
is a zero matrix of size $N_{\lambda} \times N_i$.
The forcing function and the results from the 
previous iteration can be compactly assumed into 
the following vector:
\begin{align}
	\boldsymbol{F}_{i,\nu}^{\left( n + \frac{j+1}{\eta_i}\right)} = 
	\left[ \begin{array}{c}
		\boldsymbol{h}_{i,\nu}^{\left( n + \frac{j+1}{\eta_i}\right)} - 
		\mathrm{D}\boldsymbol{h}_{i,\nu}^{\left( n + \frac{j+1}{\eta_i}\right)} \boldsymbol{d}_{i,\nu}^{\left( n + \frac{j+1}{\eta_i}\right)} \\
		\boldsymbol{0}
	\end{array} \right]
\end{align}
Now, let the square matrices $\mathbb{A}_{i,\nu}$ and
$\mathbb{A}_{\nu}$, and column vectors $\mathbb{F}_{\nu}^{\left( n + 1 \right)}$ and $\mathbb{F}_{\nu}^{\left( 
n + 1 \right)}$ be defined as below:
\begin{align}
\label{Eqn:App_temp_1}
	&\mathbb{A}_{i,\nu}^{(n+1)} = \underbrace{\left[ \begin{array}{c c c c}
	\mathbb{L}_{i, \nu}^{\left( n + \frac{1}{\eta_i}\right)} & & & \\
	-\mathbb{R}_i & \mathbb{L}_{i, \nu}^{\left( n + \frac{2}{\eta_i}\right)} & & \\
	& \ddots & \ddots & \\
	& & -\mathbb{R}_i & \mathbb{L}_{i, \nu}^{\left( n + 1\right)}
	\end{array} \right]}_{2\eta_i N_i}, \quad
	\mathbb{A}_{\nu}^{(n+1)} = \left[ \begin{array}{c c c c}
	\mathbb{A}_{1, \nu}^{(n+1)} & & & \\
	& \mathbb{A}_{2 , \nu}^{(n+1)} & & \\
	& & \ddots & \\
	& & & \mathbb{A}_{\mathcal{S}, \nu}^{(n+1)}
	\end{array} \right] \\
	&\mathbb{F}_{i, \nu}^{\left( n + 1 \right)} = \left[ \begin{array}{c}
	\boldsymbol{F}_{i,\nu}^{\left( n + \frac{1}{\eta_i} \right)} + \mathbb{R}_i  \boldsymbol{X}_i^{\left( n \right)} + \mathbb{C}_i^{\mathrm{T}} \boldsymbol{\lambda}^{\left( n \right)} \\
	\boldsymbol{F}_{i,\nu}^{\left( n + \frac{2}{\eta_i}\right)} + \mathbb{C}_i^{\mathrm{T}} \boldsymbol{\lambda}^{\left( n \right)} \\
	\vdots \\
	\boldsymbol{F}_{i,\nu}^{\left( n + 1 \right)} + \mathbb{C}_i^{\mathrm{T}} \boldsymbol{\lambda}^{\left( n \right)} 
	\end{array} \right], \quad
	\mathbb{F}_{\nu}^{\left( n + 1 \right)} = \left[ \begin{array}{c}
	\mathbb{F}_{1, \nu}^{\left( n + 1 \right)} \\
	\mathbb{F}_{2, \nu}^{\left( n + 1 \right)} \\
	\vdots \\
	\mathbb{F}_{\mathcal{S}, \nu}^{\left( n + 1 \right)}
	\end{array} \right]
\end{align}
Enforcing the algebraic constraints can 
be done using the following matrices:
\begin{align}
  \mathbb{B}_i &= \underbrace{\left[ \mathrm{O}_i 
      \quad \mathrm{O}_i \quad \hdots \quad 
      \overbrace{\mathrm{O}_i \quad 
        \boldsymbol{C}_i}^{2N_i}\right]}_{2\eta_i N_{i}}  
  \quad \mbox{for $d$-continuity method} \\
  \mathbb{B}_i &= \underbrace{\left[ \mathrm{O}_i 
      \quad \mathrm{O}_i \quad \hdots \quad 
      \overbrace{\boldsymbol{C}_i \quad
        \frac{\alpha}{\Delta t}\boldsymbol{C}_i}^{2N_i}\right]}_{2\eta_i N_{i}} 
  \quad \mbox{for Baumgarte stabilization method}
\end{align}
Using the notation above, one can then construct the 
following augmented matrix:
\begin{align}
  &\mathbb{B} = \left[\mathbb{B}_1 \quad 
    \mathbb{B}_2 \quad \hdots \quad 
    \mathbb{B}_\mathcal{S} \right]
\end{align}
The algebraic constraint (for both $d$-continuity 
and Baumgarte stabilization) can be compactly 
written as follows:
\begin{align}
	\mathbb{B} \mathbb{X}_{\nu+1}^{\left( n + 1\right)} = \boldsymbol{0}
\end{align}
We define the matrix $\mathbb{C}$ as below:
\begin{align}
	\mathbb{C} = \underbrace{\left[ \quad -\frac{1}{\eta_1} 
	\mathbb{C}_1 \quad 
	-\frac{2}{\eta_1} \mathbb{C}_1 \quad \hdots \quad -\mathbb{C}_1\quad | \quad \hdots \quad | \quad -\frac{1}{\eta_\mathcal{S}} \mathbb{C}_\mathcal{S} \quad 
	-\frac{2}{\eta_\mathcal{\mathcal{S}}} \mathbb{C}_\mathcal{S} \quad \hdots \quad -\mathbb{C}_\mathcal{S} \quad \right]}_{\sum_{i = 1}^{\mathcal{S}} 2\eta_i N_i}
\end{align}
Finally, time marching can be performed by solving the
following equation:
\begin{align}
\label{Eqn:Monolithic_Final}
	\left[ \begin{array}{c | c} 
	\mathbb{A}_{\nu}^{(n+1)} & \mathbb{C}^{\mathrm{T}} \\ \hline
	\mathbb{B} & \mathbb{O}
	\end{array} \right] \left[ \begin{array}{c} 
	\mathbb{X}_{\nu+1}^{\left( n + 1 \right)} \\ \hline
	\boldsymbol{\lambda}_{\nu+1}^{\left( n + 1\right)} - \boldsymbol{\lambda}^{\left( n \right)}
	\end{array} \right] = \left[ \begin{array}{c}
	\mathbb{F}_{\nu}^{\left( n + 1\right)} \\ \hline
	\boldsymbol{0}
	\end{array} \right]
\end{align}
which gives the values of the nodal concentrations 
and the corresponding rates within a Newton-Raphson 
iteration for all subdomains and at all subdomain 
time-levels within a system time-step. The solution
procedure is outlined in Algorithm \ref{Alg:Algorithm_Appendix}.

\begin{algorithm}[ht]
\caption{Multi-time-step transient analysis using the proposed coupling methods.\label{Alg:Algorithm_Appendix}}
\begin{algorithmic}[1]
	\STATE Generate subdomain specific matrices and vectors $\boldsymbol{M}_i$, 
	$\boldsymbol{C}_i$, and $\boldsymbol{h}_i$
	\STATE Read time integration parameters: $\Delta t$, $\Delta t_i$ ($i = 1, \cdots, \mathcal{S}$), and $\alpha$ (in the case of Baumgarte stabilization)
	\ENSURE Time integration parameters should satisfy the criteria 
	given in Theorems \ref{Theorem:First_d_continuity} and \ref{Thm:S4_Stability_Baumgarte}
	\STATE Read the initial values $\boldsymbol{d}_i^{(0)}$ ($i = 1, \cdots, 
	\mathcal{S}$)
	\STATE Calculate the initial values for the Lagrange multipliers, $\boldsymbol{\lambda}^{(0)}$, by solving the following system: \\ $ \left( \sum_{i = 1}^{\mathcal{S}} \boldsymbol{C}_i \boldsymbol{M}_i^{-1} \boldsymbol{C}_i^{\mathrm{T}}\right) \boldsymbol{\lambda}^{(0)} = -   \sum_{i = 1}^{\mathcal{S}} 
	\boldsymbol{C}_i \boldsymbol{M}_i^{-1} \boldsymbol{h}_i \left( \boldsymbol{d}_i^{(0)}, t = 0\right)$
	\STATE Calculate the initial values of the rate variables: \\
	$\boldsymbol{M}_i \boldsymbol{v}_i^{(0)} = \boldsymbol{h}_i \left( \boldsymbol{d}_i^{(0)}, t = 0\right) + \boldsymbol{C}_i^{\mathrm{T}} \boldsymbol{\lambda}^{(0)} \quad \forall i = 1,...,\mathcal{S}$ 
	\STATE Specify some convergence criteria (e.g., maximum number of iterations, 
	maximum tolerance for concentrations, rate variables, or the Lagrange multipliers)	
	\FOR{$n = 1, \cdots, \mathcal{N}$}
		\STATE Initiate the iteration counter: $\nu \leftarrow 0$
		\STATE Assign values for $\boldsymbol{d}_{i,0}^{(n + j/\eta_i)}$ (for
		$i = 1, \cdots , \mathcal{S}$ and $j = 1, \cdots , \eta_i$) \COMMENT{a consistent initial guess for the nodal concentrations}
		\REPEAT
			\STATE Set $\nu \leftarrow \nu + 1$
			\STATE Initiate/update $\mathbb{A}_{\nu-1}^{(n)}$ according to equation \eqref{Eqn:App_temp_1}
			\STATE Initiate/update $\mathbb{F}_{\nu-1}^{(n)}$ according to equation \eqref{Eqn:App_temp_1}
			\STATE Solve for $\mathbb{X}_{\nu}^{(n)}$ and 
			$\boldsymbol{\lambda}_{\nu}^{(n)}$ using equation 										\eqref{Eqn:Monolithic_Final}
		\UNTIL{convergence criteria are met}
	\ENDFOR
\end{algorithmic}
\end{algorithm}

%============================;
%  Section: Acknowledgments  ;
%============================;
\section*{ACKNOWLEDGMENTS}
The authors acknowledge the support of the National Science 
Foundation under Grant no. CMMI 1068181. The opinions 
expressed in this paper are those of the authors and do 
not necessarily reflect that of the sponsors. 

%================;
%  Bibliography  ;
%================;
\bibliographystyle{plainnat}
\bibliography{Master_References,Books}

%================================;
%  Include all the figures here  ;
%================================;
%----------------------;
%  Figure: dual-Schur  ;
%----------------------;
\begin{figure}
  \psfrag{1}{$\omega_1$}
  \psfrag{2}{$\omega_2$}
  \psfrag{3}{$\omega_3$}
  \psfrag{...}{$\cdots$}
  \psfrag{S}{$\omega_{\mathcal{S}}$}
  \psfrag{omega}{$\Omega$}
  \psfrag{i}{Interface interaction}
  \psfrag{l}{$\boldsymbol{\lambda}$}
  \includegraphics[scale = 0.75]{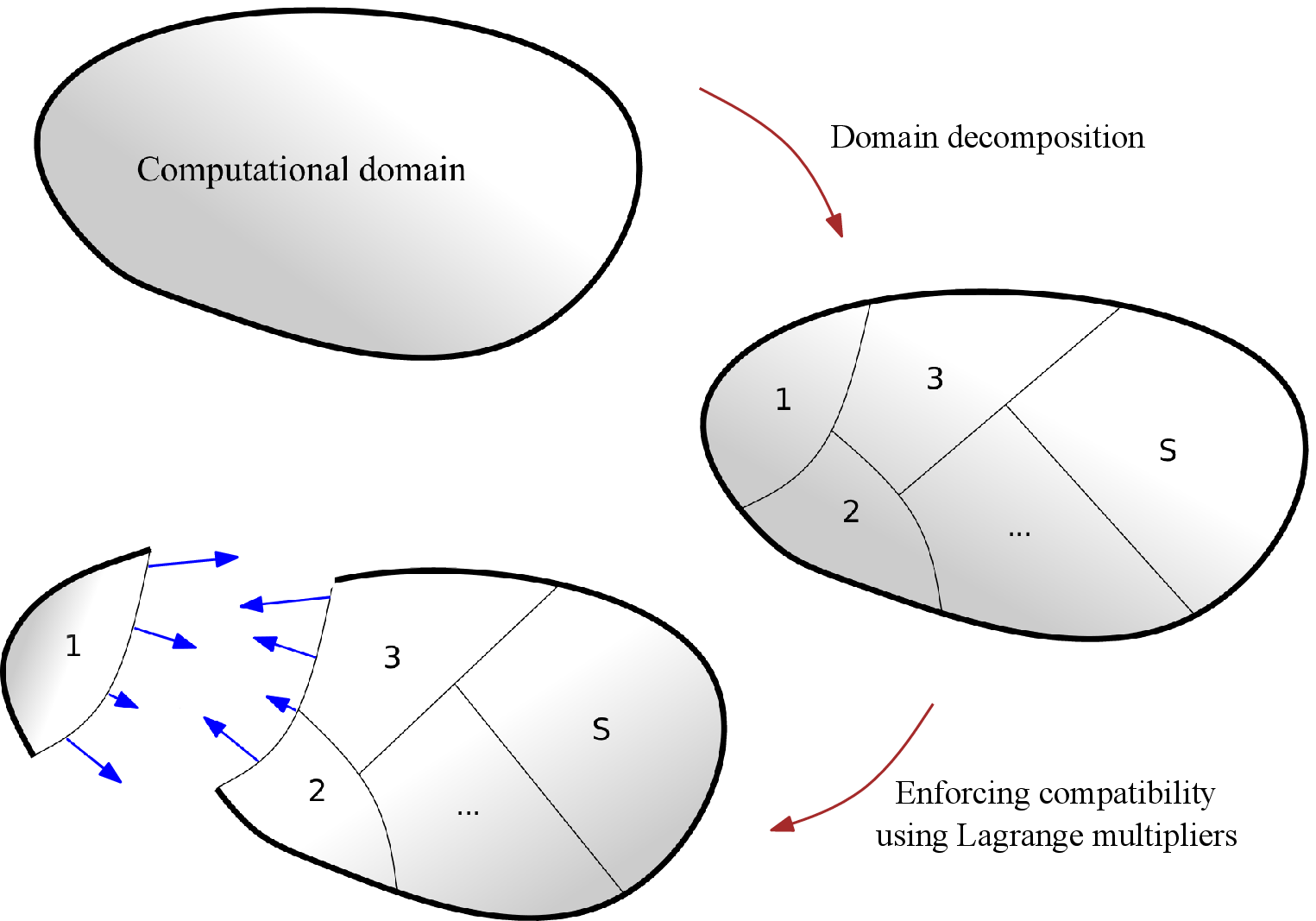}
  \caption{A pictorial description of computational 
    domain and its decomposition into subdomains, 
    subdomain interface, and interface interactions (i.e., 
    Lagrange multipliers). \label{Fig:dual_Schur}}
\end{figure}

%------------------------------------;
%  Figure: Multi-time-step notation  ;
%------------------------------------;
\begin{figure}
  \psfrag{1}{Subdomain 1}
  \psfrag{2}{Subdomain 2}
  \psfrag{S}{Subdomain $\mathcal{S}$}
  \psfrag{ex}{Information exchange}
  \psfrag{system}{System time-step $\Delta t$}
  \psfrag{tn}{$t_{n}$}
  \psfrag{tn+1}{$t_{n+1}$}
  \psfrag{dt1}{$\Delta t_1$}
  \psfrag{dt2}{$\Delta t_2$}
  \psfrag{dtS}{$\Delta t_{\mathcal{S}}$}
  \psfrag{vdot}{$\vdots$}
  \psfrag{en}{enforcement of compatibility conditions 
    at system time-levels}
  \includegraphics[scale=0.8]{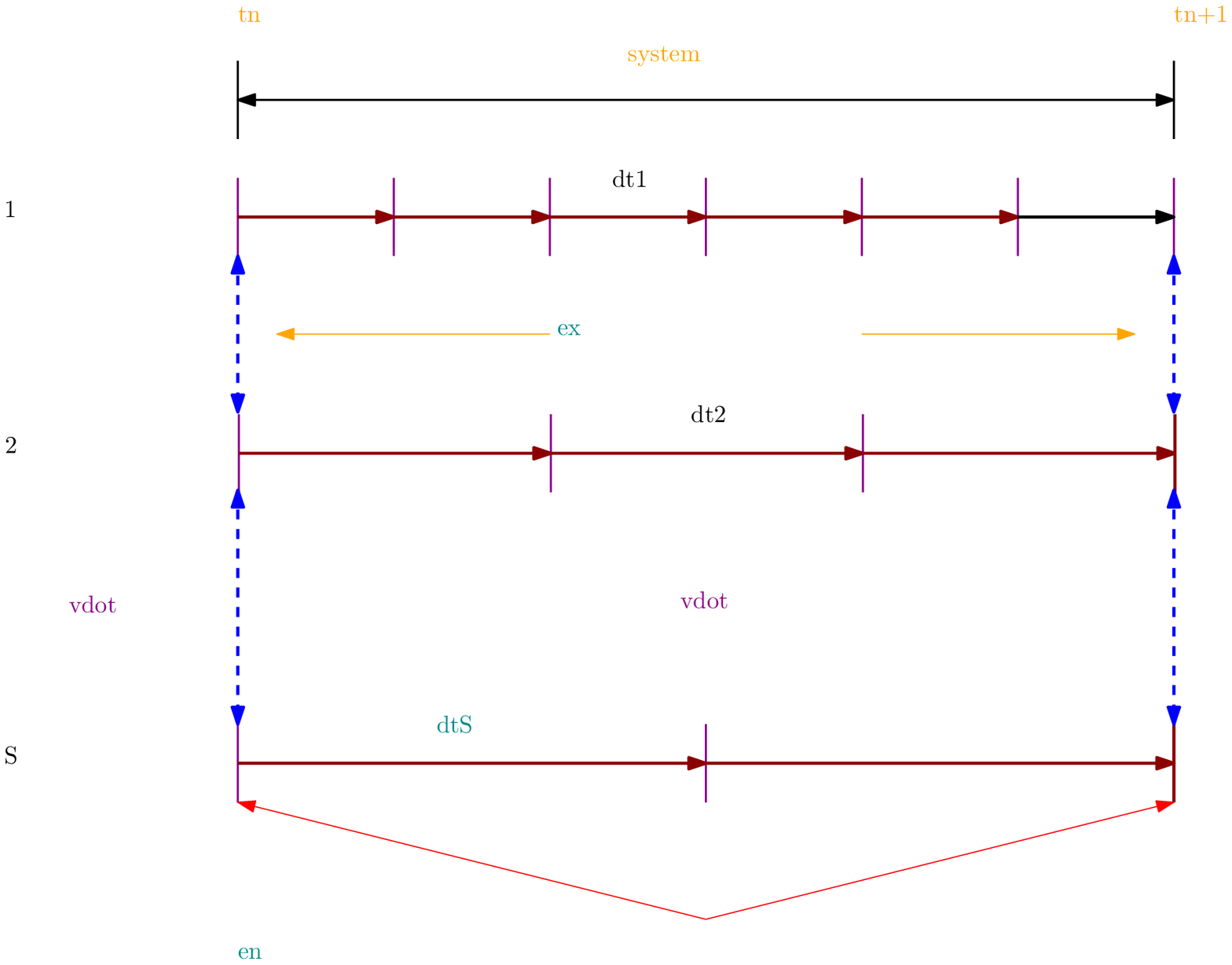}
  \caption{A pictorial description of time levels ($t_n$), 
    system time-step $(\Delta t)$, subdomain time-step 
    $(\Delta t_i)$, and subcycling. By subcycling in 
    the $i$-th subdomain we mean that $\Delta t_i < 
    \Delta t$.
    \label{Fig:Monolithic_multi_time_step_notation}}
\end{figure}

%------------------------;
%  Figure: SDOF problem  ;
%------------------------;
\begin{figure}[ht]
  \psfrag{m1}{$m_1$}
  \psfrag{m2}{$m_2$}
  \psfrag{f1}{$f_1$}
  \psfrag{f2}{$f_2$}
  \psfrag{k1}{$k_1$}
  \psfrag{k2}{$k_2$}
  \psfrag{l1}{$\lambda$}
  \psfrag{l2}{$-\lambda$}
  \includegraphics[scale = 0.8]{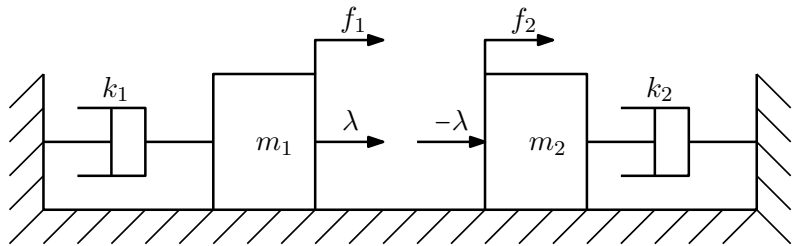}
  \caption{Split degree-of-freedom problem: A pictorial 
    description. \label{Fig:SDOF_problem}}
\end{figure}

%-----------------------;
%  SDOF : d-continuity  ;
%-----------------------;
\begin{figure}
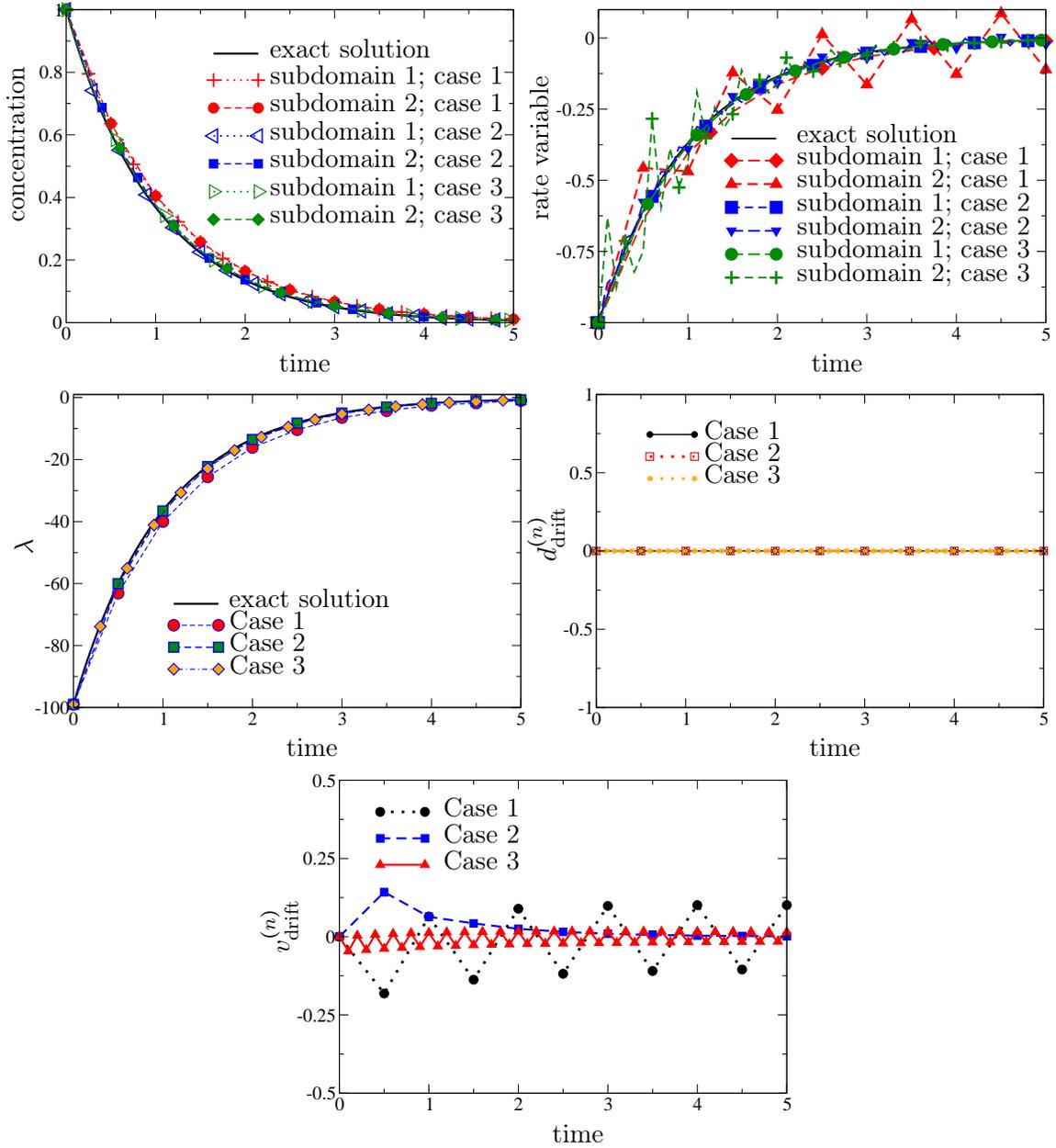

  \centering
  \psfrag{time}{time}
  \psfrag{cons}{concentration}
  \psfrag{L}{$\lambda$}
  \psfrag{drift}{$ d_{\mathrm{drift}}^{(n)}$}
  \psfrag{v_drift}{$ v_{\mathrm{drift}}^{(n)}$}
  \psfrag{rate}{rate variable}
  \psfrag{exact}{exact solution}
  \psfrag{s1_1}{subdomain 1; case 1}
  \psfrag{s1_2}{subdomain 2; case 1}
  \psfrag{s2_1}{subdomain 1; case 2}
  \psfrag{s2_2}{subdomain 2; case 2}
  \psfrag{s3_1}{subdomain 1; case 3}
  \psfrag{s3_2}{subdomain 2; case 3}
  \psfrag{s1}{Case 1}
  \psfrag{s2}{Case 2}
  \psfrag{s3}{Case 3}
  \subfigure{
    \includegraphics[scale = 0.3, clip]{SDOF_d_cons.eps}}
  \subfigure{
    \includegraphics[scale = 0.3, clip]{SDOF_d_rate.eps}}
  \subfigure{
    \includegraphics[scale = 0.3, clip]{SDOF_d_lambda.eps}}
  \subfigure{
    \includegraphics[scale = 0.3, clip]{SDOF_d_drift.eps}}
    \subfigure{
    \includegraphics[scale = 0.3, clip]{SDOF_d_v_drift.eps}}
  \caption{Split degree-of-freedom problem: 
     We have employed 
    the multi-time-step coupling method based on $d$-continuity 
    method. The values of concentrations, rate variables, 
    Lagrange multipliers, and drifts are compared with their 
    respective exact solutions. 
    It can be seen that the numerical results obtained using 
    the proposed coupling method based on $d$-continuity 
    method show very good compatibility with the exact
    solution. For values of time-integration parameters
    see Table \ref{Tbl:SDOF_d_con}. Values for the 
    time-integration parameters are given in Table
    \ref{Tbl:SDOF_d_con}. 
    \label{Fig:SDOF_d_continuity}}
\end{figure}

%--------------------------;
%  Figure: SDOF Baumgarte  ;
%--------------------------;
\begin{figure}
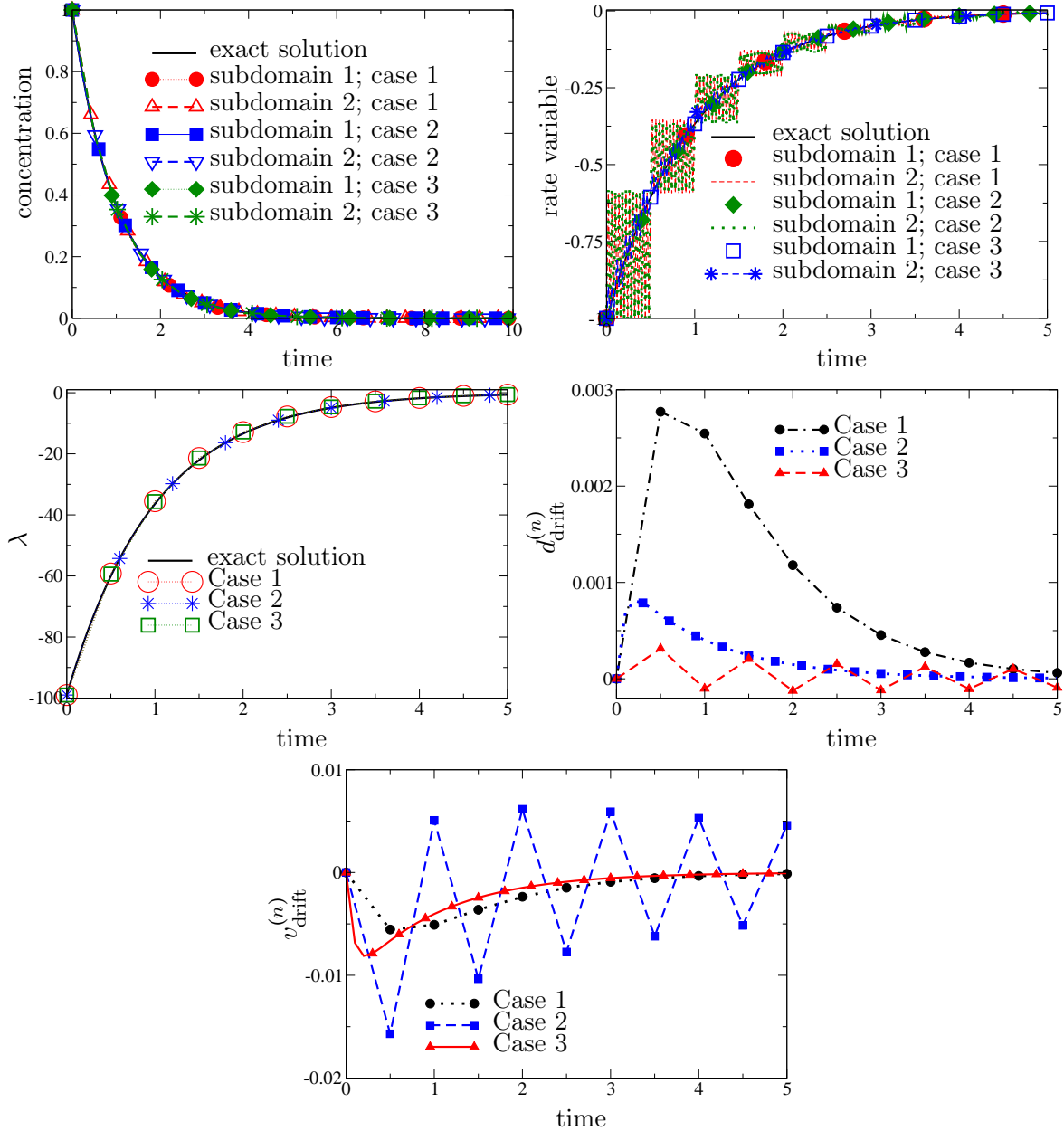

  \centering
  \psfrag{time}{time}
  \psfrag{cons}{concentration}
  \psfrag{L}{$\lambda$}
  \psfrag{rate}{rate variable}
  \psfrag{drift}{$d_{\mathrm{drift}}^{(n)}$}
  \psfrag{v_drift}{$v_{\mathrm{drift}}^{(n)}$}
  \psfrag{exact}{exact solution}
  \psfrag{s1_1}{subdomain 1; case 1}
  \psfrag{s1_2}{subdomain 2; case 1}
  \psfrag{s2_1}{subdomain 1; case 2}
  \psfrag{s2_2}{subdomain 2; case 2}
  \psfrag{s3_1}{subdomain 1; case 3}
  \psfrag{s3_2}{subdomain 2; case 3}
  \psfrag{s1}{Case 1}
  \psfrag{s2}{Case 2}
  \psfrag{s3}{Case 3}
  \subfigure{
    \includegraphics[scale = 0.3,clip]{SDOF_Baumgarte_cons.eps}
  }
  \subfigure{
    \includegraphics[scale = 0.3,clip]{SDOF_Baumgarte_rate.eps}
  }
  \subfigure{
    \includegraphics[scale = 0.3,clip]{SDOF_Baumgarte_lambda.eps}
  }
  \subfigure{
    \includegraphics[scale = 0.3,clip]{SDOF_Baumgarte_drift.eps}
  }
  \subfigure{
    \includegraphics[scale = 0.3,clip]{SDOF_Baumgarte_v_drift.eps}
  }
  \caption{Split degree-of-freedom problem: 
     The values of 
    concentrations, rate variables, Lagrange multipliers, 
    and drifts are compared with their respective exact solutions. 
    In this problem Baumgarte stabilization is used.  As it can 
    be observed, the accuracy can be improved by decreasing 
    the system time-step and increasing the Baumgarte stabilization 
    parameter $\alpha$. Note that there is no significant    	
    drift in the numerical solutions. Values for time-integration
    parameters are given in Table \ref{Tbl:SDOF_Baumgarte}.
    \label{Fig:SDOF_Baumgarte}}
\end{figure}

%------------------------------------------------------------;
% 	Figure: SDOF error plots for concentration	     ;
%------------------------------------------------------------;
\begin{figure}
\centering
\psfrag{time-step}{$\Delta t$}
\psfrag{subdomain}{Subdomains 1 and 2}
\psfrag{s1}{Subdomain 1}
\psfrag{s2}{Subdomain 2}
\psfrag{error}{Error in concentration}
\psfrag{sis2}{slope = $2$}
\subfigure[$d$-continuity method]{
\includegraphics[scale = 0.4, clip]{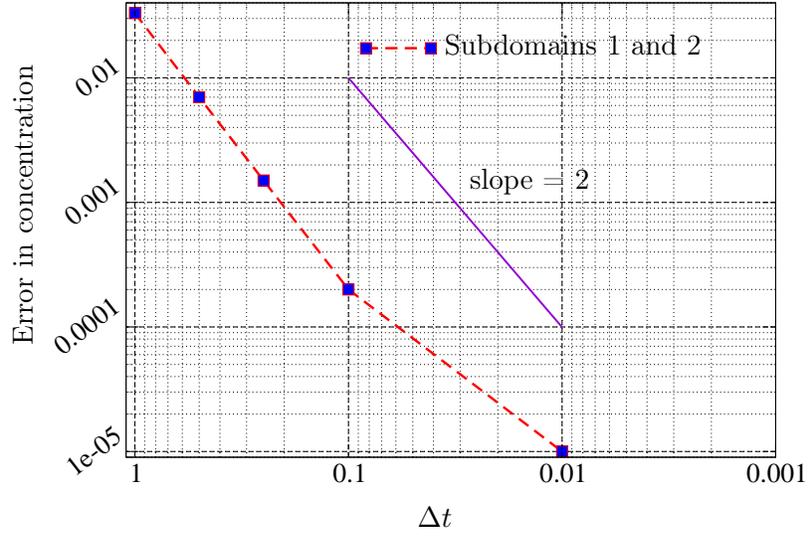}
}
\subfigure[Baumgarte stabilization method]{
\includegraphics[scale = 0.4, clip]{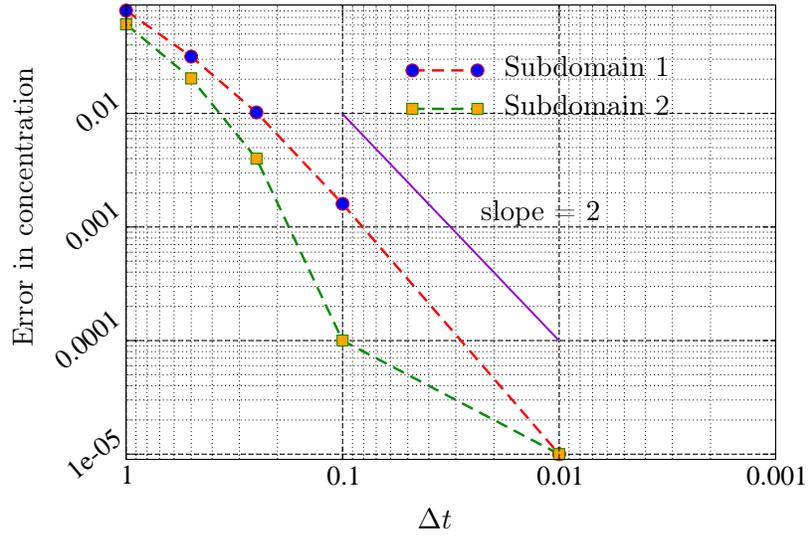}
}
\caption{Split degree-of-freedom problem: 
In these figures
absolute error vs. system time-step at $t = 1$ is plotted. 
In all cases, the subdomain time-steps are $\Delta t_i = 0.01, \; i=1,2$.
All subdomains are integrated using the midpoint rule 
($\vartheta_i = 1/2, \; i=1,2$). The Baumgarte stabilization 
parameter is $\alpha = 1$. These figures show the convergence
of the proposed method at a desirable rate; despite
subcycling, the convergence rate remains close to 2 (that of 
the midpoint rule).
\label{Fig:SDOF_Error}}
\end{figure}

%------------------------------------------------------------;
%  Figure: One-dimensional problem: A pictorial description  ;
%------------------------------------------------------------;
\begin{figure}
  \centering
  \includegraphics[scale=0.8,clip]{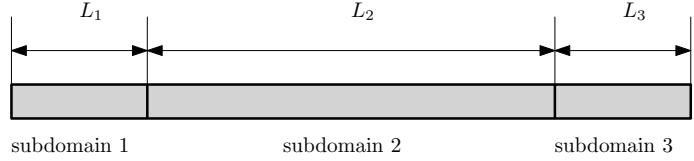}
  \caption{One-dimensional problem: The computational domain 
    is divided into three subdomains of lengths $L_1 = 0.1$, 
    $L_2 = 0.8$, and $L_3 = 0.1$. Two-node linear finite 
    elements are used in all the subdomains. The source 
    is unity in the entire domain (i.e., $f(\mathrm{x},t) 
    = 1$). It should be noted that there will be boundary 
    layers for the chosen parameters. In order to adequately 
    capture these boundary layers, we shall employ very 
    fine meshes in subdomains one and three. 
    \label{Fig:Hemker_1D}}
\end{figure}

%-------------------------------------------------;
%  Figure: One-dimensional problem, d-continuity  ;
%-------------------------------------------------;
\begin{figure}
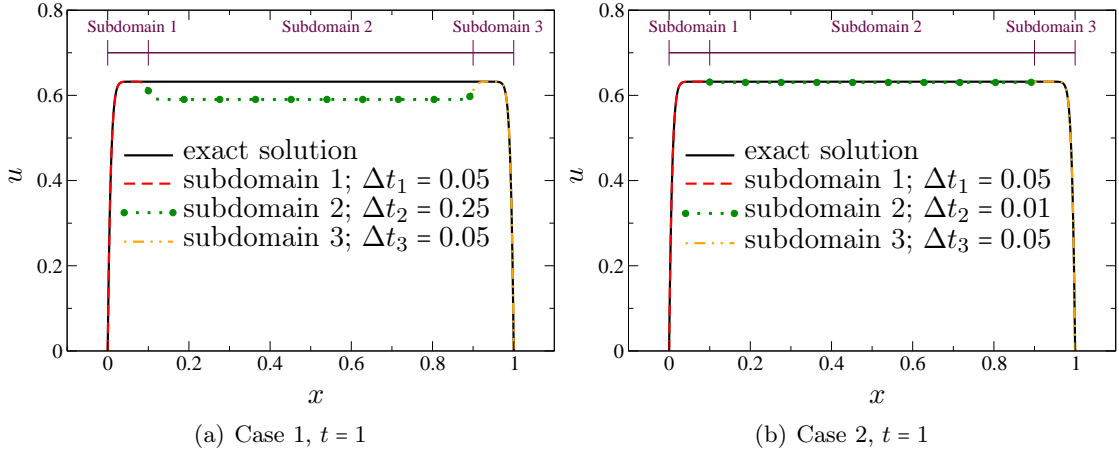

  \centering
  \psfrag{u}{$u$}
  \psfrag{x}{$x$}
  \psfrag{exact}{exact solution}
  \psfrag{sd1}{subdomain 1; $\Delta t_1 = 0.05$}
  \psfrag{sd2}{subdomain 2; $\Delta t_2 = 0.25$}
  \psfrag{sd2_2}{subdomain 2; $\Delta t_2 = 0.01$}
  \psfrag{sd3}{subdomain 3; $\Delta t_3 = 0.05$}
  \subfigure[Case 1, $t = 1$]{
    \includegraphics[scale=0.3,clip]{Hemker_1D_d_con_fig1.eps}}
  \subfigure[Case 2, $t = 1$]{
    \includegraphics[scale=0.3,clip]{Hemker_1D_d_con_fig2.eps}}
  \caption{One-dimensional problem: 
    This figure compares 
    the numerical solution obtained using the proposed 
    $d$-continuity method to the exact 
    solution. Each subdomain is meshed using $100$ 
    two-node finite elements. It should be noted that the time-stepping 
    schemes chosen are implicit in all the subdomains, 
    as it is not possible to have explicit/implicit 
    coupling under the $d$-continuity coupling method. 
    Time-integration parameters are given in 
    Table \ref{Tbl:Hemker_1D_d_con}.
    \label{Fig:1D_Hemker_d_continuity}}
\end{figure}

%----------------------------------------------;
%  Figure: One-dimensional problem, Baumgarte  ;
%----------------------------------------------;
\begin{figure}
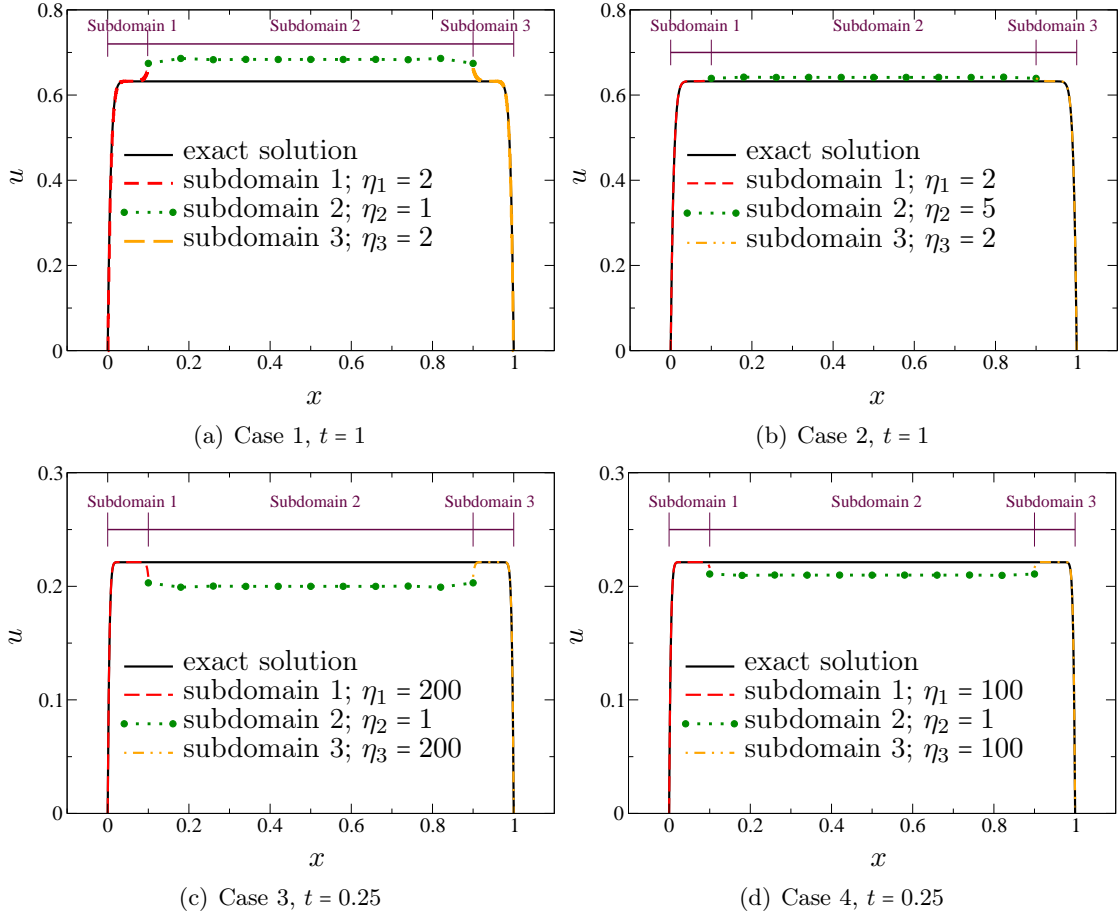

  \centering
  \psfrag{u}{$u$}
  \psfrag{x}{$x$}
  \psfrag{c}{$u$}
  \psfrag{t}{$x$}
  \psfrag{exact}{exact solution}
  \subfigure[Case 1, $t = 1$]{
    \psfrag{sd1}{subdomain 1; $\eta_1 = 2$}
    \psfrag{sd2}{subdomain 2; $\eta_2 = 1$}
    \psfrag{sd2_2}{subdomain 2; $\eta_2 = 5$}
    \psfrag{sd3}{subdomain 3; $\eta_3 = 2$}
    \includegraphics[scale = 0.3, clip]{Hemker_1D_Baumgarte_fig1.eps}}
\subfigure[Case 2, $t = 1$]{
  \psfrag{sd1}{subdomain 1; $\eta_1 = 2$}
  \psfrag{sd2}{subdomain 2; $\eta_2 = 1$}
  \psfrag{sd2_2}{subdomain 2; $\eta_2 = 5$}
  \psfrag{sd3}{subdomain 3; $\eta_3 = 2$}
  \includegraphics[scale = 0.3, clip]{Hemker_1D_Baumgarte_fig2.eps}}
\subfigure[Case 3, $t = 0.25$]{
  \psfrag{sd1}{subdomain 1; $\eta_1 = 200$}
  \psfrag{sd2}{subdomain 2; $\eta_2 = 1$}
  \psfrag{sd3}{subdomain 3; $\eta_3 = 200$}
  \includegraphics[scale = 0.3, clip]{Hemker_1D_Baumgarte_fig3.eps}}
\subfigure[Case 4, $t = 0.25$]{
  \psfrag{sd1}{subdomain 1; $\eta_1 = 100$}
  \psfrag{sd2}{subdomain 2; $\eta_2 = 1$}
  \psfrag{sd3}{subdomain 3; $\eta_3 = 100$}
  \includegraphics[scale = 0.3, clip]{Hemker_1D_Baumgarte_fig4.eps}}
\caption{One-dimensional problem: 
  The numerical 
  solution using the proposed coupling method 
  with Baumgarte stabilization is shown in this figure. As it 
  was shown in theorem \ref{Thm:S4_Stability_Baumgarte}, 
  when conditionally stable trapezoidal schemes are used, 
  multi-time-stepping can expand the acceptable values 
  of $\alpha$ without compromising the stability of the 
  coupling method. Under the proposed coupling
  method, choosing system time-step larger than the critical
  time-step of subdomains, does not cause instability. However,
  reducing the system time-step, or opting for a larger Baumgarte
  stabilization parameter $\alpha$, improves the overall accuracy.
  Time-integration parameters in different cases of the 
  are given in Table \ref{Tbl:Hemker_1D_Baumgarte}.
\label{Fig:1D_Hemker_Baumgarte_Set1}}
\end{figure}

%-------------------------------------------;
%  Figure: 1D Hemker, d-continuity, Drifts  ;
%-------------------------------------------;
\begin{figure}
\centering
\psfrag{time}{time}
\psfrag{drift}{$\left\| \boldsymbol{v}_{\mathrm{drift}}^{(n)}\right\|_{2}$}
\psfrag{s2}{Case 2}
\psfrag{es1}{Case 3; observed}
\psfrag{es1_pred}{Case 3; predicted}
\includegraphics[scale = 0.4, clip]{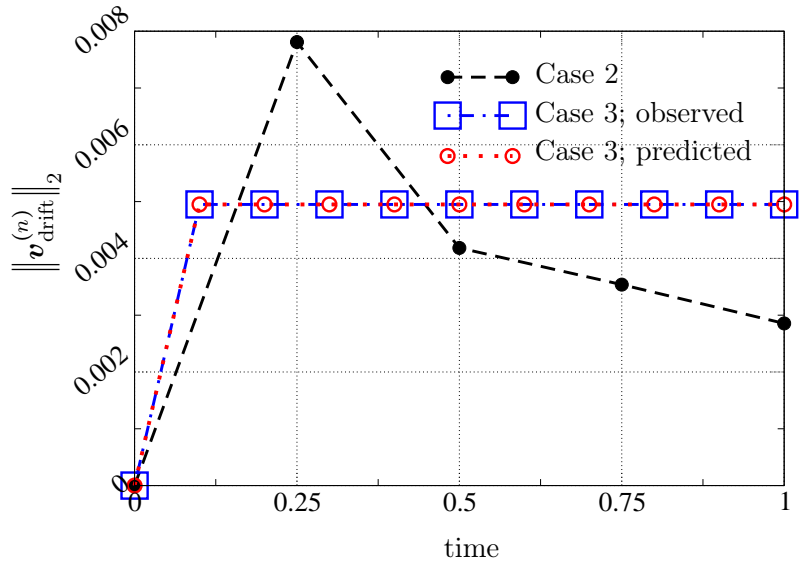}
\caption{One-dimensional problem:  This figure shows the 2-norm 
  of the drift in the rate variable under the $d$-continuity 
  method. (Note that, by algorithmic design, there will be no 
  drift in the concentration along the subdomain interface at 
  all system time levels.) 
  In case 2, subcycling and mixed time-integrators are 
used. To demonstrate the correctness of equation \eqref{Eqn:d_con_drift},
a third case is devised. For the values of the time-integration
parameters in different cases please see Table \ref{Tbl:Hemker_1D_d_con}. 
\label{Fig:1D_Hemker_d_continuity_drift}}
\end{figure}

%---------------------------------------;
%  Figure : 1D Hemker Baumgarte drifts  ;
%---------------------------------------;
\begin{figure}
\centering
\psfrag{tc}{Case 5; observed and predicted}
\psfrag{cb}{Case 2}
\psfrag{cc}{Case 4}
\psfrag{time}{time}
\psfrag{d_drift}{$\left\| \boldsymbol{d}_{\mathrm{drift}}^{(n)}\right\|_2$}
\psfrag{v_drift}{$\left\| \boldsymbol{v}_{\mathrm{drift}}^{(n)}\right\|_2$}
\subfigure[Drift in concentrations]{\includegraphics[scale = 0.4, clip]{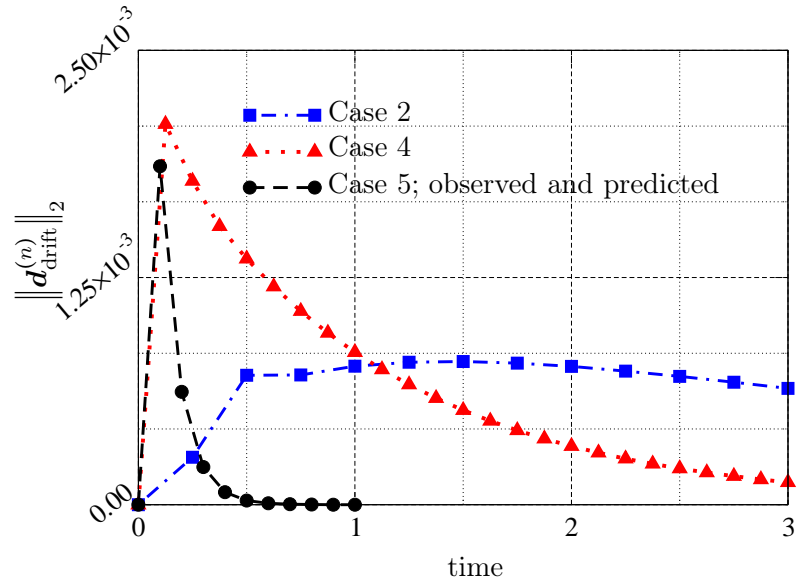}}
\subfigure[Drift in rate variables]{\includegraphics[scale = 0.4, clip]{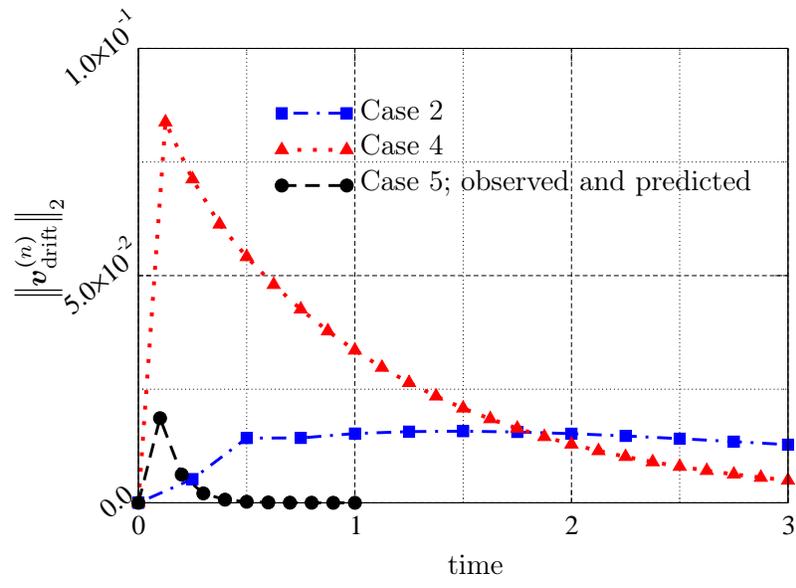}}
\caption{One-dimensional problem:  This figure shows the 
  drifts in the concentration and the rate variable for 
  various cases under the Baumgarte stabilization method. 
  As it can be observed,
the proposed method with Baumgarte stabilization enables 
explicit/implicit coupling at the expense of controlled drifts.
For the value of time-integration parameters in 
different cases see Table \ref{Tbl:Hemker_1D_Baumgarte}.
\label{Fig:1D_Hemker_Baumgarte_drift}}
\end{figure}

%----------------------------------------------------;
%  Figure: Two-dimensional transient Hemker problem  ;
%----------------------------------------------------;
\begin{figure}
  \centering
  \includegraphics[scale = 0.8,clip]{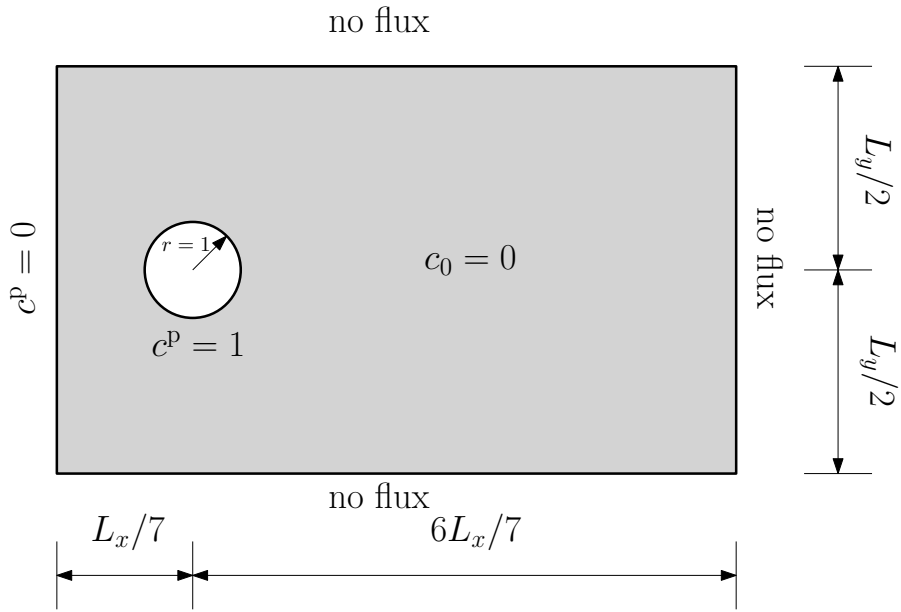}
  \caption{Two-dimensional transient Hemker problem: The 
    dimensions of the computational domain are taken as 
    $L_x = 14$ and $L_y = 8$. A circular hole is centered 
    at the origin, and has a radius of unity. Concentration 
    is unity on the circumference of the circle, and zero 
    along the left side of the domain. No-flux boundary 
    condition is enforced on the rest of the boundary. The 
    prescribed initial condition is zero. \label{Fig:Hemker_2D}}
\end{figure}

%-------------------------------------------------------------;
%  Figure: 2D transient Hemker problem: Domain decomposition  ;
%-------------------------------------------------------------;
\begin{figure}
  \centering 
  \includegraphics[scale = 0.75, clip]{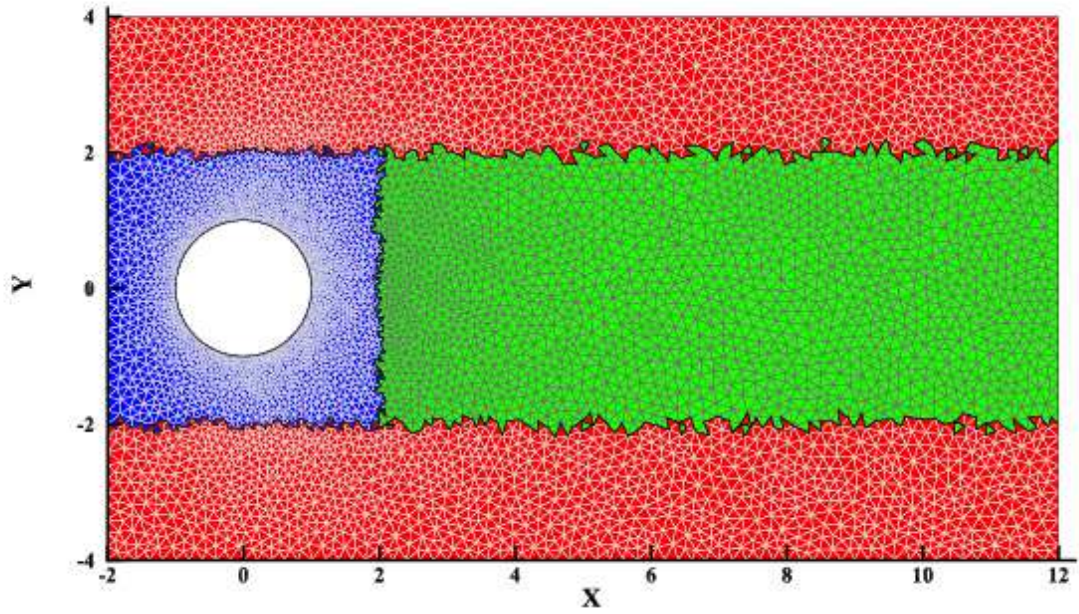}
  \caption{Two-dimensional transient Hemker problem:~This 
    figure shows the computational mesh, and the decomposition 
    of the domain into subdomains. The computational domain is 
    meshed using 11512 triangular finite elements using GMSH 
    \citep{Gmsh_paper}, and is partitioned into three subdomains. 
    The first subdomain is indicated in blue color, the second 
    subdomain is in green color, and the third subdomain is in red 
    color. (See the online version of the paper for a color picture.) 
    \label{Fig:2D_Hemker_mesh}}
\end{figure}

%------------------------------------------;
%  Figure: 2D Hemker Galerkin formulation  ;
%------------------------------------------;
\begin{figure}
  \centering
  \subfigure[$d$-continuity method]{
    \includegraphics[scale = 0.32, clip]{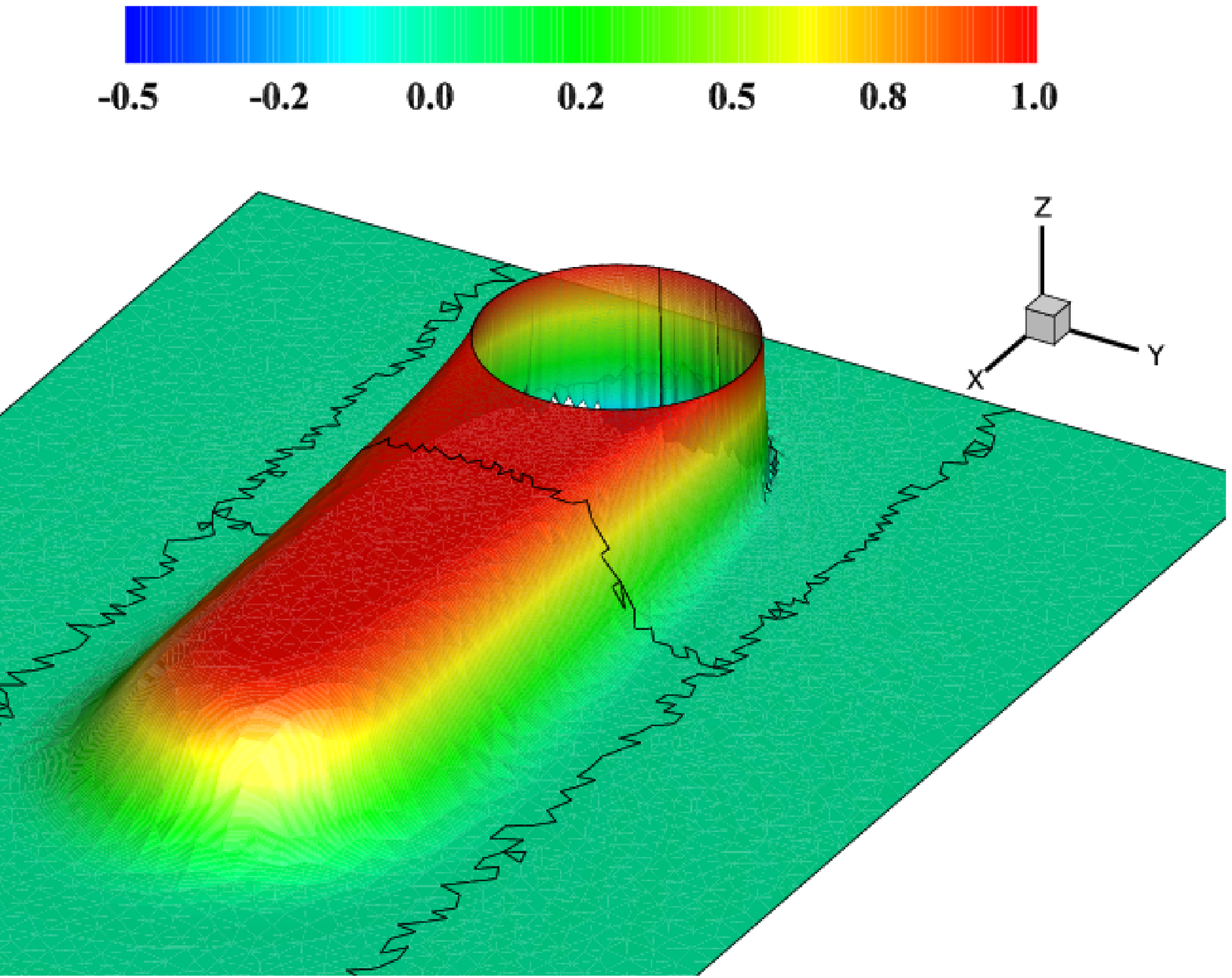}}
  \subfigure[Baumgarte stabilization]{
    \includegraphics[scale = 0.32, clip]{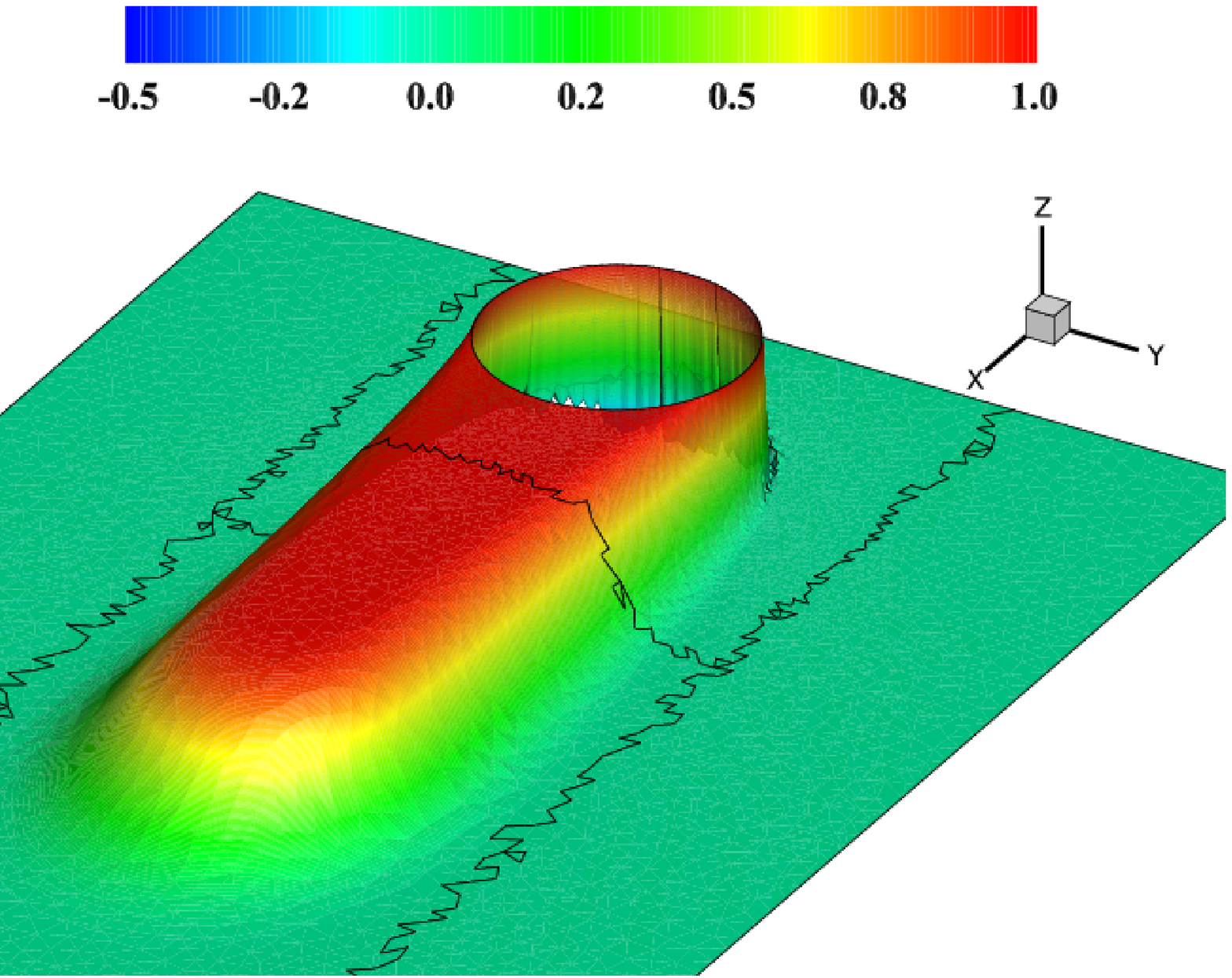}}
  \caption{Two-dimensional transient Hemker problem:
  	 The value of 
    concentrations is shown on the domain of interest at $t = 5$.
    In this particular example, Galerkin weak formulation is 
    employed. In figure (a), $d$-continuity method 
    is employed to enforce continuity at the subdomain interface.
    The computational domain is partitioned into three subdomains. 
    Figure (b) shows the results when Baumgarte stabilization
    is employed to enforce continuity at the interface. 
    Spurious oscillations due to semi-discrete Galerkin
    method can be seen in near the circle. The minimum value of 
    concentrations seen in these examples is -0.439, which is 
    significant compared to the maximum, which is unity. 
    The values of numerical time-integration parameters 
    are given in Table \ref{Tbl:2D_Hemker_1}.
    \label{Fig:2D_Hemker_Galerkin}}
\end{figure}

%---------------------------------------;
%  2D Hemker , GLS, SUPG, Galerkin 	;
%---------------------------------------;
\begin{figure}
\centering
\subfigure[$d$-continuity method]{
\includegraphics[scale = 0.32, clip]{2D_Hemker_GLS_SUPG_Galerkin_d_Continuity.eps}}
\subfigure[Baumgarte stabilization]{
\includegraphics[scale = 0.32, clip]{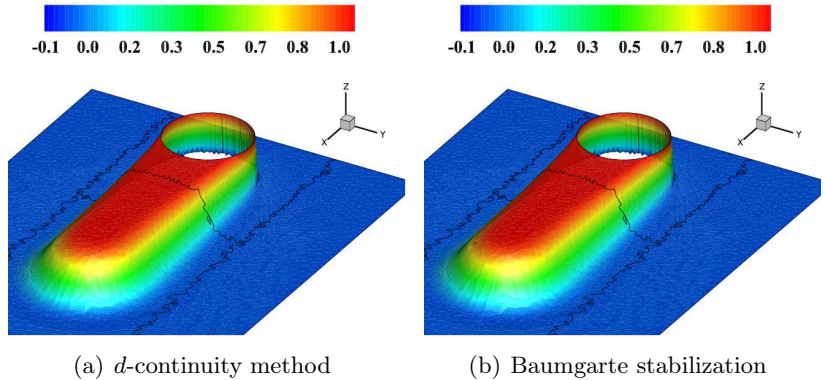}}
\caption{Two-dimensional transient Hemker problem:~
	 Concentrations 
	at $t = 5$ are shown. GLS formulation is used in subdomain 1, SUPG 
	formulation is used in subdomain 2, and the standard Galerkin 
	formulation is used in subdomain 3. The minimum value for 
	concentrations is -0.062 in both cases. Time-integration 
	parameters are given in Table \ref{Tbl:2D_Hemker_2}.
\label{Fig:2D_Hemker_GLS_SUPG_Galerkin}}
\end{figure}

%---------------------------------------------------;
% 	Figure: 2D Hemker, drift in concentrations      ;
%---------------------------------------------------;
\begin{figure}
\centering
\psfrag{d_con}{$d$-continuity}
\psfrag{baum_gal}{Baumgarte; Standard Galerkin}
\psfrag{baum_mixed}{Baumgarte; GLS-SUPG-Galerkin}
\psfrag{time}{time}
\psfrag{d_drift}{$\left\| \boldsymbol{d}_{\mathrm{drift}}^{(n)} \right\|_{\infty}$}
\includegraphics[scale = 0.5, clip]{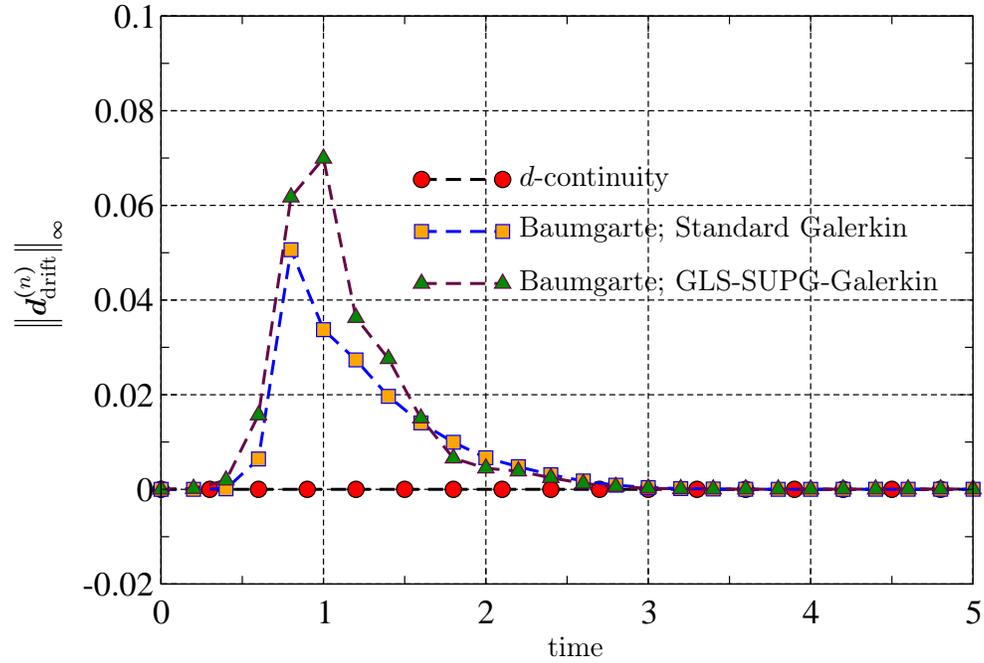}
\caption{Two-dimensional transient Hemker problem:~
 The maximum drift in concentrations is plotted against
time. The time integration parameters are the same
as in Figures \ref{Fig:2D_Hemker_Galerkin} and 
\ref{Fig:2D_Hemker_GLS_SUPG_Galerkin}. In case of
$d$-continuity method, there will be no drift in 
concentrations. As it can be observed in the 
Baumgarte stabilization method drifts are 
controlled. \label{Fig:2D_Hemker_drift}}
\end{figure}

%-------------------------------------------;
%  Figure: Bimolecular reaction -- problem  ;
%-------------------------------------------;
\begin{figure}
  \centering
  \subfigure[A pictorial description of the problem.]{
    \includegraphics[scale=0.6,clip]{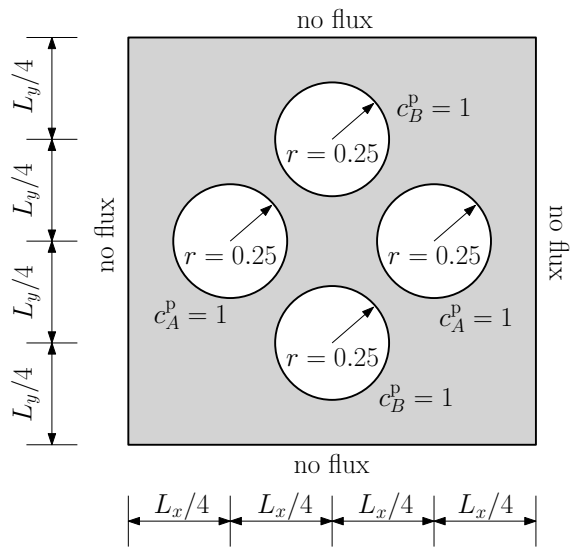}}
  \subfigure[Domain decomposition]{
    \includegraphics[scale=0.4,clip]{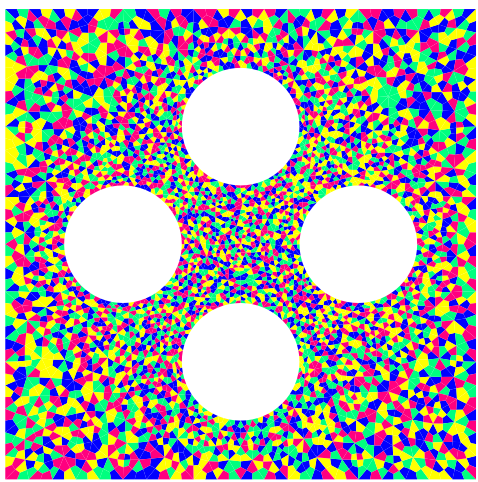}}
  \caption{Diffusion-controlled fast bimolecular reaction:~The 
    initial condition for the concentrations of all reactants is 
    taken to be zero. The computational domain is meshed using 
    5442 four-node quadrilateral elements, and is divided into 
    four subdomains using METIS \citep{METIS_paper}. Subdomain one is 
    indicated in blue color, subdomain two is indicated in green 
    color, subdomain three is in yellow color, and subdomain 
    four is in red color. (See the online version of the paper 
    for a color picture.) \label{Fig:Bimolecular_Baumgarte}}
\end{figure}

%------------------------------------------------;
%  Figure: Bimolecular no advection: Invariants  ;
%------------------------------------------------;
\begin{figure}
\centering
\subfigure[Concentration of invariant $F$ at $t = 0.01$.]{
    \includegraphics[scale=0.27,clip]{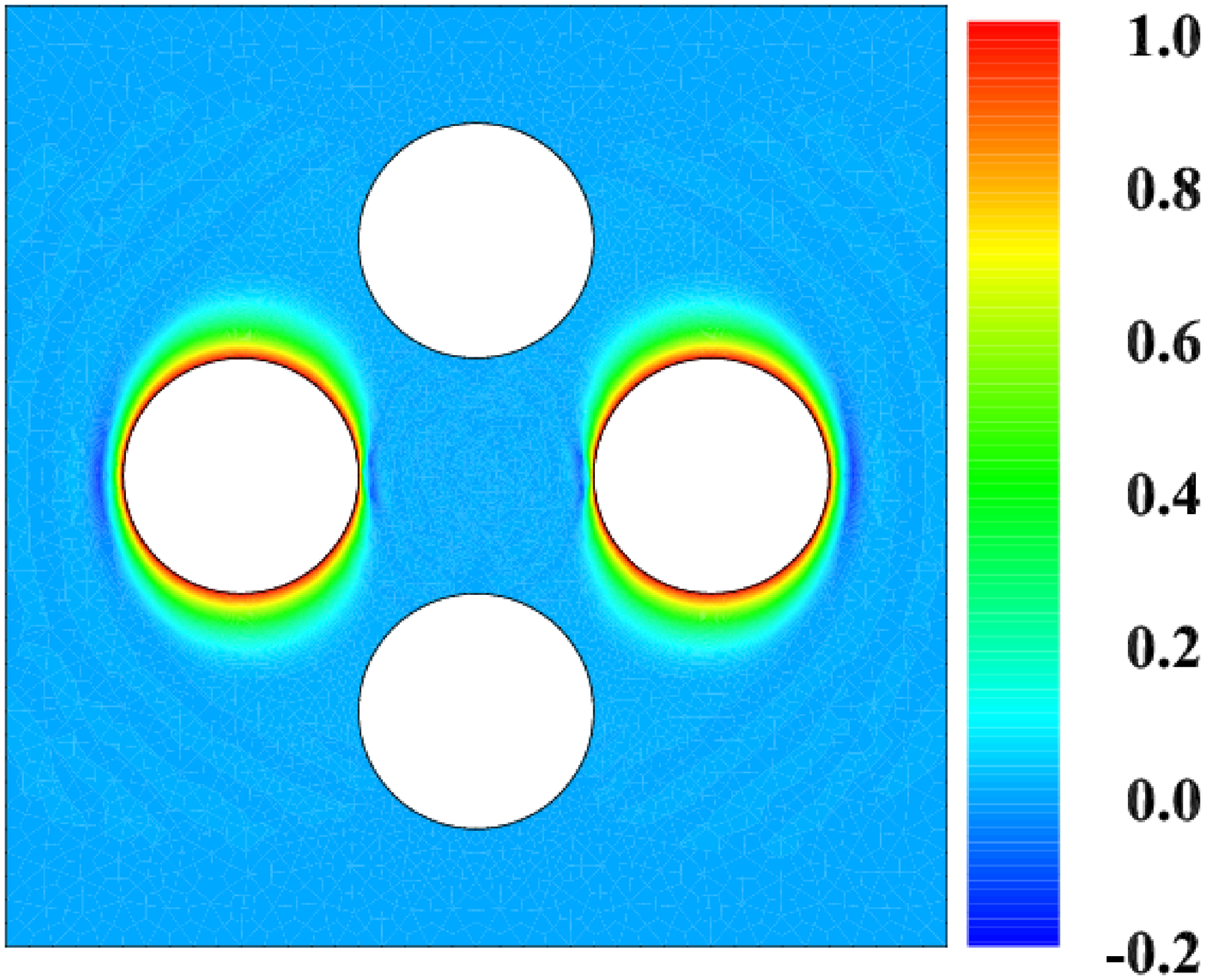}}
  \subfigure[Concentration of invariant $F$ at $t = 0.1$.]{
    \includegraphics[scale=0.27,clip]{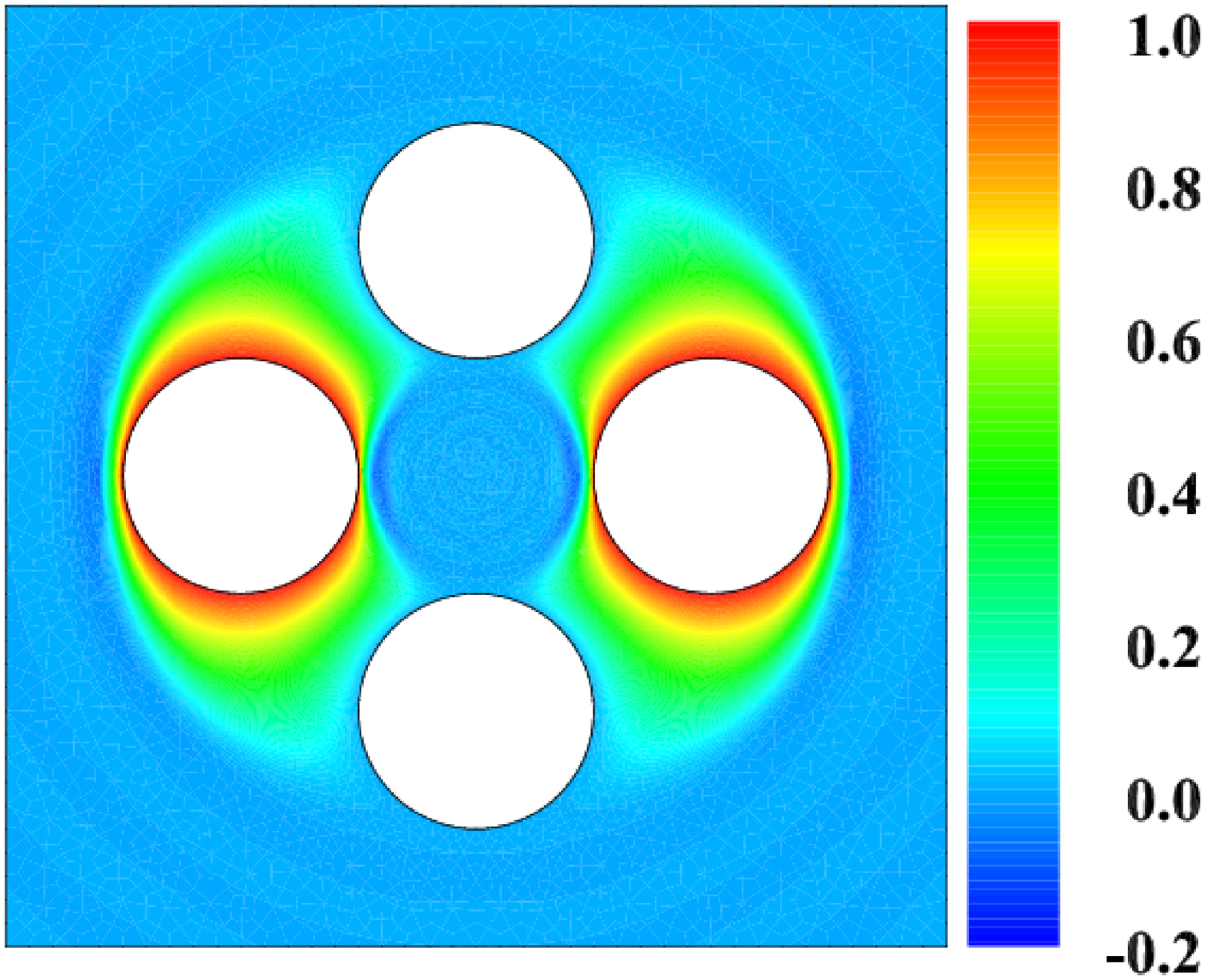}}
  \subfigure[Concentration of invariant $G$ at $t = 0.01$.]{
    \includegraphics[scale=0.27,clip]{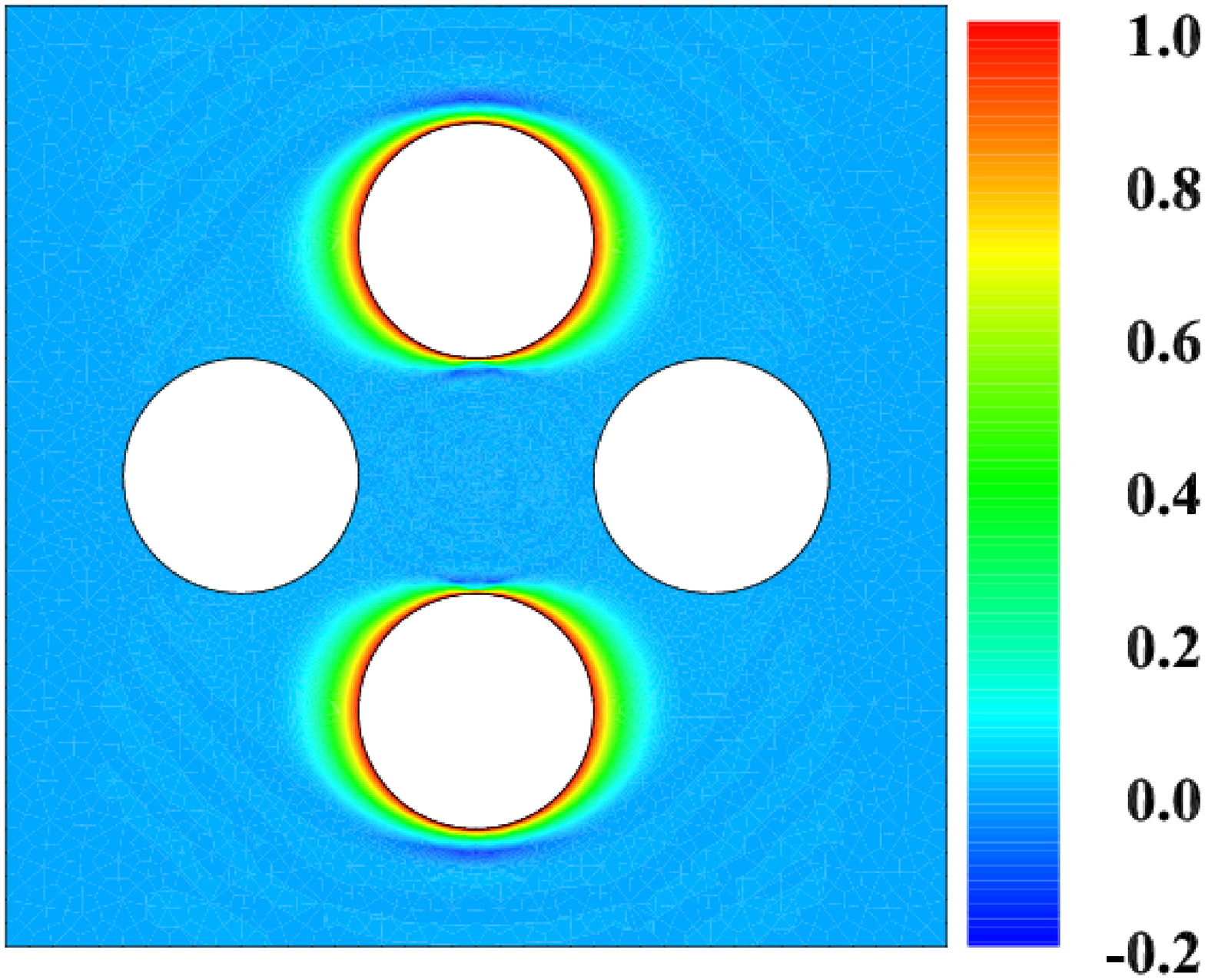}}
  \subfigure[Concentration of invariant $G$ at $t = 0.1$.]{
    \includegraphics[scale=0.27,clip]{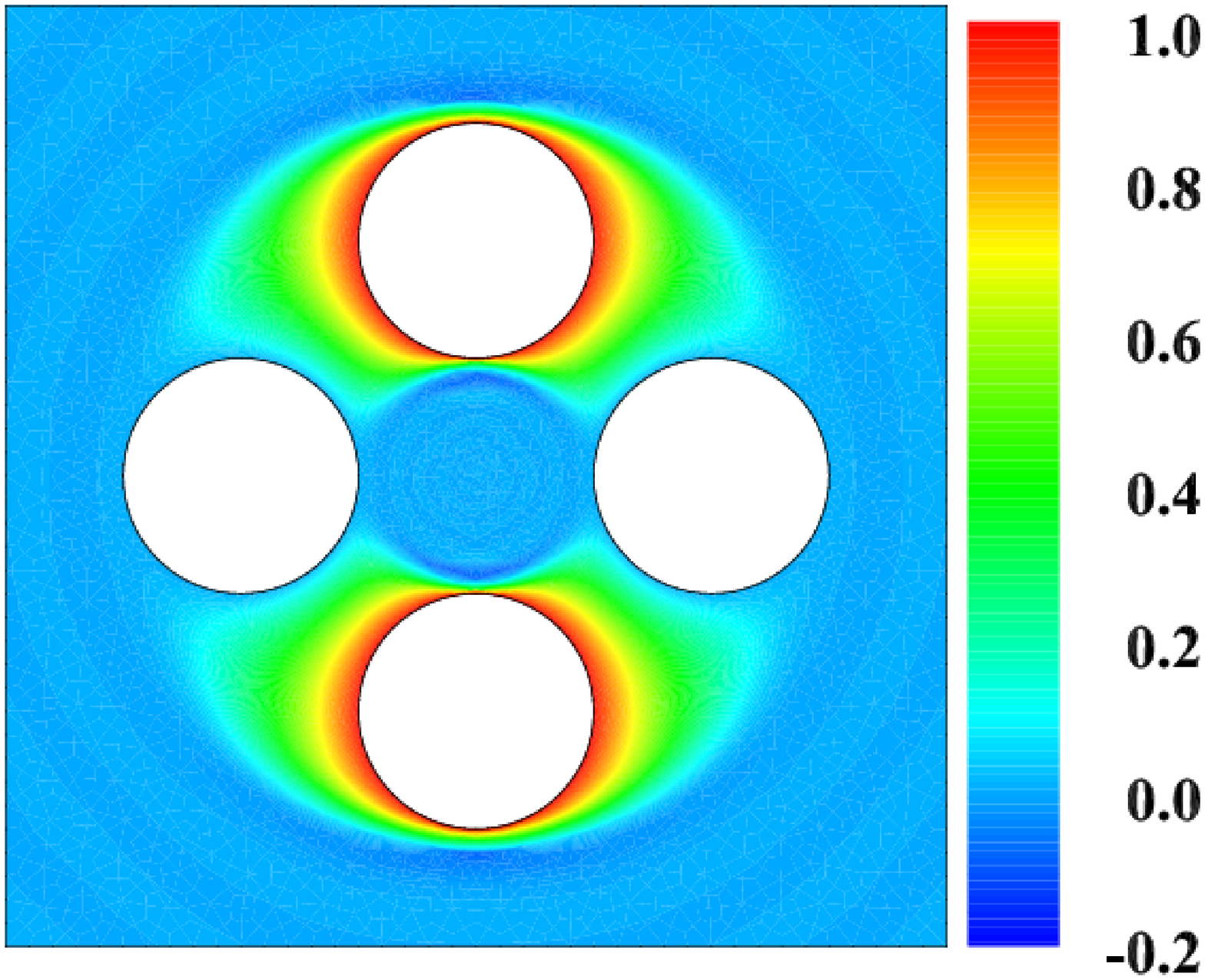}}
  \caption{Diffusion-controlled fast bimolecular reaction:
   This figure shows the concentrations 
    of the invariants $F$ and $G$ at $t = 0.01$ and 
    $t = 0.1$. \label{Fig:Bimolecular_Baumgarte_set1_FandG}}
\end{figure}

%----------------------------------------------------------------------;
%  Figure: Bimolecular no advection, Baumgarte, reactants and product  ;
%----------------------------------------------------------------------;
\begin{figure}
  \centering
  \subfigure[Concentration of reactant $A$ at $t = 0.01$.]{
    \includegraphics[scale=0.27,clip]{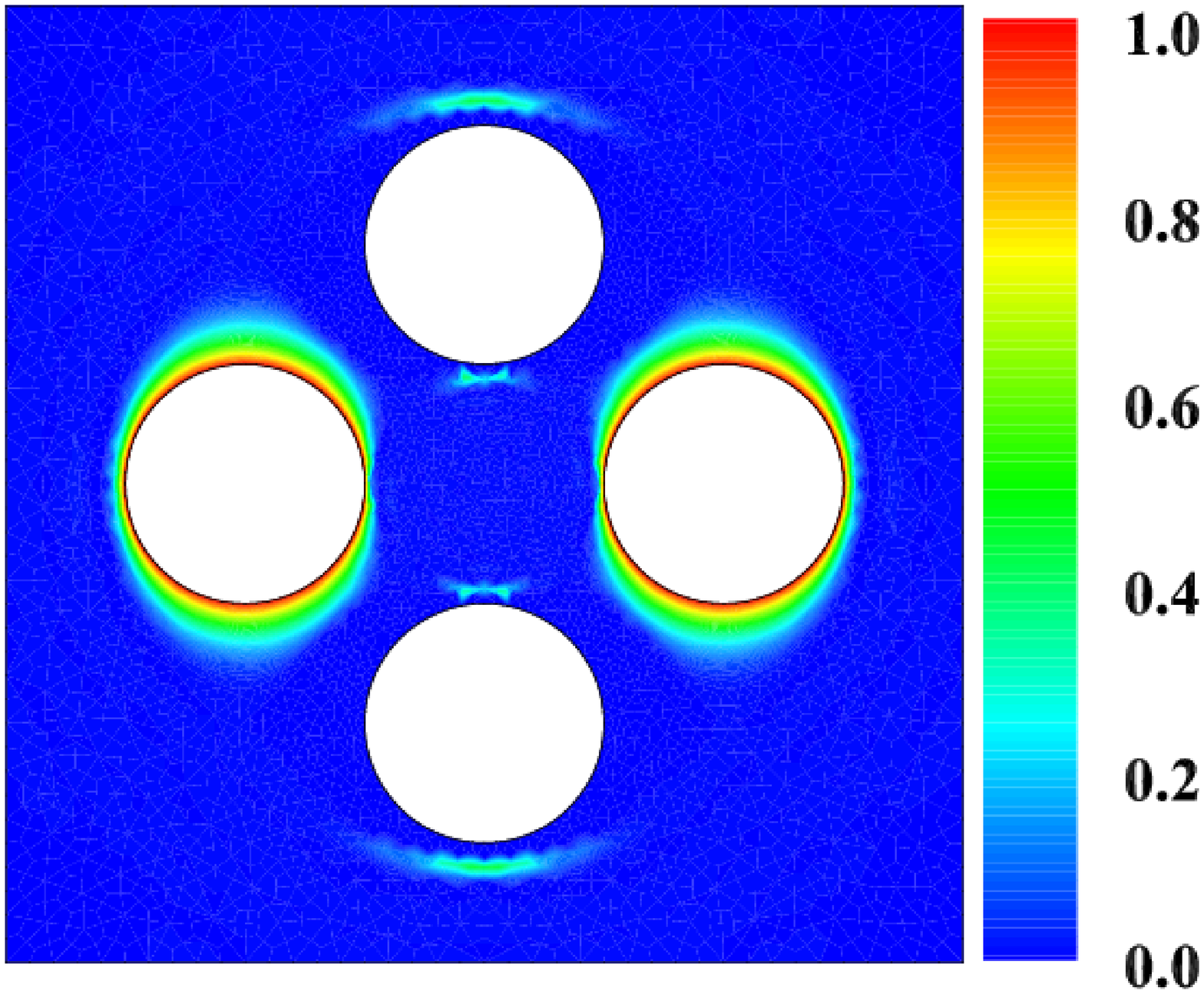}}
  \subfigure[Concentration of reactant $A$ at $t = 0.1$.]{
    \includegraphics[scale=0.27,clip]{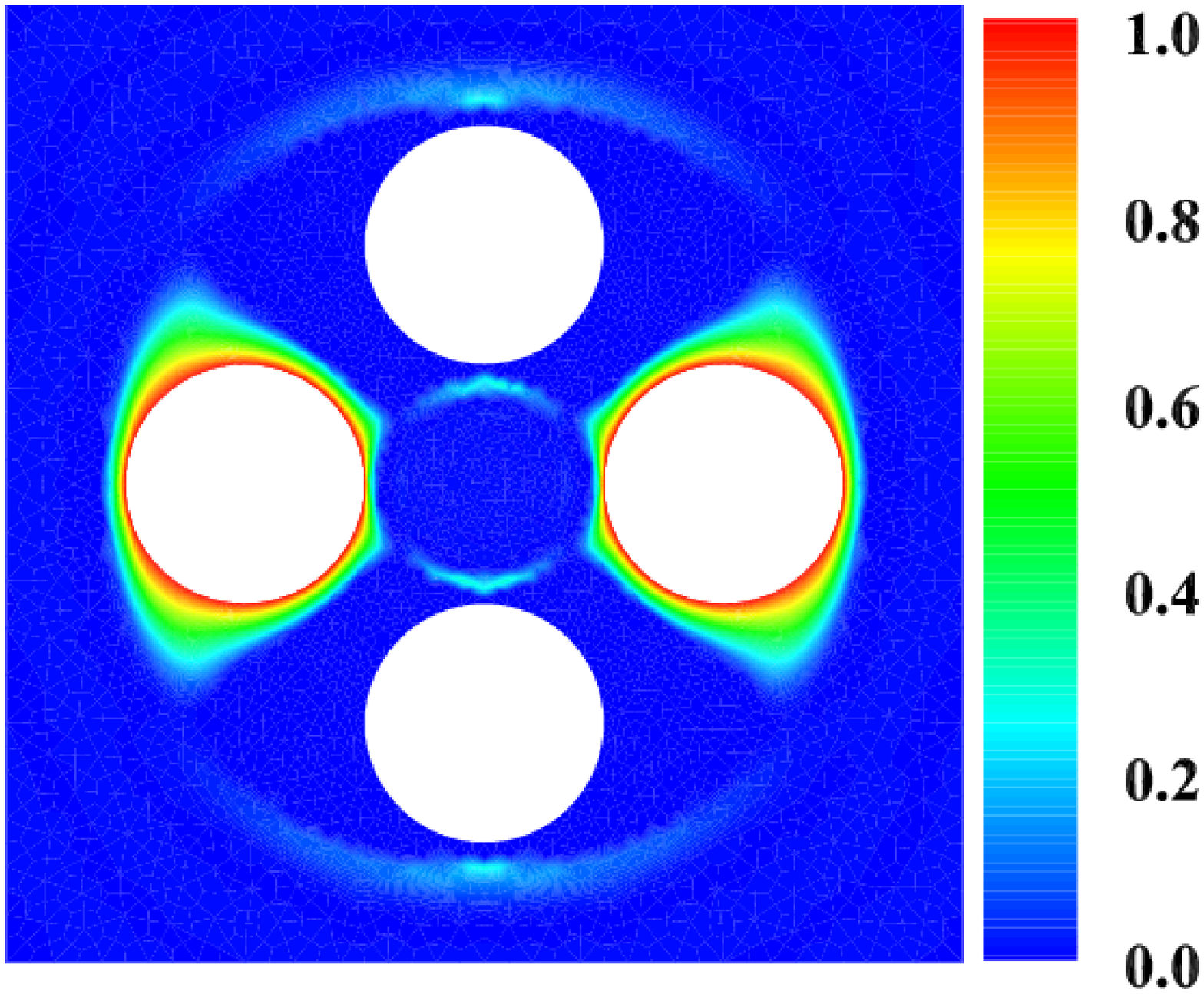}}
  \subfigure[Concentration of reactant $B$ at $t = 0.01$.]{
    \includegraphics[scale=0.27,clip]{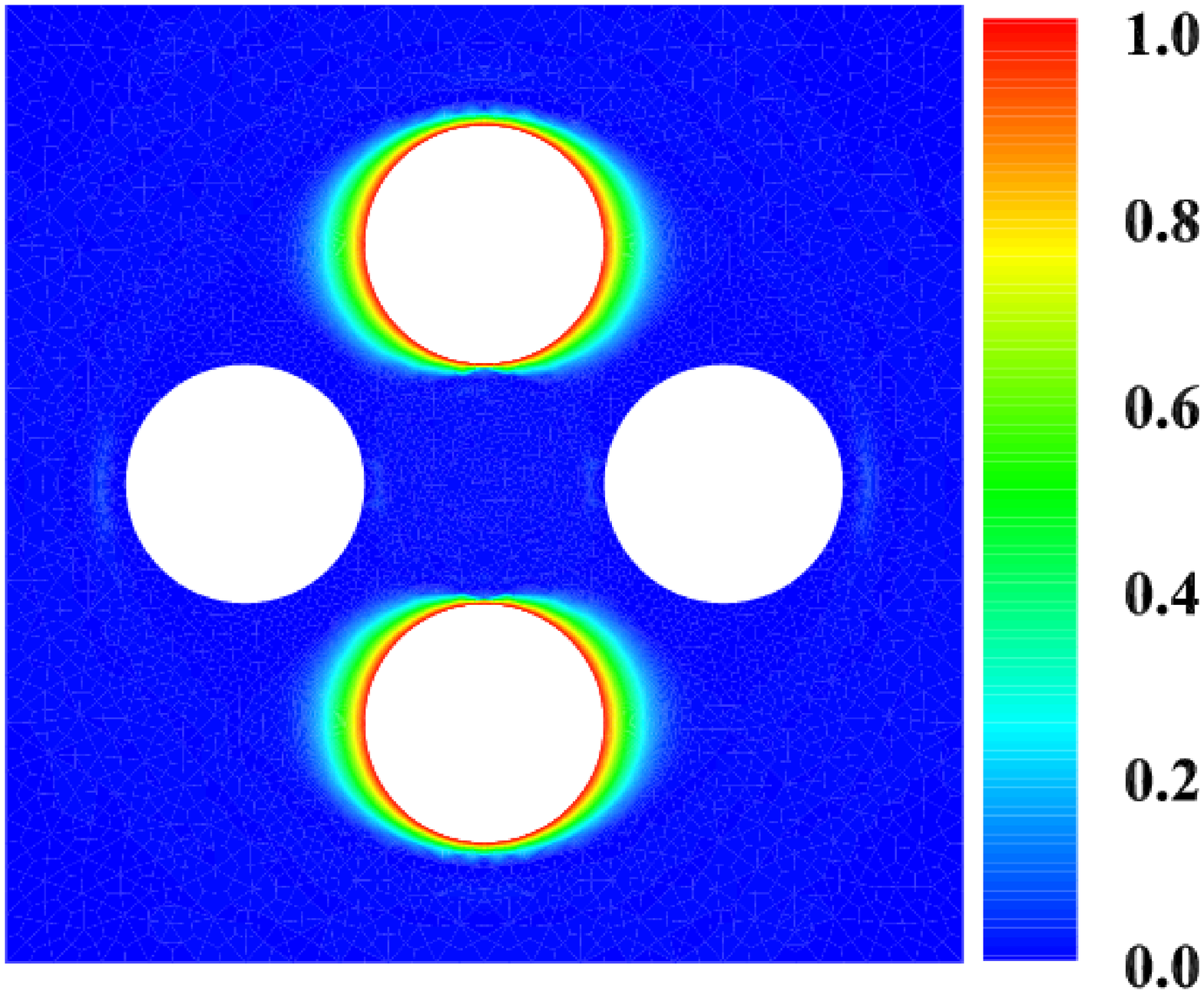}}
  \subfigure[Concentration of reactant $B$ at $t = 0.1$.]{
    \includegraphics[scale=0.27,clip]{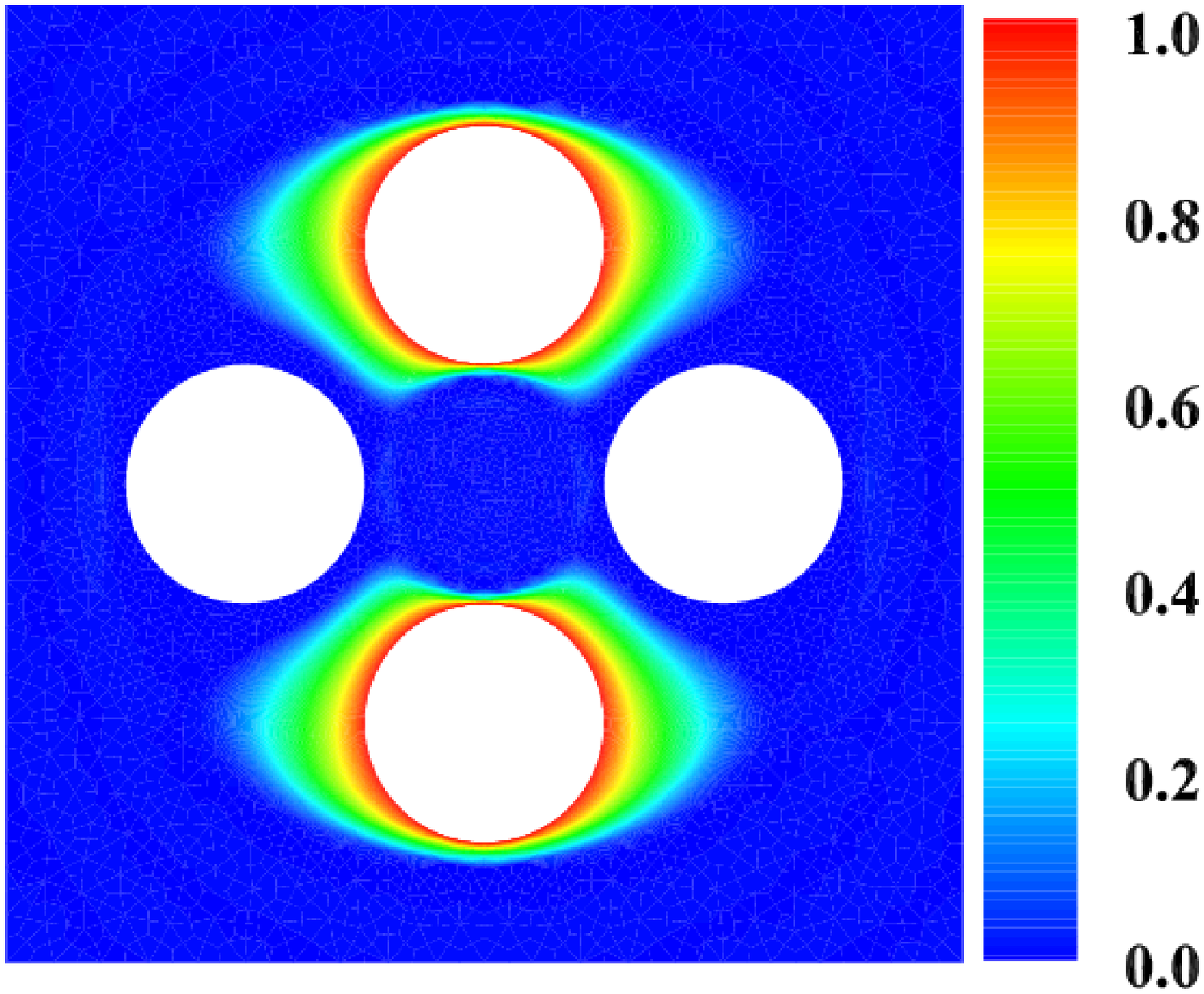}}
  \subfigure[Concentration of product $C$ at $t = 0.01$.]{
    \includegraphics[scale=0.27,clip]{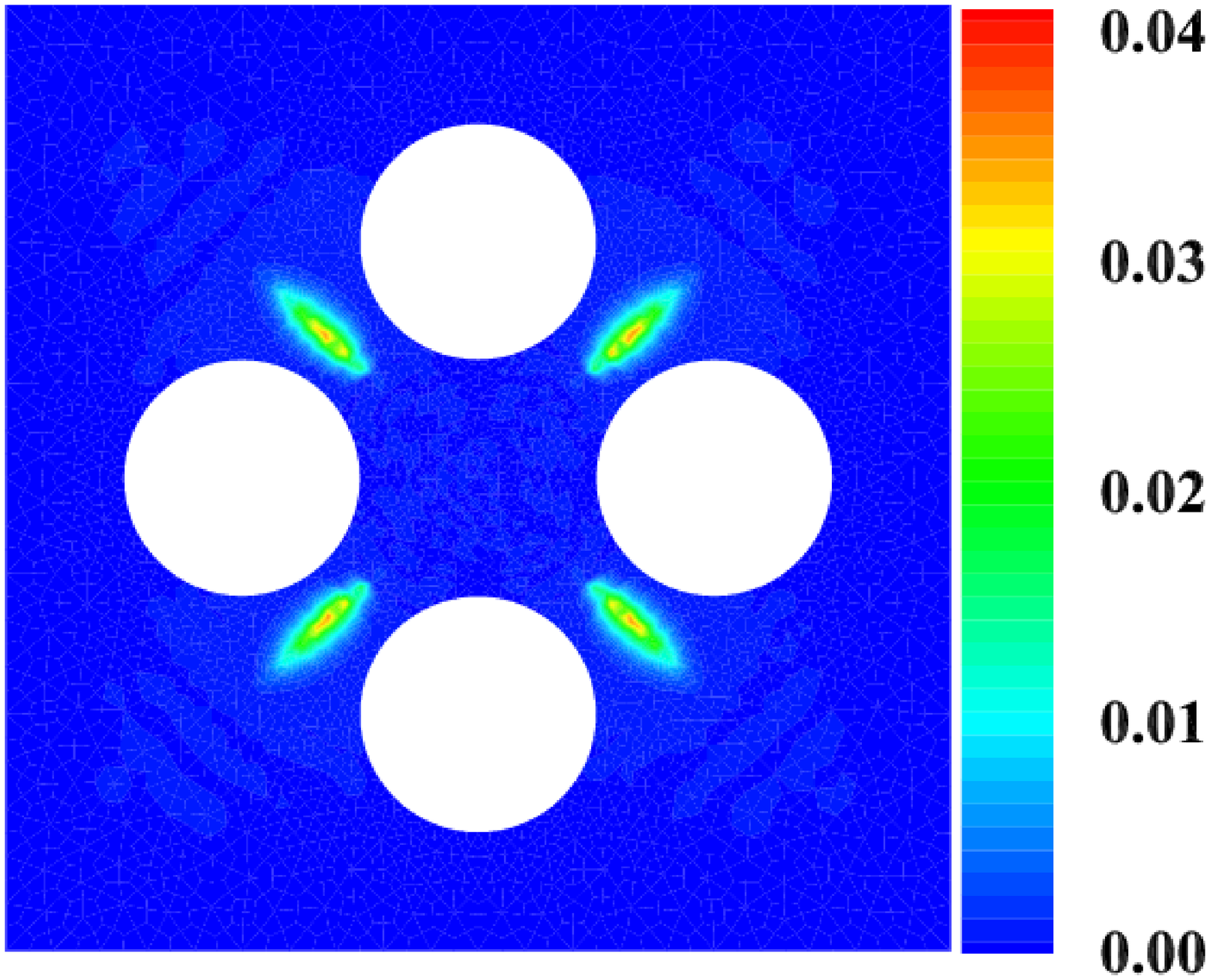}}
  \subfigure[Concentration of product $C$ at $t = 0.1$.]{
    \includegraphics[scale=0.27,clip]{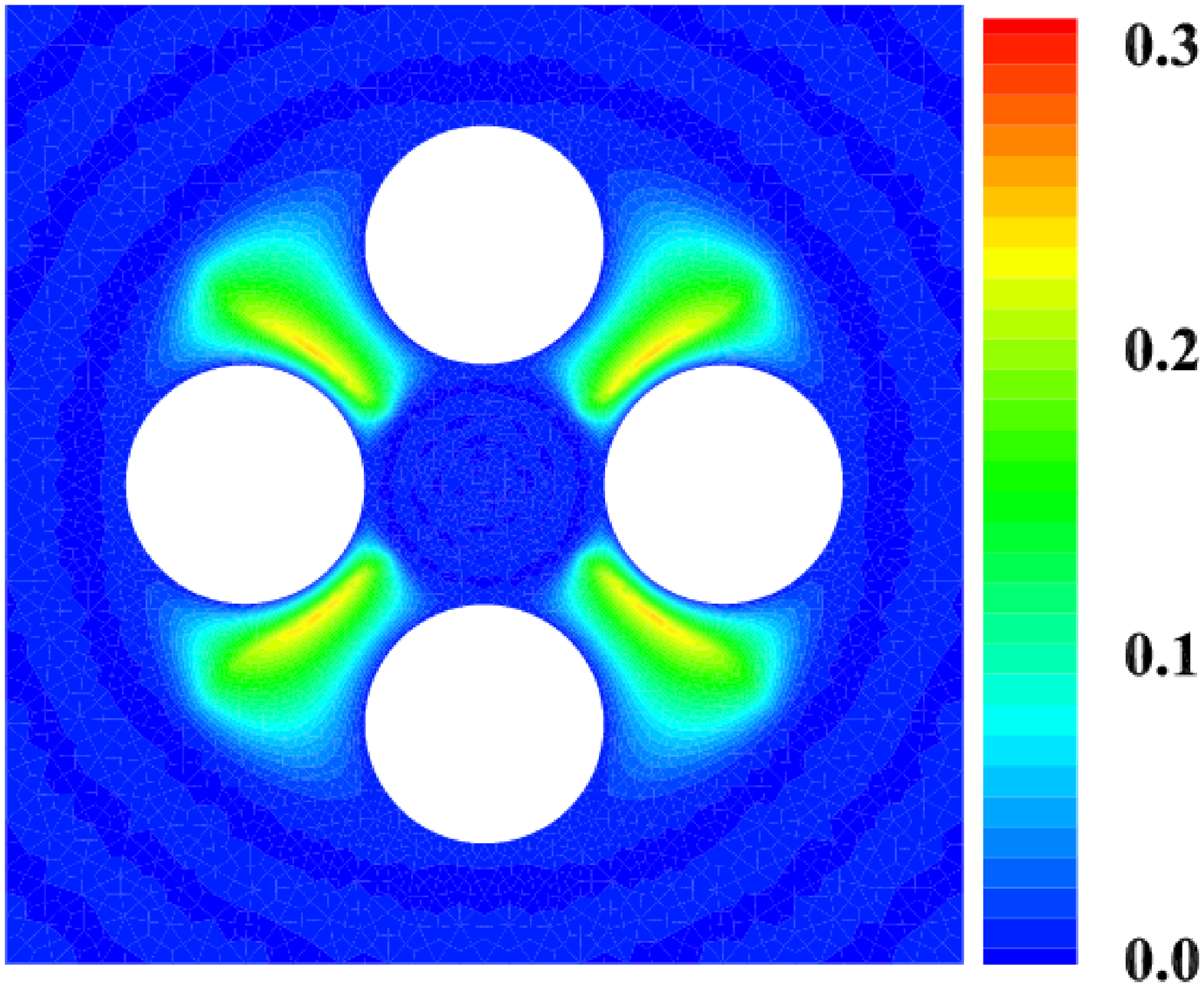}}
  \caption{Diffusion-controlled fast bimolecular reaction:
  ~Concentrations of the reactants and 
    the product are shown at $t = 0.01$ and $t = 0.1$. 
    \label{Fig:Bimolecular_Baumgarte_set1}}
\end{figure}

%-----------------------------;
%  Figure: Bimolecular drift  ;
%-----------------------------;
\begin{figure}
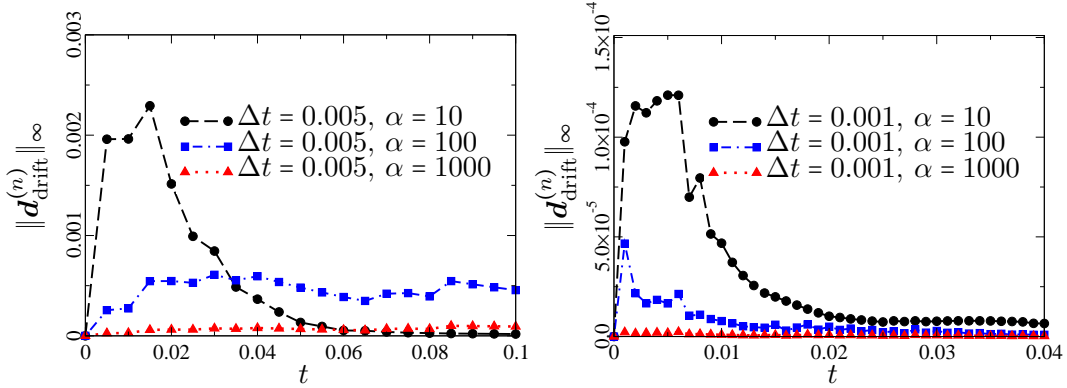

\centering
\psfrag{t}{$t$}
\psfrag{d}{$\| \boldsymbol{d}_{\mathrm{drift}}^{(n)}\|_{\infty}$}
\subfigure{
\psfrag{s1}{$\Delta t = 0.005$, $\alpha = 10$}
\psfrag{s2}{$\Delta t = 0.005$, $\alpha = 100$}
\psfrag{s3}{$\Delta t = 0.005$, $\alpha = 1000$}
\includegraphics[scale = 0.265, clip]{Bimolecular_Baumgarte_drift_1.eps}
}
\subfigure{
\psfrag{s1}{$\Delta t = 0.001$, $\alpha = 10$}
\psfrag{s2}{$\Delta t = 0.001$, $\alpha = 100$}
\psfrag{s3}{$\Delta t = 0.001$, $\alpha = 1000$}
\includegraphics[scale = 0.265, clip]{Bimolecular_Baumgarte_drift_2.eps}
}
\caption{Diffusion-controlled fast bimolecular 
  reaction:~ This figure shows the drift in the 
  concentration of the chemical species $C$ in 
  the $\infty$-norm along the subdomain interface 
  under the Baumgarte stabilization coupling method. 
  The subdomain time-steps are $\Delta t_1 = \Delta 
  t_3 = 5 \times 10^{-4}$, and $\Delta t_2 = \Delta 
  t_4 = 10^{-3}$. Implicit Euler method is employed 
  in subdomains 1 and 3 (i.e., $\vartheta_1 = 
  \vartheta_3 = 1$), and midpoint rule is employed 
  in subdomains 2 and 4 (i.e., $\vartheta_2 = \vartheta_4 
  = 1/2$). There are 10047 interface constraints in this 
  problem. 
  The main observation is that the drift always decreases 
  with decrease in system time-step. On the other hand, 
  the drift typically decreases with increase in the Baumgarte 
  stabilization parameter $\alpha$. But the presence of 
  subcycling and mixed methods will alter the monotonic 
  decreasing property over the entire time of interest.  
  Although subcycling is present in this 
  problem, it has been observed that the 
  drifts followed the general trend predicted 
  by equation \eqref{Eqn:S4_Drift_Measurement}, 
  which assumes no subcycling. 
  \label{Fig:Bimolecular_Baumgarte_Drift}}
\end{figure}

%------------------------------------------------------------;
%  Figure: Bimolecular reaction problem, rectangular domain  ;
%------------------------------------------------------------;
\begin{figure}
  \centering
  \psfrag{ly}{$L_y / 2$}
  \psfrag{lx}{$L_x$}
  \psfrag{cA}{$c_A^{\mathrm{p}}$}
  \psfrag{cB}{$c_B^{\mathrm{p}}$}
  \psfrag{cA0}{$c_A \left( \mathbf{x}, t = 0\right) = 0$}
  \psfrag{cB0}{$c_B \left( \mathbf{x}, t = 0\right) = 0$}
  \psfrag{cC0}{$c_C \left( \mathbf{x}, t = 0\right) = 0$}
  \psfrag{nf}{$c_A^{\mathrm{p}} = c_B^{\mathrm{p}} = 0$}
  \subfigure[A pictorial description of the problem.]{
    \label{Fig:Bimolecular_Problem_with_advection_a}
  \includegraphics[scale = 0.8]{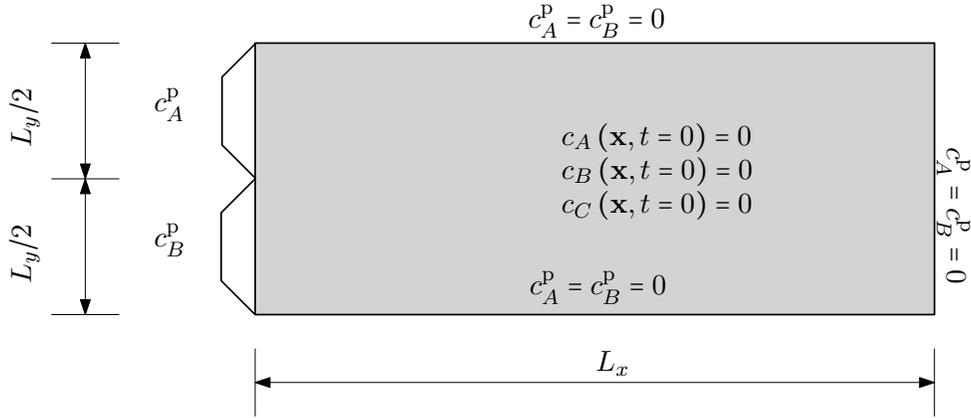}}
  \subfigure[Decomposition of the computational domain 
  into subdomains.]{
  \label{Fig:Bimolecular_Problem_with_advection_b}
  \includegraphics[scale = 0.75, clip]{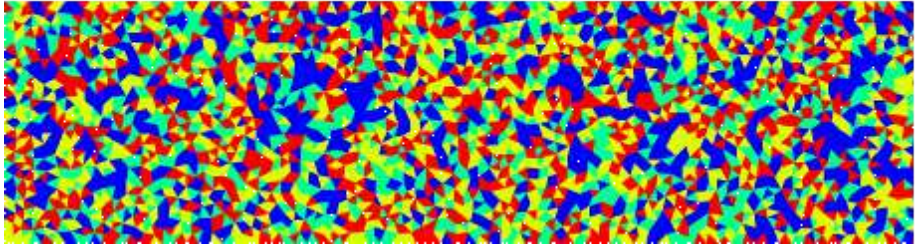}
  }
  \caption{Fast bimolecular reaction with advection:~  
    Chemical species $A$ and $B$ pumped into the reaction 
    chamber from the left side and produce the product $C$ 
    as a result of the chemical reaction. The computational 
    domain is divided into four subdomains, which are indicated 
    using different colors. Subdomains 1, 2, 3 and 4 are, 
    respectively, indicated in blue, green, yellow and red 
    colors. The computational domain is meshed using 4148 
    three-node triangular elements. The decomposition 
    of the computational domain is done using the METIS 
    software package \citep{METIS_paper}. Note that each subdomain 
    consists of many non-contiguous parts, and is highly 
    unstructured. (See the online version of the paper for 
    a color picture.) \label{Fig:Bimolecular_Problem_with_advection}}
\end{figure}

%----------------------------------------------------;
%  Figure: Bimolecular reaction: d-continuity set-2  ;
%----------------------------------------------------;
\begin{figure}
  \subfigure[Concentration of the product $C$ at $t = 0.5.$]{
    \includegraphics[scale=0.65]{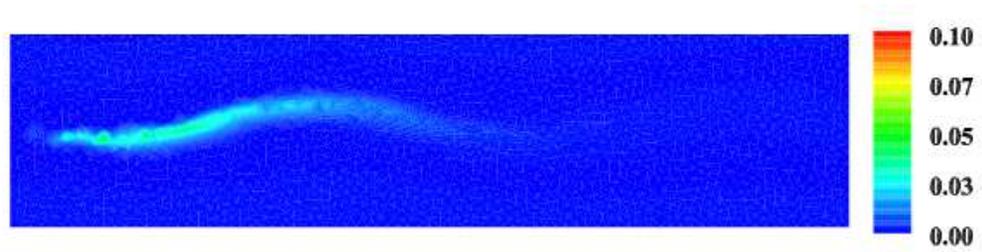}}
  \subfigure[Concentration of the product $C$ at $t = 1.5.$]{
    \includegraphics[scale=0.65]{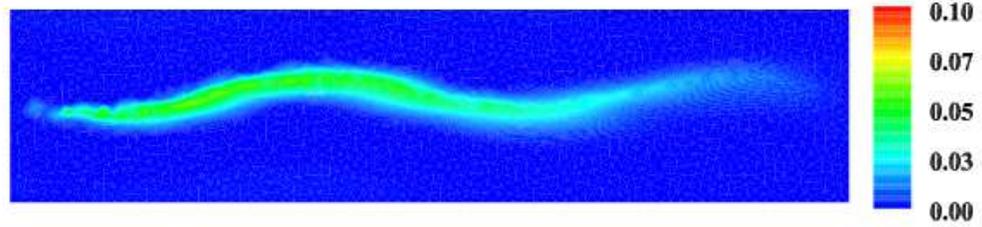}}
    \subfigure[Concentration of the product $C$ at $t = 4.0.$]{
    \includegraphics[scale=0.65]{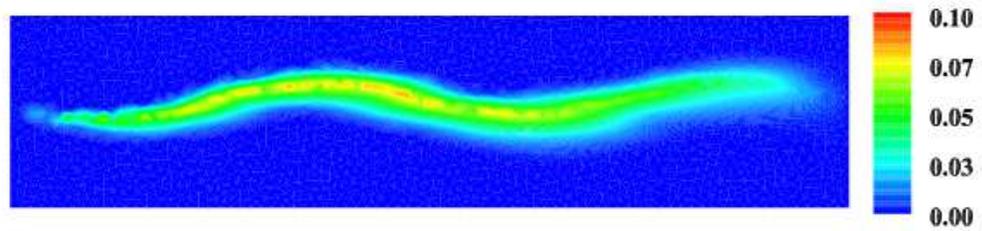}}
  \caption{Fast bimolecular reaction with advection:~This 
    figure shows the concentration of the product $C$ at 
    various instances of time obtained using the proposed 
    $d$-continuity multi-time-step coupling method. The 
    system time-step is taken to be $\Delta t = 0.1$, and 
    the subdomain time-steps are $\Delta t_1 = 0.01$, 
    $\Delta t_2 = 0.05$, $\Delta t_3 = 0.01$ and $\Delta 
    t_4 = 0.05$. Implicit Euler method is employed in 
    subdomains 1 and 3 (i.e., $\vartheta_1 = \vartheta_3 = 1$),
     and the midpoint rule is employed 
    in subdomains 2 and 4 (i.e., $\vartheta_2 = \vartheta_4 = 1/2$). 
    As one can see from the figure, 
    there is no drift along the subdomain interface, and 
    the proposed coupling method performed well. 
  \label{Fig:Bimolecular_d_Continuity}}
\end{figure}

\end{document}